\documentclass[bj,authoryear]{imsart}

\usepackage{
amsmath, 
amsfonts,
amsthm, 
amscd,
amssymb,
enumitem,
mathrsfs,
bbm}
\usepackage{url}
\usepackage{hyperref}
\usepackage{xcolor} 

\startlocaldefs
\theoremstyle{plain}

\newtheorem{theorem}{Theorem}[section]
\newtheorem*{theorem*}{Theorem}
\newtheorem{lemma}[theorem]{Lemma}
\newtheorem{condition}[theorem]{Condition}
\newtheorem{corollary}[theorem]{Corollary}
\newtheorem{proposition}[theorem]{Proposition}
\newtheorem*{proposition*}{Proposition}
\theoremstyle{definition}
\newtheorem{definition}[theorem]{Definition}
\newtheorem{remark}[theorem]{Remark}

\DeclareMathOperator{\Var}{Var}

\DeclareMathOperator{\Unif}{Unif}
\DeclareMathOperator{\supp}{supp}
\DeclareMathOperator{\Span}{span}

\DeclareMathOperator{\an}{an}
\DeclareMathOperator{\alt}{alt}
\DeclareMathOperator{\bc}{bc}
\DeclareMathOperator{\tr}{tr}
\DeclareMathOperator{\lip}{lip}
\DeclareMathOperator{\ess}{ess}
\DeclareMathOperator{\diag}{diag}
\DeclareMathOperator{\pen}{pen}

\newcommand{\E}{\mathbb E}
\newcommand{\Prob}{\mathbb P}
\newcommand{\R}{\mathbb R}
\newcommand{\N}{\mathbb N}
\newcommand{\Z}{\mathbb Z}
\newcommand{\iid}{\overset{\text{i.i.d.}}{\sim}}
\endlocaldefs

\begin{document}

\begin{frontmatter}
\title{Posterior Consistency under dependent errors in non-linear inverse problems with applications to time evolution equations}
\runtitle{Posterior Consistency under dependent errors}

\begin{aug}
\author[A]{\inits{A.}\fnms{Aditya}~\snm{Jambhale}\ead[label=e1]{aj644@cam.ac.uk}}
\address[A]{Department of Pure Mathematics and Mathematical Statistics,
University of Cambridge, Cambridge, United Kingdom\printead[presep={,\ }]{e1}}
\end{aug}

\begin{abstract}
    Let $\theta \mapsto \mathcal G(\theta)$ be a map from a parameter to the solution of some non-linear partial differential equation modelling a dynamical system. We provide conditions to show posterior consistency for $\theta$ under observations of $\mathcal G(\theta)$ with dependent noise, modeling potential time-correlations of the sensors used to make the measurements. We provide two examples, namely the Navier-Stokes equations with unknown viscosity and initial condition, as well as a Reaction-Diffusion type equation with diffusivity coefficient that is varying (heterogeneous) in space, which requires novel stability estimates for a Darcy transport operator with accumulation in time. 
\end{abstract}


\end{frontmatter}

\tableofcontents

\section{Introduction}\label{sec:introduction}

\textbf{Bayesian Inference in Non-linear Dynamical Systems.}
In a wide range of scientific applications, for example chemistry, atmospheric science, and geophysics, one takes measurements of a dynamical system $(u(t, \cdot): t\geq 0)$ evolving in $L^2(\Omega, V)$, the space of square-integrable functions on some $d$-dimensional domain $\Omega$ (either contained in $\R^d$ or the $d$-dimensional torus $\mathbb T^d$), and valued in some finite dimensional vector space $V = \R^{d_V}$ with the Euclidean inner product. However, the parameter(s) of interest $\theta$ (which determine the dynamics) are generally not directly observable from the solution. Denoting by $\Theta$ the set of possible parameters, we can formalize this by a \textit{non-linear} map 
\begin{equation}\label{eq:forwardmap}
\mathcal G: \Theta \to L^2([0, S]\times \Omega, V),
\end{equation}
where $S >0$ is some fixed time horizon (for notational convenience we assume $S = 1$).
Given some $\theta \in \Theta$, a popular observational model is
\begin{equation}\label{eq:introsetting}
    Y_i = \mathcal G(\theta)(t_i, X_i) + \epsilon_i, \quad \text{ for } i = 1, \dots N,
\end{equation}
where $(t_i, X_i)\iid \Unif([0, S]\times \Omega)$ are the design points and $\epsilon = (\epsilon_i)_{i = 1}^N$ are some observational errors. The Bayesian approach for recovering $\theta$ is to place a prior probability measure $\Pi$ (in general depending upon $N$) on $\Theta$, and then to update it with the observed data $\mathcal D_N \equiv (Y_i, t_i, X_i)_{i = 1}^N$ according to Bayes' rule yielding a posterior $\Pi(\cdot | \mathcal D_N)$. If we assume that the distribution of $\epsilon$ has a density with respect to the Lebesgue measure on $V^N$, then the law $\mathbb Q_\theta^N$ of $\mathcal D_N$ on $(V\times [0, S]\times \Omega)^N$ has a density $q_\theta^N$ with respect to the Lebesgue measure too. In this case, the posterior is the probability measure given by the formula
\begin{equation*}
    d\Pi(\theta |\mathcal D_N) \propto q_\theta^N(\mathcal D_N)d\Pi(\theta).
\end{equation*}

Because there are computationally feasible methods to sample from such a posterior, for example MCMC methods like pCN or Metropolis-Hastings, these methods have been widely adopted in applications following a seminal paper of A. Stuart \cite{S}, with theoretical foundations emerging since then in \cite{N20}, \cite{MNP}, \cite{MNP2}, in particular see \cite{N}.
More recent work has focused on problems specifically in the time-evolution setting, particularly for parabolic partial differential equations (PDEs), see \cite{NT}, \cite{KN}, \cite{KN2}, \cite{NPR}, \cite{CN}, \cite{N26+}, \cite{K}, or \cite{MvdMvdV}. 
We are interested in the asymptotic behavior of the posterior distribution, assuming the data is generated by some ground truth $\theta_0\in \Theta$. A first result towards this is \textit{posterior consistency}, namely that for some rate $\delta_N \to 0$, and constant $C$,
\begin{equation}\label{eq:posteriorconsistency}
    \Pi(\theta : d(\theta, \theta_0)\leq C \delta_N | \mathcal D_N) \to 1 \quad \text{ in }\mathbb Q_{\theta_0}^N \text{ probability as }N \to \infty,
\end{equation}
for some metric $d$ on $\Theta$, i.e. that the posterior will converge to the ground truth as more data is provided.

\noindent
\textbf{Two Inverse Problems.} To illustrate the theory, we are interested here in two parabolic non-linear PDEs posed in the $d$-dimensional torus $\Omega = \mathbb T^d$. A common feature for both inverse problems is that information about the parameters 'accumulate' in time, allowing one to weaken the assumptions required to show backwards stability for $\mathcal G$ (namely a local H\"older continuity estimate for the inverse map $\mathcal G^{-1}$, see Proposition \ref{thm:HRDEBack} and \eqref{eq:backward2}) when compared to the elliptic PDE setting. 

Firstly, we consider $u: [0, 1] \times \Omega \to \R$ (i.e. periodic in the last $d$ arguments) solving a reaction-diffusion type equation
\begin{equation}\label{eq:introRDE}
    \begin{split}
        \frac{\partial u}{\partial t} - \nabla \cdot (D\nabla u) &= f(u) \quad \text{ on } (0, 1]\times \Omega, \\
        u(0, \cdot) &= \varphi \quad \text{ on }\Omega,
    \end{split}
\end{equation}
where $\nabla$ and $\nabla\cdot$ are the gradient and divergence operators respectively, $f: \R\to \R$ is a time-independent, smooth and compactly supported `reaction' function, and we denote by $f(u)$ the function $f(u)(x) = f(u(x))$; moreover, $D: \Omega \to [D_{\min}, \infty)$ is called the 'diffusivity' (for some constant $D_{\min} >0$) and $\varphi: \Omega \to \R$ is the initial condition. We assume that $f$, $\varphi$ and $D_{\min}$ are known, while $D$ is unknown. When $d\leq 3$, \eqref{eq:introRDE} admits a unique solution (see Proposition \ref{prop:HRDEReg}), which we denote by $u_D$, from which we obtain a forward map 
\begin{equation}\label{eq:HRDEforwardmap}
    \mathcal G(D) = u_D.
\end{equation}
Physically, $u$ represents the amount of some substance as it moves around a medium and is created or destroyed by some reaction. It is created where $f(u)$ is positive and destroyed where $f(u)$ is negative. The term $-\nabla \cdot(D\nabla u)$ controls how much $u$ diffuses throughout the medium. When the medium is \textit{homogeneous}, i.e. $D$ is a constant, we have the usual Reaction Diffusion equation. Thus, we call \eqref{eq:introRDE} the \textit{heterogeneous} reaction diffusion equation, HRDE for short, and it has been considered as a more realistic alternative to the classical reaction diffusion equation in some biological processes \cite{RR} and in wave front analysis \cite{WW}.

To handle this mathematically, we prove a stability estimate for a Darcy type transport operator (see \cite{BCDVPW}, \cite{ABJN}, \cite{N24}, \cite{NvdGW} and \cite{NS}) in a ``parabolic'' setting. Assuming the \textit{critical points of the solution move in time}, i.e. that for some diffusivity $D_1$ there is a constant $\delta > 0$ so that 
\begin{equation}\label{eq:introcondition}
    \inf_{x \in \Omega} \int_0^1 |\nabla u_{D_1}(t, x)|^2\, dt \geq \delta,
\end{equation}
we obtain an estimate of the form, for all diffusivities $D_1$, $D_2$,
\begin{equation}\label{eq:backward1}
    \|D_1 - D_2\|_{L^2}\lesssim \|\nabla \cdot([D_1 - D_2]\nabla u_{D_1})\|_{L^2([0, 1]\times \Omega)}^\eta,
\end{equation}
for some $\eta$ depending on the smoothness of $D_1$ and $D_2$. See Lemma \ref{lem:transportstability} for details. The set of initial conditions $\varphi$ that make \eqref{eq:introcondition} hold for some $\delta$ will be shown to be generic in suitable topologies, see Proposition \ref{prop:opendense}. Moreover, provided the experimenter can choose the initial condition, there is a choice of $\varphi$ depending only on known quantities such that \eqref{eq:introcondition} holds, see Proposition \ref{prop:RDvarphi}.

Next, we are interested in solutions of the periodic 2D Navier-Stokes equations, namely $u: [0, 1]\times \Omega \to \R^2$ periodic in the last 2 arguments satisfying
\begin{equation}\label{eq:introNS}
    \begin{split}
        \frac{d u}{dt} + \nu A u + B(u, u)&= f \quad \text{ on } [0, 1] \times \Omega, \\
        \nabla \cdot u &= 0\quad \text{ on } [0, 1] \times \Omega, \\
        u(0, \cdot) &= \varphi \quad \text{ on } \Omega,
    \end{split}
\end{equation}
where $A = -P\Delta$, $B = P[(u\cdot\nabla)u]$, $\Omega = \mathbb T^2$, $\varphi: \Omega \to \R^2$ some initial vector field, $f: \Omega \to \R^2$ some time-independent forcing function, $\nu \geq \nu_{\min}$ for some constant $\nu_{\min} >0$, and for $i = 1, 2$,
\[((u\cdot\nabla)u)_i = \sum_{j = 1}^2 u_j \frac{\partial u_i}{\partial x_j},\]
$\Delta$ is the Laplacian, and $P$ is the projection onto the Hilbert space $H\subseteq L^2$ of divergence free, periodic, mean-zero vector fields, to be defined later.
Here, we assume $f$ and $\nu_{\min}$ are known, while $\nu$ and $\varphi$ are unknown. Solutions to \eqref{eq:introNS} exist and are unique (See \textsection 9 in \cite{R}), which we will denote by $u_{\nu ,\varphi}$. As such, we obtain a forward map
\begin{equation}\label{eq:NSforwardmap}
    \mathcal G(\nu, \varphi) = u_{\nu, \varphi}.
\end{equation}
Proving posterior consistency requires a stability estimate of the form
\begin{equation}\label{eq:backward2}
|\nu_1 - \nu_2| + \|\varphi_1 - \varphi_2\|_{L^2(\Omega)} \lesssim \|\mathcal G(\nu_1, \varphi_1) - \mathcal G(\nu_2, \varphi_2)\|_{L^2([0, 1]\times \Omega)}^\eta,
\end{equation}
where $\eta$ depends on the smoothness of $\varphi_1$ and $\varphi_2$, see Proposition \ref{prop:NSBack}. 

\noindent
\textbf{The Dependence Structure.} On the statistical aspects of \eqref{eq:introsetting}, most work has been done assuming the errors are i.i.d., while little has been done for other observational models which may match with applications more closely, see \cite{NS} and \cite{SS2}. In analyzing time-evolution equations such as \eqref{eq:introRDE} and \eqref{eq:introNS}, a natural set-up is that if the same sensing equipment is used at different times, then the observational errors for a given sensor may be time-correlated. This is the structural assumption we place on the $\epsilon_i$. Specifically, we consider a ``space-independent'' (or ``time-dependent'') scheme. Let $\mathcal G$ and $V = \R^{d_V}$ be as in \eqref{eq:forwardmap}, and
\begin{equation}\label{eq:timedependent}
    Y_{ij} = \mathcal G(\theta)(t_i, X_{ij}) + \epsilon_{ij}, \quad \text{for }i = 1, \dots T, \quad j = 1, \dots n,
\end{equation}
where $t_i = (2i-1)/2T$ are deterministic, $X_{ij}\iid \Unif(\Omega)$, $\epsilon_{ij}$ are jointly Gaussian random variables independent of the $X_{ij}$, the total number of observations is $N = Tn$, and 
\begin{alignat*}{2}
    \epsilon_{ij}&\perp \epsilon_{i'j'} && \text{ for all } j \neq j', \text{ and }i, i' = 1, \dots T,\\
    \E[\epsilon_{ij}\epsilon_{i'j}^T] &= M_{ii'}^j &&\text{ for all } i \neq i' \text{ and }j = 1, \dots n, \text{ and } \\
    \E[(\epsilon_{ij}^k)^2] &= (\sigma_j^k)^2 &&\text{ for all }k = 1, \dots d_V,
\end{alignat*}
where $M^j_{ii'}$ is a $d_V \times d_V$ matrix, $\epsilon_{ij} = (\epsilon^1_{ij}, \dots \epsilon_{ij}^{d_V})\in \R^{d_V}$, and $\sigma_j = (\sigma_j^k)\in \R^{d_V}$ has all positive entries. 
Intuitively, we assume that different sensors have independent errors, but for a single sensor at various points in time its observations may be dependent. We sometimes re-index \eqref{eq:timedependent} as
\begin{align}
   Y_\ell = \mathcal G(\theta)(t_\ell, X_\ell) &+ \epsilon_\ell \quad \text{ for } \ell = 1, \dots N, \nonumber \\
   \epsilon = (\epsilon_1, \dots, \epsilon_N) &\sim N(0, \Gamma_N),\label{eq:setting}
\end{align}
for some $d_VN\times d_VN$ invertiblecovariance matrix $\Gamma_N$, $t_i \equiv t_i^N \in [0, 1]$ are deterministic, and $X_i \iid \Unif(\Omega)$. We denote by $\Prob_\theta^N$ the law of the data $D_N = (Y_i, X_i)^N_{i = 1}$ on $(V\times \Omega)^N$, and by $p_\theta^N$ the density of this law with respect to the Lebesgue measure. 

\begin{remark}
    While on the surface a mixed design model might seem artificial, we have done so to ease computation in the proofs while maintaining a sensible dependence structure. Indeed, because we wish to model time-correlations, we want the correlation of $\epsilon_\ell$ and $\epsilon_k$ to depend on $|t_\ell - t_k|$. In order to avoid a random covariance matrix then, we fix the $t_\ell$'s in the design. However, to avoid accruing unnecessary approximation errors in the proofs, we assume a random design on the spatial variables.
    
    Schemes other than \eqref{eq:timedependent} can be considered as long as the empirical norm it induces concentrates suitably well around the $L^2([0, 1]\times \Omega)$ norm, see Remark \ref{rmk:whyscheme}.
\end{remark}

Our main statistical contribution is to provide general conditions on the forward map $\mathcal G$ and the covariance matrix $\Gamma_N$ which imply the posterior is consistent in the sense of \eqref{eq:posteriorconsistency}. We then give concrete examples where these conditions can be verified, namely with the forward map as described in \eqref{eq:HRDEforwardmap} and \eqref{eq:NSforwardmap}, though we postpone the statement of these results until we have established more notation.

\begin{remark}
    In the course of proving posterior consistency, we get a worse rate than in the iid case. However, once an estimator that converges at a sufficiently fast rate can be computed, the posterior can be shown to converge at the same rate as in the iid case. 
    In particular, the posterior contraction we show implies the existence of such estimators (e.g. the posterior mean). 

    To prove the same rate as in the i.i.d. case, one linearizes the forward map at the estimator, and constructs a penalized least-squares estimator on $\theta\in \Theta$. Because we linearize the forward map, the loss function is convex. One can then construct plug-in tests at the rate in the i.i.d. case, and hence show posterior contraction at that rate as well.
\end{remark}

In the course of proving these results, we prove the existence of a polynomial time initializer in the setting of \eqref{eq:introNS} with forward map as in \eqref{eq:NSforwardmap}.
In this setting, the floating-point operations which count towards computational cost are those that are involved with matrix--vector products and linear system solves. Note that this differs from the notion of computational cost used in MCMC literature; see Remark 2.3 in \cite{W} for details. 

\begin{proposition*}
    Consider the setting of \eqref{eq:introNS} with ground truth $(\nu_0, \varphi_0)$, with $\varphi_0 \in H^\alpha(\Omega)$ for $\alpha > d/2 +1$. Moreover, assume that $\nu, \|\varphi_0\|_{H^\alpha} \leq R$ for some $R>0$. There exists constants $c, C_1, C_2$ depending on $\Omega$, $R$, $\nu_{\min}$ and an exponent $\kappa = \kappa(d, \alpha) < 2$ such that for all $N \geq 1$, there is an estimator $\hat \varphi_N$ satisfying the convergence rate, for $\beta = \frac{2(\alpha +1) - d - 2}{2(\alpha +1) + d +2}$,
    \begin{equation*}
        \Prob_{\nu_0, \varphi_0}^N \big(\|\varphi_0 - \hat \varphi_N\|_\infty \leq C_1 N^{-\beta/2}\big) \gtrsim 1 - \exp\big(-cN^{\frac{d+2}{2(\alpha+1) + d+2}}\big)
    \end{equation*}
    while $\hat \varphi_N$ is computable within $C_2N^\kappa$ floating-point operations. 
\end{proposition*}
The estimator in question is $\hat \varphi_N = \hat u_N(0, \cdot)$ for $\hat u_N$ as in \eqref{eq:estimator}, and we note this proposition holds in the i.i.d. errors case as well. The proof follows mutatis mutandis from the elliptic setting of \cite{W} to the parabolic setting, see (21) and Theorem 2.4 therein. We record this result because one obtains a better rate by choosing a penalization more suited to the parabolic PDE setting than the na\"ive generalization of \cite{W}. This verifies the condition needed in \cite{NW24} to show polynomial time guarantees for Bayesian posterior computation, and we note that without such an initializer, the computation of the posterior via MCMC methods may require exponential run-times, see \cite{BMNW}. 

\subsection{Notation}\label{sec:notation}

We assume throughout that $\Omega$ is the $d$-dimensional torus, i.e. $\mathbb T^d$ where $\mathbb T = \R/\Z$. In keeping with the PDE literature, this is for notational simplicity and all the results can be adapted for $\Omega\subseteq \R^d$ a bounded domain with appropriate changes to the boundary conditions. We assume that $\Theta$ is a linear space satisfying
\begin{equation}\label{eq:Thetadef}
    \Theta \subseteq L^2(\mathcal Z, W),
\end{equation}
where $\mathcal Z$ is an $\ell$-dimensional domain (we adopt the convention that an $\ell$-dimensional domain is a bounded open subset of $\R^\ell$ with smooth boundary or $\mathbb T^\ell$), and W is a finite dimensional vector space. 

We will frequently denote the time-cylinder $[0, 1]\times \Omega = Q$. For $N\in \N$, write $[N]$ for the set $\{1, \dots, N\}$. For a vector $v\in \R^k$ for some $k$, we will denote its coordinates by $v^i$ for $i = 1, \dots k$, so that $v = (v^1, \dots, v^k)$.

For $X$ any normed space, we denote by $\|\cdot\|_X$ the norm on $X$, and by $X^\ast$ the topological dual of $X$. Also, for any $r\geq 0$, $B_X(r)$ denotes the closed ball of radius $r$ centered at the origin in $X$. For any $\mathcal O$ a domain in $\R^d$, we let $L^p(\mathcal O)$ denote the space of $p$-power integrable real-valued functions on $\mathcal O$. In the particular case of $L^p$-norms, we will write $\|\cdot\|_p$ when the domain is clear from context. We also define the spaces $C^k(\mathcal O)$ and $H^k(\mathcal O)$ to have their usual meanings as function spaces consisting of $k$-times differentiable (resp. weakly differentiable) functions, with derivatives defining uniformly continuous functions (resp. $L^2(\mathcal O)$ functions). The definition of $H^k(\mathcal O)$ extends to real valued $k$, see \cite{GT}. $C^\infty_c(\mathcal O)$ denotes the space of infinitely differentiable functions which are compactly support within $\mathcal O$, and $H^k_c(\mathcal O)$ the closure of $C^\infty_c(\mathcal O)$ with respect to the $H^k(\mathcal O)$-norm.

For some finite dimensional vector space $W$, we denote $C^k(\mathcal O, W)$ and $H^k(\mathcal O, W)$ to be the spaces of functions from $\mathcal O$ to $W$ so that each coordinate (w.r.t to any basis), lies in $C^k(\mathcal O)$ or $H^k(\mathcal O)$ respectively. When $W = \R^m$, we often write $C^k(\mathcal O)^m$ or $H^k(\mathcal O)^m$ for this space as well. To any Banach space $X \subseteq L^1$, we define $\dot X = \{ f \in X : \int f = 0\}$.

We also need to define certain functions spaces involving time. For a separable Banach space $X$ and $p \in [1, \infty]$ define $L^p([0, S], X)$ as the space of measurable functions $f:[0, S] \to X$ such that
\begin{equation*}
    \|f\|_{L^p([0, S], X)}^p = \int_{0}^S \|f(t)\|_X^p \, dt < \infty,
\end{equation*}
see \textsection 5.9.2 in \cite{Evans} for details. Moreover, for $m\geq 0$, we define the Sobolev spaces in time by
\begin{equation*}
    H^m([0, S], X) = \bigg\{ f \in L^2([0, S], X) : \partial_t^i f \in L^2([0, S], X), \text{ for }0\leq i \leq m\bigg\},
\end{equation*}
equipped with the norm $\|f\|^2_{H^m([0, S], X)} = \sum_{j = 0}^m \|\partial_t^j f\|^2_{L^2([0, S], X)}$. In the special case where $X = H^k(\mathcal O)$, we often abbreviate 
\[H^m([0, 1], H^k(\mathcal O)) = H^m_tH^k_x.\]
Similarly, for the spaces $L^p([0, 1], L^q(\mathcal O))$ we write $L^p_tL^q_x$. Lastly, we define spaces that deal with the smoothness yielded by solutions to parabolic equations. See e.g. \textsection 2.1 in \cite{LM2}. For $s\in [0, \infty)$, we define the \textit{parabolic} (or \textit{anisotropic}) Sobolev space $H^{s, s/2}(Q)$ of functions $f \in L^2(Q)$ such that the norm 
\begin{equation*}
    \|f\|_{H^{s, s/2}(Q)}^2 = \int_{0}^1 \|f(t, \cdot)\|^2_{H^{s}(\Omega)} \, dt+ \int_{\Omega} \|f(\cdot, x)\|^2_{H^{s/2}([0, 1])} \, dx
\end{equation*}
is finite. We will not mention $Q$ in the notation as it is the only domain these spaces are used on.

For some index set $A$ and $p \in [1 ,\infty]$, we define $\ell^p(A)$ to be the set of sequences indexed by $A$ such that
\begin{equation*}
    \|(a_i)_{i\in A}\|^p_{\ell^p(A)} = \sum_{i\in A}|a_i|^p < \infty.
\end{equation*}
When $A$ is finite and $n = |A|$, we often write this space as $\ell^p(\R^n)$, and we will omit $A$ when no confusion may arise. 
Next, for the space of $n \times n$ matrices, denoted by $\mathcal L(\R^n, \R^n)$, we define the norm $\|\cdot \|_{p\to p}$ as the operator norm on $\R^n$ equipped with the $\ell^p$ metric. For the case of $p = 2$, this is the usual operator norm for matrices, and we denote this norm simply by $\|\cdot\|_{2\to 2} = \|\cdot\|$. Throughout, we use $C$ and $c$ to denote generic constants that do not depend on the variables of the problem, and unless otherwise specified they can change from line to line. 

Where we deal with the Navier-Stokes equations in $2$D, we follow the notation in \textsection 3 of \cite{KN}. In particular, we define the Sobolev space of mean-zero divergence free vector fields
\begin{equation*}
    \dot H^\alpha_\diamond(\Omega) = \bigg\{ u \in H^\alpha(\Omega)^2 : \nabla \cdot u = 0, \quad \int_\Omega u = 0\bigg\}, \quad H = \dot H^{0}_\diamond.
\end{equation*}

Finally, let $p, q$ be two probability densities such that if $q = 0$, then $p = 0$. Recall that the Kullback-Leibler divergence and variation for $p, q$ are given by
\begin{align}\label{eq:KLFormula}
\begin{split}
    KL(p, q) &= \E_p[\log\frac{p}{q}],\text{ and } \\
    V_{2, 0}(p, q) &= \E_p[(\log \frac{p}{q} - KL(p, q))^2].
\end{split}
\end{align}
To specify the dependence structure, we need the notion of a vector autoregressive (VAR) process. For $T, k\in \N$, let $\epsilon_1, \dots, \epsilon_T$ be random variables valued in $\R^k$. Then, we say that $(\epsilon_i)_{i = 1}^T$ is a $VAR(p)$ process with parameters $A_1, \dots, A_p, \Sigma$ if 
\begin{equation}\label{eq:VAR}
    \epsilon_t = A_1\epsilon_{t-1} + A_2 \epsilon_{t-2} + \cdots + A_p \epsilon_{t-p} + \eta_t \quad \text{ for } \quad t\geq p,
\end{equation}
where $A_1, \dots, A_p$ are $k\times k$ matrices, and $\eta_t \iid N(0, \Sigma)$ for some $k\times k$ covariance matrix $\Sigma$. 
\subsection{Main results}
To start, we list the assumptions on the covariance matrix from \eqref{eq:timedependent} that are sufficient to yield posterior contraction. We leverage these conditions by bounding the Kullback-Leibler divergence and variation, in order to satisfy the small ball condition \eqref{eq:smallball} of Theorem \ref{thm:PosteriorGeneral}. Firstly, we need that the ``direction of worst variance'' is bounded:

\begin{condition}\label{con:variancebound}
    Let $\Gamma_N$ be the covariance from \eqref{eq:setting}. Suppose that there is some constant $C_1$ independent of $N$ such that
    \[
        \|\Gamma_N\| \leq C_1,
    \]
    where $\|\cdot\|$ denotes the matrix operator norm.
\end{condition}
Next, we need a stronger assumption on the precision matrix $\Gamma_N^{-1}$:

\begin{condition}\label{con:inversematrix}
    For the setup in \eqref{eq:setting} with forward map valued in $V = \R^{d_V}$ as in \eqref{eq:forwardmap}, we assume that 
    \[
        \max_{i = 1, \dots, Nd_V}\sum_{j=1}^{Nd_V} |(\Gamma_N^{-1})_{ij}| \leq C_2,
    \]
    for some constant independent of $N$. 
\end{condition}
Note that the norm on square matrices given by $\max_i \sum_j |A_{ij}|$ is the same as $\|A\|_{\infty \to \infty}$, see \textsection 5.6.5 in \cite{HJ}.

If the dynamical system of interest has the same average value amongst all solutions, then we can require a weaker assumption on $\Gamma_N$. For example, in the Navier-Stokes equations Condition \ref{con:inversematrix2}.a is satisfied as $\int u(t, x) \, dx = 0$ for all $t$ and initial conditions $\varphi$. 
Conceptually, this is because if the $\epsilon_i$ are highly correlated, the observations are essentially a constant shift of the solution, and if the average value of the solution is known, this can be corrected for. This phenomenon of dependent data has been observed before, see \cite{Y}.

\begin{condition}\label{con:inversematrix2}
    For the setup in \eqref{eq:setting}, with forward map $\mathcal G$ and vector space $V = \R^{d_V}$ as in \eqref{eq:forwardmap}, suppose the following:
    \begin{enumerate}[label = (\alph*)]
        \item There is a known function $A: [0, 1]\to V$ independent of all $\theta\in\Theta$ so that
            \[
                \int_{\Omega} \mathcal G(\theta)(t, x)\, dx = A(t).
            \]
        \item For a constant $C_2$ independent of $N$, $\tr(\Gamma_N^{-2}) \leq C_2N$.
    \end{enumerate}
\end{condition}
By Remark \ref{rmk:inversematrixconditions}, Condition \ref{con:inversematrix2}.b is indeed a weaker assumption than Condition \ref{con:inversematrix}. Moreover,  
Condition \ref{con:inversematrix} can be satisfied for many assumptions on the covariance structure, for example vector-valued autoregressive processes, i.e. $\text{VAR}(p)$ processes. 

\begin{proposition}\label{prop:VARCheck}
    Let $(\epsilon_{ij})$ be as in \eqref{eq:timedependent}, and suppose that for each $j$, $(\epsilon_{ij})_{i = 1}^T$ is a stationary Gaussian $\text{VAR}(p_j)$ process ($p_j\in \N$) with parameters $A_1, \dots A_{p_j}$ and $\Sigma_j$ as in \eqref{eq:VAR}. If we further suppose that $p_j$, and $A_1, \dots, A_{p_j}$ are uniformly bounded above entry-wise independent of $j$, and the minimum eigenvalue of $\Sigma_j$ is bounded away from $0$ independently of $j$, then Condition \ref{con:inversematrix} holds.
\end{proposition}
See \textsection \ref{sec:dependenceCheck} for the proof of this proposition as well as another way to verify Condition \ref{con:inversematrix} which covers some time series with long-range dependence.
With these conditions, we can bound the KL-divergence and variation:

\begin{proposition}\label{prop:inversematrix}
    Suppose $\theta, \vartheta \in \Theta$, forward map $\mathcal G$ is as in \eqref{eq:forwardmap}, and $U >0$ are so that $\|\mathcal G(\theta)\|_\infty$, $\|\mathcal G(\vartheta)\|_\infty \leq U$. Let $p_\theta^N$ and $p_\vartheta^N$ the densities as introduced after \eqref{eq:setting}. Further, suppose that Condition \ref{con:inversematrix} \textbf{or} \ref{con:inversematrix2} holds for some constant $C_2$. Then, for some constant $C = C(C_2, U)$, we have that the KL-divergence and variation from \eqref{eq:KLFormula} satisfy
   \begin{align*}
       KL(p_\theta^N, p_\vartheta^N) &\leq CN \| \mathcal G(\theta) - \mathcal G(\vartheta)\|_{L^\infty_tL^2_x}^2, \text{ and } \\
       V_{2, 0}(p_\theta^N, p_\vartheta^N) &\leq CN\| \mathcal G(\theta) - \mathcal G(\vartheta)\|_{L^\infty_tL^2_x}^2.
   \end{align*}
\end{proposition}

Next, we discuss the class of priors for which we show posterior consistency. We work with Gaussian process priors, which live in linear spaces. If the parameter space $\Theta$ is not already a linear (for example the diffusivity in \eqref{eq:introRDE} or viscosity in \eqref{eq:introNS}), then the use of a \textit{link function} is required, see the discussion before and proof of Theorem \ref{thm:HRDEMain}. Following \cite{N}, we introduce a condition on the RKHS of the prior:

\begin{condition}\label{con:RKHS}
    Let $\Pi'$ be a centered Gaussian Borel probability measure with RKHS $\mathcal H$ on the linear space $\Theta$ and $\ell$-dimensional domain $\mathcal Z$ from \eqref{eq:Thetadef}.  Suppose further that $\Pi'(\mathcal R) = 1$ for some separable normed linear subspace $(\mathcal R, \|\cdot\|_{\mathcal R})$. Suppose further that $\mathcal H$ satisfies the continuous embedding
    \[ \mathcal H \subseteq H^{\alpha}_c(\mathcal Z) \text{ if }\kappa \geq 1/2, \quad \text{ or }\quad \mathcal H \subseteq H^\alpha(\mathcal Z) \text{ if }\kappa < 1/2,\]
    for some regularity parameter $\alpha > 0$. 
\end{condition}

We define the prior $\Pi = \Pi_N$ as the law $\mathcal L(\theta)$ after the rescaling
\begin{equation}\label{eq:priorrescaling}
    \theta = N^{-\ell/(4(\alpha+\kappa) + 2\ell)}\theta', \quad \theta'\sim\Pi'.
\end{equation}
We also note that given the data the $D_N$, the \textit{posterior} is given by
\begin{equation}\label{eq:posteriordef}
    d\Pi(\theta | D_N) \propto  p_\theta^N(D_N) d\Pi(\theta),
\end{equation}
where $p_\theta^N$ is defined after \eqref{eq:setting}.

Posterior contraction rates are proven using the existence of frequentist tests that work at an exponential rate. Commonly used is Hellinger testing theory as in \cite{GvdV} and \cite{N}; however, in the case of dependent errors, this theory breaks down. So, we construct tests via concentration inequalities for a wavelet estimator relying on the smoothness of $\mathcal G(\theta)$, using an idea from \cite{GN11}. 

\begin{condition}[Forward Regularity]\label{con:forwardregularity}
    Let $\mathcal G$, $V$ be from \eqref{eq:forwardmap} and $\mathcal Z$, $W$ be from \eqref{eq:Thetadef}. Suppose $\mathcal G$ is measurable, and let $\kappa \geq 0$. Suppose that for some normed linear subspace $(\mathcal R, \|\cdot\|_{\mathcal R})$ of $L^2(\mathcal Z, W)$ and all $M>0$, there exists a finite constant $U \geq 1$, and $L>0$ (that may depend on $M$) such that
    \[
        \sup_{\theta \in \Theta\cap B_{\mathcal R}(M)} \sup_{(t, x)\in Q} |\mathcal G(\theta)(t, x)|_V \leq U
    \]
    and
    \[
        \|\mathcal G (\theta_1) - \mathcal G(\theta_2)\|_{L^\infty_t L^2_x}\leq L\|\theta_1 - \theta_2\|_{(H^\kappa(\mathcal Z, W))^\ast}\text{ for all }\theta_1, \theta_2 \in \Theta\cap B_{\mathcal R}(M).
    \]
    We call $\mathcal R$ the regularization space.
\end{condition}

\begin{condition}[Forward Mapping]\label{con:forwardmapping}
    Let $\mathcal R$, $\mathcal H$ be as defined in Condition \ref{con:forwardregularity} and \ref{con:RKHS}. 
    Suppose that there is an $a \geq 0$ such that
    \[
        \mathcal G(\mathcal H)\subseteq H^{a, a/2}.
    \]
    Furthermore, suppose that for any $M$ there exists a $U$ depending on $M$ such that
    \begin{align*}
        \mathcal G(B_{\mathcal R}(M)) &\subseteq B_{H^{6, 3}}(U), and\\
        \mathcal G(B_{\mathcal H}(M)) &\subseteq B_{H^{a, a/2}}(U).
    \end{align*}
\end{condition}

\begin{remark}\label{rmk:besovIsSobolev}
    We refer the reader to \textsection \ref{sec:notation} for the definition of these parabolic Sobolev spaces. The choice of an anisotropic space reflects the best regularity one can prove for solutions to parabolic PDEs. When dealing with other equations, such as hyperbolic PDEs, then one would change this condition to involve isotropic spaces. A larger choice of $a$ in this condition will ultimately yield a faster contraction rate.
\end{remark}

\begin{theorem}[Existence of tests]\label{thm:TestsExist}
    Let $\mathcal G$ be as in \eqref{eq:forwardmap}, and suppose that Condition \ref{con:RKHS} holds for some regularization space $\mathcal R$, RKHS $\mathcal H$, $\ell$-dimensional domain $\mathcal Z$ and regularity parameter $\alpha$. Suppose that Conditions \ref{con:forwardregularity} and \ref{con:forwardmapping} hold for these choices, as well as for some regularity parameter $a \geq (d+1)\lor 6$.
    
    Then, set
    \begin{equation}\label{eq:rates}
        \delta_N = N^{-\frac{a}{2a + d +2}} \quad \text{ and }\quad \varepsilon_N = N^{-\frac{\alpha+\kappa}{2\alpha + 2\kappa + \ell}}, 
    \end{equation}
    and assume that $\varepsilon_N \lesssim \delta_N$. Also, with $n$, $T$ and $N = nT$ as in the setting of \eqref{eq:timedependent} and \eqref{eq:setting}, set 
    \begin{equation}\label{eq:rates2}
        N^{\frac{2d+3}{2a + d + 2}} \lesssim n \lesssim N\delta_N  \quad \text{ and }\quad \delta_N^{-1}\lesssim T \lesssim N^{\frac{2a - d - 1}{2a + d +2}},
    \end{equation}
    and for some $M>0$ let $\Theta_N$ be the regularization set
    \[
         \Theta_N = \{\theta\in \mathcal R : \theta = \theta_1 + \theta_2, \|\theta_1\|_{(H^\kappa(\mathcal Z))^\ast}\leq M\varepsilon_N, \|\theta_2\|_{\mathcal H}\leq M, \|\theta\|_{\mathcal R}\leq M\},
    \]
    and suppose that $\theta_0 \in \Theta_N$. 
    If Condition \ref{con:variancebound} holds, then there are tests (indicator functions) $\phi_N$ such that for any $K\in \N$ and $M>0$, there is some $C = C(M, K)$ sufficiently large so that
    \[
        \E^N_{\theta_0}[\phi_N] \to 0, \text{ and } \sup_{\substack{\theta \in \Theta_N \\ C\delta_N < \|\mathcal G(\theta)  -\mathcal G (\theta_0)\|_{L^2}}} \E_{\theta}^N[1 - \phi_N] \leq e^{-K N\varepsilon_N^2}
    \]
    provided $N$ is large enough.
\end{theorem}
Note that we have constructed tests in the metric $\|\mathcal G(\theta) - \mathcal G(\vartheta)\|_{L^2}$, though the mixed design of \eqref{eq:timedependent} naturally yields tests in the 
\begin{equation*}
    \frac{1}{T}\sum_{i = 1}^T \|\mathcal G(\theta)(t_i , \cdot) - \mathcal G(\vartheta)(t_i, \cdot)\|_{L^2(\Omega)}^2
\end{equation*}
(psuedo-)metric. We circumvent this difficulty by showing the approximation error of our design is small, see Proposition \ref{prop:Mconcentration}. In order to get contraction directly on $\Theta$, we impose a condition following \cite{N}. This amounts to requiring the ``inverse'' of $\mathcal G$ satisfies a H\"older continuity condition on bounded sets. 

\begin{condition}[Backwards Stability]\label{con:stabilityestimate}
    Let $\mathcal G$ be as in \eqref{eq:forwardmap}, and $\mathcal R$ be as in Condition \ref{con:RKHS} and \ref{con:forwardregularity}, with ground truth $\theta_0 \in \mathcal R \cap \Theta$. For some $\eta >0$ and all $M>0$, suppose there exists a constant $L' = L'(\mathcal G, M)$ such that for all $\delta >0$ small enough, 
    \[
        \sup\{ \| \theta - \theta_0\|_{L^2} : \theta \in \Theta\cap B_{\mathcal R}(M), \|\mathcal G(\theta) - \mathcal G(\theta_0)\|_{L^2(Q)} \leq \delta\} \leq L' \delta^\eta.
    \]
\end{condition}

This leads us to the main contribution, which is posterior contraction in non-linear inverse problems on regression with dependent errors. 

\begin{theorem}\label{thm:Main}
    Let $\mathcal G$ be as in \eqref{eq:forwardmap}, and suppose that Condition \ref{con:RKHS} holds for some prior $\Pi'$, regularization space $\mathcal R$, RKHS $\mathcal H$, $\ell$-dimensional domain $\mathcal Z$ and regularity parameter $\alpha$. Suppose that Conditions \ref{con:forwardregularity} and \ref{con:forwardmapping} hold for these choices with some regularity parameter $a \geq (d+1)\lor 6$ and that the ground truth $\theta_0\in \mathcal R \cap \mathcal H$. In the setting of \eqref{eq:timedependent} and \eqref{eq:setting} with observations $D_N$, let $\delta_N$, $\varepsilon_N$, $n$ and $T$ be as in \eqref{eq:rates} and \eqref{eq:rates2} respectively, with $\varepsilon_N \lesssim \delta_N$.
    
    Let $\Pi_N$ be the rescaled prior as in \eqref{eq:priorrescaling}. Moreover, assume that Condition \ref{con:variancebound} holds, and one of Condition \ref{con:inversematrix} or \ref{con:inversematrix2} holds. If $\Pi_N(\cdot|D_N)$ is the posterior distribution, then for all $b >0$, we can choose $m$ large enough such that as $N\to \infty$, 

    \[
        \Prob^N_{\theta_0}\big(\Pi_N(\theta \in \Theta, \|\mathcal G(\theta) - \mathcal G(\theta_0)\|_{L^2(Q)} \leq m \delta_N|D_N)\leq 1 - e^{-bN\varepsilon_N^2}\big) \to 0.
    \]
    Moreover, assuming Condition \ref{con:stabilityestimate} for the current choice of $\mathcal R$ and some $\eta , L' >0$, we deduce that
    \[
        \Prob^N_{\theta_0}\big(\Pi_N(\theta \in \Theta, \|\theta - \theta_0\|_{L^2} \leq L'(m\delta_N)^\eta|D_N)\leq 1- e^{-bN\varepsilon^2_N}\big) \to 0.
    \]
\end{theorem}
The proof is a combination of Theorem \ref{thm:PosteriorGeneral}, Theorem \ref{thm:smallball} and Theorem \ref{thm:TestsExist}. Finally, we show this theorem can be applied in PDE examples of interest, such as the ones arising from \eqref{eq:introRDE} and \eqref{eq:introNS}. See Theorems \ref{thm:HRDEMain} and \ref{thm:NSMain} for details.

\subsection{Summary of the Contributions}

We prove posterior contraction rates for non-linear inverse problems in a dependent data setting where Hellinger testing does not straightforwardly apply, see Theorems \ref{thm:Main}, \ref{thm:HRDEMain}, and \ref{thm:NSMain}. To do so, we construct ``plug-in'' tests via a wavelet-based estimator, see Theorem \ref{thm:TestsExist} and \textsection \ref{sec:theory}; this in turn requires wavelets which characterize the smoothness observed in solutions to parabolic PDEs, see \textsection \ref{sec:AnisotropicWavelets} and \textsection \ref{sec:besovIsSobolev}.

The same wavelet-based estimator constructed in \textsection \ref{sec:Construction} provides an initializer (or ``warm start'') for MCMC techniques in initial value problems for parabolic PDEs such as 2D Navier-Stokes, see \cite{NT}. In particular, such an initializer guarantees polynomial time computation of the posterior, which in the case of a ``cold start'' may take exponential time, see \cite{BMNW}. 

We also contribute a stability estimate for Darcy's operator in a time-evolution setting, and prove that the time-evolution structure can be exploited to prove stronger results compared to the elliptic Darcy problem, see \textsection \ref{sec:HRDE} and \textsection \ref{sec:findingphi}. We also prove a joint stability on viscosity and initial condition in the case of the 2D Navier-Stokes equations, see \textsection \ref{sec:NS}.

\section{Contraction Rate Result}\label{sec:theory}
Proofs for this section can be found in \textsection \ref{sec:ProofsforTheory}.

\subsection{Verifying Conditions \ref{con:inversematrix} and \ref{con:inversematrix2}}\label{sec:dependenceCheck}
We will show Condition \ref{con:inversematrix} holds for a wide range of covariances in the setting of \eqref{eq:timedependent}.
\begin{remark}\label{rmk:inversematrixconditions}
Note that Condition \ref{con:inversematrix2}.b is a weaker condition than Condition \ref{con:inversematrix}. Indeed, if Condition \ref{con:inversematrix} holds, then
\begin{align*}
    \tr(\Gamma_N^{-2}) &= \sum_{i, j = 1}^{Nd_V} [(\Gamma_N^{-1})_{ij}]^2 \leq Nd_V \max_{i} \sum_j [(\Gamma_N^{-1})_{ij}]^2 \leq Nd_V \max_{i} \bigg(\sum_j |(\Gamma_N^{-1})_{ij}|\bigg)^2 \leq Nd_VC_2^2,
\end{align*}
so Condition \ref{con:inversematrix2}.b holds. 
\end{remark}
\begin{remark}\label{rmk:simplifyingConds}
In the setting of \eqref{eq:timedependent}, with $M_{ij}$ and $\sigma_j^k$ as in the display after that, define $M_j \equiv M^{(N)}_j$ to be the $d_VT\times d_VT $ block matrix whose entries are
\[
    M_j = \begin{pmatrix}
        M^j_{11} - D_j & M^j_{12} & \cdots  & M^j_{1T} \\
        M^j_{21} & M^j_{22} - D_j& \cdots  & 0 \\
        \vdots & \vdots & \ddots  & \vdots \\
        M^j_{T1} & 0 & \cdots  & M^j_{TT} - D_j
    \end{pmatrix},
\]
where $D_j = \diag ((\sigma_j^1)^2, \dots, (\sigma_j^{d_V})^2)$.
Then, set
\[
    \Gamma^{(j)} = I_{T}\otimes D_j + M_j\quad \text{ for } j = 1, \dots n,
\]
where $\otimes$ is the Kronecker product of matrices.
Finally, we recall the $d_VN \times d_VN$ matrix from \eqref{eq:setting} is given in block-diagonal form by 
\[
    \Gamma_N = \diag\big(\Gamma^{(1)}, \Gamma^{(2)}, \dots, \Gamma^{(n)}\big)
\]
Due to the block structure of $\Gamma_N$, we have
\[
    \Gamma_N^{-1} = \diag\Big((\Gamma^{(1)})^{-1}, (\Gamma^{(2)})^{-1}, \dots (\Gamma^{(n)})^{-1}\Big)
\]
Thus, to show Condition \ref{con:inversematrix} holds, it suffices to show the condition holds for all the matrices $\Gamma^{(j)}$.
\end{remark}
The following result provides a sufficient condition in addition to Proposition \ref{prop:VARCheck} for Condition \ref{con:inversematrix} to be satisfied, which includes some time series.
\begin{proposition}\label{prop:longrange}
    Let $M_j$ and $\sigma_j^k$ be as in Remark \ref{rmk:simplifyingConds}. Suppose that there are positive constants $c_1$ so that $c_1 < (\sigma_j^k)^2$ for all $j$ and $k$. Also, suppose that for some $c < c_1$, $M_j$ satisfies $\max_j \| M_j \|_{\infty \to \infty} \leq c$ for sufficiently large $T$. Then, Condition \ref{con:inversematrix} is satisfied.

    In particular, if $\max_j \|M_j\|_\infty \leq c(d_VT)^{-1}$ or $\max_j \| M_j \| \leq c (d_VT)^{-1/2}$ for sufficiently large $T$, then Condition \ref{con:inversematrix} holds.
\end{proposition}

\subsection{Construction of Tests}\label{sec:ConOfTests}

\subsubsection{Anisotropic wavelets}\label{sec:AnisotropicWavelets}
We will construct wavelets in the case of the domain $Q = [0, 1]\times \Omega$, and we will show they characterize the parabolic Sobolev spaces from \textsection\ref{sec:notation}. As far as we are aware, this characterization has not been proven before for $Q$. It has been done on $\R^{d+1}$, see \textsection 5.2 in \cite{T} for details, and similar wavelets have been considered for the time-cylinder in \cite{SS}.

Let $\phi^{\bc}_{lk}$ (resp. $\psi^{\bc}_{lk}$) be boundary corrected, $S$-smooth Daubechies scaling functions (resp. wavelets) with base resolution $L_0 = L_0(S)$ from \cite{CDV93}. Note that at every level $l \geq L_0$, there are $2^l$ scaling functions and $2^{l}$ wavelets. Next, let $\Phi_{lk}$ (resp. $\Psi_{lk}$) be periodized $d$-dimensional Meyer scaling functions (resp. wavelets). Note that at level $l$, there are $2^{dl}$ scaling functions and $2^{(d+1)l}-2^{dl}$ wavelets. We recall a few facts in \textsection \ref{sec:WaveletFacts} of appendix \ref{sec:besovIsSobolev}, and a general reference is \textsection 4.3 in \cite{GN}.

Now, we will define on $(t, x) \in [0, 1]\times \Omega$
\[
    \Phi_{lk}^{\an}(t, x) = \phi^{\bc}_{l_tk_t}(t)\Phi_{l_xk_x}(x),
\]
where $l = (l_t, l_x)\in \N^2$ has $l_t, l_x\geq L_0$, and $k = (k_t, k_x) \in [2^{l_t}]\times [2^{dl_x}]$.
To introduce the anisotropy, we define the spaces
\begin{align*}
    V^{\an}_l &= \Span \big(\Phi^{\an}_{jk} : j_t =2(l+1), \, j_x = l+1, k \in [2^{j_t}]\times [2^{dj_x}]\big) \\
    W^{\an}_l &= V^{\an}_{l} \cap (V^{\an}_{l-1})^\perp.
\end{align*}
Next, we form a tensor basis
\[\Psi^{\an}_{lk}(t, x) = \psi^{\bc}_{l_tk_t}(t)\Psi_{l_xk_x}(x),\]
where $l = (l_t, l_x)\in \N^2$ has $l_t, l_x\geq L_0 - 1$, we set $\psi^{\bc}_{l_tk_t} = \phi^{\bc}_{L_0k_t}$ when $l_t = L_0 -1$, and $\Psi_{l_xk_x} = \Phi_{L_0 k_x}$ when $l_x = L_0-1$; moreover,
\begin{equation*}
    k = (k_t, k_x)\in \mathcal K(l) \equiv [2^{l_t}]\times [2^{d(l_x +1)}-2^{dl_x}]. 
\end{equation*}
We will also define, if $l_t = L_0 -1$ and $l_x \geq L_0$, that
\begin{equation*}
    \mathcal K(l) = [2^{L_0}]\times [2^{d(l_x +1)}-2^{dl_x}],
\end{equation*}
and if $l_t \geq L_0$ and $l_x = L_0 -1$, then
\begin{equation*}
    \mathcal K(l) = [2^{l_t}]\times [2^{dL_0}].
\end{equation*}
We define the ``level'' of such a wavelet as 
\begin{equation}\label{eq:leveldef}
\lfloor (l_t/2)\lor l_x\rfloor \equiv \mathbb L(l).
\end{equation}

\begin{remark}

    We note that 
    \[
        \{ \Psi_{lk}^{\an} : l_x, l_t \geq L_0-1, \text{ and }k \in \mathcal K (l)\}
    \]
    form an orthonormal basis of $L^2([0, 1]\times \Omega)$. Orthonormality is clear, and we know that this basis is dense in functions of the form $\sum f_i(t)g_i(x)$ for $f_i \in C([0, 1])$ and $g_i\in C(\Omega)$ in the supremum norm, which are in turn dense in $C([0, 1]\times \Omega)$ by Stone-Weierstrass, which are dense in $L^2([0, 1]\times \Omega)$.

    Moreover, we also have that
    \begin{align*}
        V^{\an}_{l} &= \Span(\Psi^{\an}_{jk}: \mathbb L(j) \leq l, k \in \mathcal K(j))\\
        W_l^{\an} &= \Span(\Psi^{\an}_{jk}: \mathbb L(j) = l, k \in \mathcal K(j)).
    \end{align*}
\end{remark}
\begin{definition}\label{def:AnisotropicSpace}
    For $s > 0$, fix a smoothness $S$ so that $\phi^{\bc}_{lk}, \psi^{\bc}_{lk} \in C^r$ for a natural number $r > s$. We define the anisotropic Sobolev spaces as
    \[H^{s, \an} = \{ f\in L^2(Q) : \|f\|_{H^{s, \an}} < \infty\},\]
    where 
    \[\|f\|_{H^{s, \an}} = \left( \sum_{\substack{l \in \N^2\\ l_t, l_x \geq L_0 -1}} \sum_{k \in \mathcal K(l)}2^{2s\lfloor \max(l_t/2, l_x)\rfloor }|\langle f, \Psi^{\an}_{lk}\rangle|^2\right)^{1/2}.\]
\end{definition}
\begin{remark}
    The definition of these Sobolev spaces can be extended to Besov spaces $B^{s, \an}_{pq}$, see Definition \ref{def:AnisotropicBesov} in the appendix \textsection \ref{sec:besovIsSobolev}, and a few of these spaces will be needed for the proofs in Appendix \ref{sec:besovIsSobolev}. We also quickly note that the definition of $H^{s, \an}$ depends only on the spaces $V_{L_0 -1}^{\an}$ and $W_{j}^{\an}$ for $j \geq L_0$. Indeed, suppose that $f\in L^2(Q)$. If we let $f_{L_0-1}$ be the orthogonal projection of $f$ onto $V_{L_0 -1}^{\an}$ and $f_j$ be the orthogonal projection on to $W_j^{\an}$, then the following norm is equivalent to $\|\cdot\|_{H^{s, \an}}$:
    \begin{equation}\label{eq:atomicdecomp}
        \|f\|^2 = \sum_{j \geq L_0 -1} 2^{2sj}\|f_j\|^2_{L^2(Q)}.
    \end{equation}
    These wavelet norms characterize the anisotropic Sobolev spaces $H^{s, s/2}$ mentioned in Remark \ref{rmk:besovIsSobolev}.
\end{remark}
\begin{proposition}\label{prop:besovIsSobolev}
    Let $m\in \N$, then 
    \begin{align*}
        \bigcap_{k=0}^{m} H^k([0, 1], H^{2m-2k}(\Omega)) &= H^{2m, \an},\\
        \bigcap_{k = 0}^{m+1} H^k([0, 1], H^{2m+1 - 2k}(\Omega)) &\subseteq H^{2m+1, \an},
    \end{align*}
    where $\subseteq$ indicates a continuous embedding, and equality indicates equivalent norms. As a consequence, we have $H^{s, \an} = H^{s, s/2}$ for all $s\geq 0$. 
\end{proposition}
\begin{remark}
    For the proof of this proposition, see Appendix \ref{sec:besovIsSobolev}. Henceforth (except in Appendix \ref{sec:besovIsSobolev}), we will denote $H^{s, \an}$ by the usual notation of $H^{s, s/2}$.
\end{remark}

We also want to prove an $L^\infty$ embedding result for $H^{s, s/2}$, which is covered in the next proposition.

\begin{proposition}\label{prop:BesovEmbeddings}
    For $s > (d+2)/2$, we have that
    \[
        H^{s, s/2} \hookrightarrow C(Q).
    \]
\end{proposition}
In addition, the usual embeddings between Besov spaces hold, see Proposition \ref{prop:BesovEmbeddings2}. We will use the following corollary to control the bias for taking $V^{\an}_j$ as an approximation space in Proposition \ref{prop:Mestimation}, and its proof follows from the wavelet characterization of the $L^2$ norm.

\begin{corollary}\label{cor:AnisotropicApprox}
    Suppose that $\alpha \geq 0$, $u \in H^{\alpha, \alpha/2}$ and let $u_L$ be the projection of $u$ onto $V_L^{\an}$. Then, there is a constant $c = c(\alpha)$ so that 
    \[
        \| u - u_L \|_{L^2} \leq c\|u\|_{H^{\alpha, \alpha/2}} 2^{-\alpha (L+1)}.
    \]
\end{corollary}
We will also need to know the metric complexity of a ball in $H^{s, s/2}$ for Proposition \ref{prop:Mestimation}, which we record here. It is a consequence of the ``transference method'' in \textsection 5.3 in \cite{T}.
\begin{corollary}\label{cor:metricentropy}
    Let $B(R) = B_{H^{s, s/2}}(R)$ be a ball of radius $R$ in $H^{s, s/2}$, and $s > (d+2)/2$. Then, 
    \[\log N(B(R), \|\cdot \|_{L^\infty}, \epsilon) \lesssim \bigg(\frac{R}{\epsilon}\bigg)^{\frac{d+2}{s}}.\]
\end{corollary}

\subsubsection{Construction of the Estimator}\label{sec:Construction}

We now use anisotropic wavelets to construct an estimator for our regression problem.
\begin{remark}
    Throughout this section, we will assume that $V$ is $1$-dimensional in the setting of \eqref{eq:setting}. In examples such as the Navier-Stokes equations, $V$ is indeed not $1$ dimensional, but in the vector valued case, we can estimate $u_\theta$ coordinate-wise and get the same results as in the real-valued case.
\end{remark}

Throughout this section, we will assume that the forward map $\mathcal G$ satisfies Conditions \ref{con:forwardregularity} and \ref{con:forwardmapping}, for a regularity level $a > (d+2)/2$, and $\theta \in \mathcal R$. We will also write $\mathcal G(\theta) = u_\theta$ for simplicity. Recall $\mathbb L$ from \eqref{eq:leveldef} giving the ``level'' of a wavelet.
Consider minimizing the penalized least squares loss
\begin{align*}
    L_N(\beta) = \frac{1}{N}\sum_{i=1}^T\sum_{j = 1}^n\bigg(Y_{ij} - \sum_{\mathbb L(l)\leq L}\sum_{k \in \mathcal K(l)}\beta_{lk}&\Psi_{lk}^{\an}(t_i,X_{ij})\bigg)^2 + \delta^2_N\|\beta\|^2_{h^a},
\end{align*}
over $\beta = (\beta_{lk})_{\mathbb L(l) = L_0 -1 , k\in \mathcal K(l)}^L$ a real-valued tuple (assumed to be zero for $\mathbb L(l) > L$), and with ellipsoidal Sobolev $h^\alpha$-norms
\[
    \|\beta\|^2_{h^\alpha} = \sum_{L_0 -1 \leq \mathbb L(l)}\sum_{k \in \mathcal K(l)}2^{2\mathbb L(l)\alpha}(\beta_{lk})^2,
\]
where the growth rates of $\delta_N$, $L$ are: $\delta_N \asymp N^{-a/(2a+d+2)}$ and 
\begin{equation}\label{eq:growthrates}
    2^{-L} \asymp N^{-\frac{1}{2a + d + 2}}, \quad \text{ i.e. } 2^{L(d+2)} \asymp N \delta_N^2.
\end{equation}
Because of the penalty, the loss function is strongly convex, so we can define $\hat \beta$ as the unique minimizer of $L_N(\beta)$. Then, define $\hat u = \hat u_{L}$ by
\begin{equation}\label{eq:estimator}
    \hat u_{L}(t, x) = \sum_{L_0 -1\leq \mathbb L(l)\leq L, k}\hat \beta_{lk}\Psi_{lk}^{\an}(t, x), \text{ for }t\in [0, 1], x\in\Omega.
\end{equation}
The projection of the true regression function $u_\theta$ on to the approximation space $V^{\an}_L$ is defined to be $u_{\theta, L}$ as
\begin{equation}\label{eq:projectionoftruth}
    u_{\theta, L} = \sum_{L_0 - 1\leq \mathbb L(l)\leq L, k} \beta_{lk}\Psi_{lk}^{\an}, \text{ where }\beta_{lk} = \langle u, \Psi_{lk}^{\an}\rangle.
\end{equation}
Next, for elements $v, v'\in V^{\an}_L$ with coefficients $\beta^1$ and $\beta^2$ respectively, let us define `empirical' norms
\[
    \|v-v'\|^2_{(N)} = \frac{1}{N}\sum_{i=1}^T\sum_{j=1}^n\left(\sum_{L_0-1\leq \mathbb L(l)\leq L, k}(\beta^{1}_{lk}-\beta_{lk}^{2})\Psi_{lk}^{\an}(t_i, X_{ij})\right)^2
\]
Also note that by Corollary \ref{cor:AnisotropicApprox} that if $\theta \in \mathcal H$, then $u_\theta \in H^{a, a/2}$ and
\begin{equation}\label{eq:biascontrol2}
    \|u_\theta - u_{\theta, L}\|_2 \lesssim 2^{-a L} \asymp \delta_N.
\end{equation}
Now we deduce a result from $M$-estimation:
\begin{proposition}\label{prop:Mestimation}
    In the setting of \eqref{eq:timedependent}, suppose that Conditions \ref{con:forwardregularity} and \ref{con:forwardmapping} hold, with $a > (d+2)/2$, and that Condition \ref{con:variancebound} holds for some constant $C_1$.
    Let $\Theta_N$ be the regularization set from Theorem \ref{thm:smallball} for some $B$ and sequence $\varepsilon_N$, and let $\delta_N$, $\hat u$ be as in \eqref{eq:growthrates} and \eqref{eq:estimator} respectively. Also, suppose that
    \[
        n \lesssim N\delta_N\text{ and } T \gtrsim \delta_N^{-1}, \text{ and } \varepsilon_N \lesssim \delta_N.
    \]
    Then, for any $K>0$ there is a constant $c$ so that for all $\theta\in \Theta_N$, we have
    \[
        \Prob^N_\theta\left( \|\hat u - u_{\theta}\|^2_{(N)} + \delta_N^2\|\hat \beta\|^2_{h^a} > c\delta_N^2\right) \leq e^{-KN\delta_N^2},
    \]
    where the constant $c$ depends only on $B, C_1, K$ and $d$.
\end{proposition}

\subsubsection{Proof of Theorem \ref{thm:TestsExist}}
Our goal is to leverage Proposition \ref{prop:Mestimation} to demonstrate the existence of tests in the $\|\mathcal G(\theta) - \mathcal G(\vartheta)\|_{L^2}$ metric on $\Theta$.
To do this, we need to show the empirical norm for our observation scheme concentrates around the $L^2(Q)$ norm. We cannot use Lemma \ref{lem:deterministicMconcentration} as in the proof of Proposition \ref{prop:Mestimation} on $\hat u$ as it depends on the $X_i$, and hence Bernstein's inequality does not immediately apply. 

\begin{proposition}\label{prop:Mconcentration}
    Let $K\in \N$ be some fixed positive constant. Assume the set-up and conditions of Proposition \ref{prop:Mestimation} hold with $a \geq 6$, and 
    \[
        n \gtrsim N^{\frac{2d+3}{2a + d + 2}} \quad \text{ and }\quad T \lesssim N^{\frac{2a - d - 1}{2a + d +2}}, 
    \]
    as well as that $\theta \in \Theta_N$ and $\varepsilon_N \lesssim \delta_N$. Then there exists some threshold $N_0 \in \N$ so that for $N \geq N_0$,
    \[
        \Prob_{\theta}^{N}(\| \hat u - u_{\theta, L} \|^2_{L^2([0, 1]\times \Omega)} \leq 2\| \hat u - u_{\theta, L}\|^2_{(N)} + \frac{c}{T^2} ) \geq 1 - e^{-KN\delta_N^2}
    \]
    where $N_0 = N_0(K, a, d, B)$, and $c$ depends only on $B$, $C_1, K$, and $d$.
\end{proposition}
\begin{remark}\label{rmk:whyscheme}
    Lemma \ref{lem:deterministicMconcentration} and Proposition \ref{prop:Mconcentration} are the reason why we specialize to the case of \eqref{eq:timedependent}, and any other specific choice of design in \eqref{eq:setting} will work provided similar facts can be established. 
    
    For the proof see the Appendix, \textsection \ref{sec:MconcentrationSupp}. With this, we can now construct tests.
\end{remark}

\begin{proof}[Proof of Theorem \ref{thm:TestsExist}]
    The tests we consider are
    \[
        \phi_N = \mathbbm{1}_{\{\|\hat u - u_{\theta_0}\|^2_{L^2(Q)}> E\delta_N^2 \text{ or } \|\hat \beta\|^2_{h^a} > c\}}
    \]
    for $c$ and $E$ appropriately chosen and large enough.
    For the type one and two error, we can apply Proposition \ref{prop:Mestimation} and then Proposition \ref{prop:Mconcentration} (noting that $\varepsilon_N \lesssim \delta_N$) and \eqref{eq:biascontrol2}.
\end{proof}

\section{Applications to PDEs}\label{sec:applications}
Proofs for this section can be found in \textsection \ref{sec:ProofsforApplications}.
\subsection{Heterogeneous Reaction Diffusion Equation}\label{sec:HRDE}
Throughout this section, we are in the setting of \eqref{eq:introRDE} with $d \leq 3$. In this case, we have
\begin{proposition}\label{prop:HRDEReg}
    Assume $D\in C^1(\Omega)$. For $\varphi \in H^1(\Omega)$, $D \in H^s(\Omega)$, $f \in C_c^\infty(\R)$ and $d \leq 3$, there exists a unique solution $u$ of \eqref{eq:introRDE}. Moreover, suppose that $\|\varphi\|_{H^\gamma}, \|D\|_{H^{\gamma}}\leq B$ for $0\leq \gamma \leq s$. Then we have
    \[
        u \in H^{\gamma + 1, (\gamma+1)/2}, \quad \text{ with }\quad \|u\|_{H^{\gamma + 1, (\gamma +1)/2}}\leq C
    \]
    where $C = C(B, f, \gamma, d, D_{\min}, \|D\|_{C^1}, S)$.
\end{proposition}
\noindent
See Appendix \ref{sec:PDEReg}, Theorem \ref{thm:strongsolsHRDE}.
This proposition justifies defining the forward map $\mathcal G$ as in \eqref{eq:HRDEforwardmap}, with the parameter space taken to be 
\begin{equation}
    \Theta = \{ D\in C^1(\Omega) : \inf_{x\in \Omega} D(x) \geq D_{\min}\}, 
\end{equation}
where $D_{\min}$ is as in \eqref{eq:introRDE}, and we will consider regularization spaces $\mathcal R = H^\beta(\Omega)$ for $\beta$ large enough.

\textbf{Differentiation of the forward map.} As in \cite{CN}, to prove forward and backwards stability, we will compute the derivative of the forward map in a suitable space. To achieve this, we will first define the auxiliary map
\begin{equation}\label{eq:levelsetmap}
\begin{split}
    G: C^1(\Omega) \times F &\to L^2(Q)\times H^1(\Omega)\\
    (D, u) &\mapsto \left(\frac{\partial u}{\partial t} - \nabla \cdot (D\nabla u) - f(u),u(0) - \varphi \right), 
\end{split}
\end{equation}
where $F = L^2([0, 1], H^2(\Omega)) \cap H^1([0, 1], L^2(\Omega))$, with the usual norm. We note also by Theorem 4, \textsection 5.9 in \cite{Evans} that  $F \hookrightarrow C([0, 1], H^1(\Omega))$, so the second component of \eqref{eq:levelsetmap} is valid.
\begin{lemma}\label{lem:levelsetmap}
    The map $G$ from \eqref{eq:levelsetmap} is differentiable in the Fr\'echet sense, with derivative
    \[
        DG(D, u)[h, v] = \left(\frac{\partial v}{\partial t} - \nabla \cdot (D\nabla v) - \nabla \cdot(h\nabla u) - f'(u)v, v(0)\right).
    \]
\end{lemma}

\noindent
See section \textsection \ref{sec:levelsetmap} for the proof. Note that the subset 
\begin{equation}\label{eq:defofU}
    \mathcal U = \{ D \in C^1 : \inf_{x \in \Omega} D(x) > D_{\min}/2 \}
\end{equation}
is an open subset of $C^1$. Also, by Theorem \ref{thm:strongsolsHRDE}, on this subset, $\mathcal G$ is the unique function such that
\[
    G(D, \mathcal G(D)) = 0.
\]
Finally, we know that the map
\[
    D_2G(D, \mathcal G(D))[v] = DG(D, \mathcal G(D))[0, v] = \left(\frac{\partial v}{\partial t} - \nabla \cdot(D\nabla v)-f'(u)v, v(0)\right)
\]
is a linear isomorphism from $F \to L^2(Q)\times H^1(\Omega)$ by Proposition \ref{prop:LinearParabolicRegularity}. Thus, by Theorem B.1 in \cite{CN} we have

\begin{proposition}
    Let $d\leq 3$. For $\mathcal U$ as in \eqref{eq:defofU} and $F$ from \eqref{eq:levelsetmap}, the map $\mathcal G: \mathcal U\to F$ defined as in \eqref{eq:HRDEforwardmap} is $C^1$ in the Fr\'echet sense, Moreover, for any $D \in \mathcal U$, its Fr\'echet derivative at $D$ is given by
    \begin{align*}
        D\mathcal G(D) : C^1(\Omega) &\to F \\
        h &\mapsto i_h
    \end{align*}
    where $i_h \in F$ is the unique strong solution to the linear parabolic PDE
    \[
        \begin{cases}
            \partial_t i_h - \nabla \cdot (D\nabla i_h) - f'(u_D)i_h &= -\nabla\cdot(h\nabla u_D) \text{ on } [0, 1]\times \Omega \\
            i_h(0)&=0\quad \text{ on } \Omega.
        \end{cases}
    \]
\end{proposition}

We can now derive forward and backward stability estimates.

\noindent \textbf{Forwards Regularity.} Recall that $Q = [0, 1]\times \Omega$. We know from Proposition \ref{prop:HRDEReg} that for $\beta > d/2 + 1$ and $M > 0$, there exists some $U$ such that
\[
    \sup_{D \in \Theta \cap B_{H^\beta(\Omega)}(M)} \|\mathcal G(D)\|_{L^\infty(Q)} \leq U, 
\]
which verifies the first part of Condition \ref{con:forwardregularity}, provided that the regularization space controls the $H^\beta(\Omega)$ norm. (Note that we have also used Proposition \ref{prop:BesovEmbeddings}). Condition \ref{con:forwardmapping} also holds by Proposition \ref{prop:HRDEReg} provided $\beta \geq 5$.

\begin{proposition}[Forward Regularity for HRDE]\label{prop:HRDEForward}
    Suppose that $\beta > d/2 +1$ and $\mathcal R = H^\beta(\Omega)$. Then, for all $M>0$, there exists constants $U$ and $L$ depending only on $M, D_{\min}, f, \varphi$, and $d$ such that
    \begin{equation*}
        \sup_{D\in \Theta \cap B_{H^\beta(\Omega)}(M)} \|\mathcal G(D)\|_{L^\infty(Q)} \leq U, 
    \end{equation*}
    and
    \begin{equation*}
        \|\mathcal G(D_1) - \mathcal G(D_2)\|_{L^\infty_t L^2_x} \leq L\|D_1 - D_2\|_{L^2} \, \text{ for all }D_1, D_2\in \Theta \cap B_{\mathcal R}(M).
    \end{equation*}
\end{proposition}

\noindent \textbf{Backwards Stability.} We again utilize a trick from the differentiability of $G$. However, we first collect some facts about stability for the transport operator $\nabla \cdot(h\nabla u)$.

\begin{condition}\label{con:Darcy}
    Let $u$ be the solution to \eqref{eq:introRDE} for the diffusivity $D$. Suppose that 
    \[
        \inf_{x\in \Omega} \int_0^1 | \nabla u(t, x)|^2 \, dt \geq c_0
    \]
    holds for some positive constant $c_0$.
\end{condition}

\begin{condition}\label{con:Darcy2}
    Let $u$ be the solution to \eqref{eq:introRDE} with $D$ the diffusivity, and suppose that $\Omega' \subseteq \Omega$ is some open set such that
    \begin{itemize}
        \item $D \equiv D_{\min} + 1$ on $\Omega \setminus \Omega'$, and 
        \item For some positive constants $\mu, c_0 >0$, and time $t \in [0, 1]$,
            \[
                \inf_{x \in \Omega'} \frac{1}{2} \Delta u(t, x) + \mu |\nabla u(t,x)|^2 \geq c_0.
            \]
    \end{itemize}
\end{condition}
\begin{remark}
    Conditions \ref{con:Darcy} and \ref{con:Darcy2} are incomparable conditions, and we will show that both can be satisfied with a specific choice of $\varphi$ the initial condition depending on $D_0$ only through $D_{\min}$ and $\|D\|_{C^1}$, see Propositions \ref{prop:RDvarphi} and \ref{prop:NSvarphi}. In assuming Condition \ref{con:Darcy}, we get stability on the entire domain $\Omega$.\footnote{Indeed, for a general bounded domain $\Omega \subseteq \R^d$, we need only assume that $D_0$ is known at $\partial \Omega$.} However, we prove stability in the $L^1$ norm, which leads to a worse rate in $L^2$. In assuming Condition \ref{con:Darcy2} we have to assume more knowledge about $D_0$, but obtain a better rate.

    In the case of Condition $\ref{con:Darcy2}$ it is not possible to take any $\Omega'$ such that $\overline {\Omega'} = \Omega$. At any local maximum $x^\ast\in \Omega$ of $u(t, \cdot)$, we have $\nabla u(t, x^\ast) = 0$, and $\Delta u(t, x^\ast) \leq 0$.\footnote{And then take some $x_n\in \Omega'$ such that $x_n \to x^\ast$.}
    However, we will show that Condition \ref{con:Darcy2} can be satisfied for any $\overline V \subsetneq \Omega$ for a specific choice of $\varphi$ in Appendix \textsection \ref{sec:findingphi}.
    
    The difficulty with stability for Darcy's transport operator lies in trying to recover information about $h$ from $\nabla \cdot (h\nabla u)$ at critical points of $u$. Condition \ref{con:Darcy} translates into the condition that critical points of $u(t, \cdot)$ change with time. We show that this condition is generic in Proposition \ref{prop:opendense}.
\end{remark}

\begin{lemma}[Stability for the Transport Operator]\label{lem:transportstability}
    Suppose that $D_1$ is such that $\|D_1\|_{H^\beta} \leq M$, for $\beta > d/2 + 1$. Then,
    \begin{enumerate}[label = (\alph*)]
        \item If Condition \ref{con:Darcy} holds for $u_{D_1}$ and some constant $c_0$, we have for all $h \in H^1(\Omega)$,
            \begin{align*}
                \|h\|_{L^1(\Omega)} &\leq C(M, d, c_0, \beta, D_{\min}, f)\int_0^1 \| \nabla \cdot(h\nabla u_{D_1})(t, \cdot)\|_{L^2(\Omega)}\, dt \\
                &\leq C \|\nabla \cdot(h \nabla u_{D_1})\|_{L^2(Q)}.
            \end{align*}
        \item If Condition \ref{con:Darcy2} holds for $u_{D_1}$ and some open set $\Omega' \subsetneq \Omega$, then for all $h\in H^1_c(\Omega')$, 
            \[
                \|h\|_{L^2(\Omega)} \leq C(M, d, c_0, \mu, \beta, D_{\min}, f) \|\nabla \cdot(h\nabla u_{D_1})\|_{L^2(Q)}.
            \]
    \end{enumerate}
\end{lemma}

\begin{proposition}[Backwards Stability]\label{thm:HRDEBack}
    Suppose that the Condition \ref{con:Darcy}, or respectively \ref{con:Darcy2}, is satisfied for $u_{D_1}$ where $D_1 \in \Theta$ is the true parameter, and some domain $\Omega'$ in the case of Condition \ref{con:Darcy2}. Then, for the choice of regularization space $\mathcal R = H^\beta(\Omega)$, or $H^\beta_c(\Omega')$, so that $\beta > d/2+ 4$, and
    \[
        \eta = \begin{cases}
            \frac{\beta -1}{\beta +1}\frac{\beta}{\beta + d/2 + 1} & \text{ if Condition \ref{con:Darcy} holds,} \\
            \frac{\beta -1}{\beta + 1} & \text{ if Condition \ref{con:Darcy2} holds,}
        \end{cases}
    \]
    we have that for all $M >0$, there is a constant $L'$ depending on $M, d, c_0, D_{\min}, \beta,$ and also $\mu$ in the case of Condition \ref{con:Darcy2}, such that
    \[
        \|D_1- D_2\|_{L^2(\Omega)} \leq L'\|\mathcal G (D_1) - \mathcal G(D_2) \|_{L^2(Q)}^{\eta} \quad \text{ for all }D_1, D_2 \in B_{\mathcal R}(M)\cap \Theta.
    \]
\end{proposition}

Now we have established the properties of the forward map necessary to verify all the conditions of Theorem \ref{thm:Main}. The only issue is the parameter space is not linear because the diffusivity be everywhere positive. To fix this, we need to introduce a linking function
\[
    g(\theta) = D_{\min} + \exp(\theta).
\]
Note that from Proposition 3.9 in \cite{Taylor}, $g$ maps $L^\infty \cap H^k$ to $H^k$ and takes bounded sets to bounded sets.
    
\begin{theorem}\label{thm:HRDEMain}
    Assume that Condition \ref{con:Darcy} or \ref{con:Darcy2} is satisfied for the ground truth $D_0 \equiv g(\theta_0)$. 
    Suppose that a Gaussian process prior $\Pi'$ verifying Condition \ref{con:RKHS} with $\mathcal R = H^\beta(\Omega)$ or $H^\beta_c(\Omega')$ resp. for some $\beta > 4 + d/2$ and RKHS $\mathcal H \subseteq H^\alpha(\Omega)$ for some $\alpha > \beta + d/2$. Let the prior $\Pi_N$ be the image measure under $g$ of the probability measure from \eqref{eq:priorrescaling} for $\kappa = 0$. 
    Consider the forward map $\mathcal G : \Theta \to L^2(Q)$ defined in \eqref{eq:HRDEforwardmap}, and assume that $\theta_0\in H^\alpha\cap \mathcal H$ or $H^\alpha_c(\Omega')\cap \mathcal H$ resp. 

    Let $\Pi(\cdot | (Y_{ij}, X_{ij}))$ be the posterior distribution arising from observations of the model \eqref{eq:timedependent} where Condition \ref{con:variancebound} holds, and suppose one of Condition \ref{con:inversematrix} or \ref{con:inversematrix2} holds. Let $\delta_N = N^{-(\alpha+1)/(2(\alpha+1) + d +2)}$, $\varepsilon_N = N^{-\alpha/(2\alpha + d)}$. Then for all $b>0$, we can choose $m$ large enough such that as $N\to \infty$
    \[
        \Prob^N_{D_0} \big( \Pi_N(D\in \Theta: \|\mathcal G(D)-\mathcal G(D_0)\|_{L^2(Q)} \leq m\delta_N | D_N) \leq 1 - e^{-bN\varepsilon_N^2}\big) \to 0.
    \]
    Moreover, letting $\eta$ be as in Proposition \ref{thm:HRDEBack}, and constant $M$ large enough, we also have
    \[
        \Prob_{D_0}^N\big(\Pi_N(D \in \Theta: \|D - D_0\|_{L^2(\Omega)} \leq M\delta_N^\eta | D_N) \leq 1 - e^{-bN\varepsilon_N^2}\big) = o(1).
    \]
\end{theorem}
\begin{proof}
    The proof uses Theorem \ref{thm:Main}, where the parameter space is $\Theta' = \mathcal R$ and the forward map is $\mathcal G' = \mathcal G \circ g$. To satisfy the Conditions \ref{con:forwardregularity} and \ref{con:stabilityestimate}, we use Propositions \ref{prop:HRDEForward} and \ref{thm:HRDEBack} coupled with the fact that $g$ is bi-Lipschitz on bounded sets, see the ideas in \textsection 2.1.1.3 and Proposition 2.1.5 in \cite{N}. Then, push forward the prior via the map $g$ using the fact that the Bayesian updates commute with taking the image measure to deduce the result.
\end{proof}

\subsection{$2$D Navier-Stokes Equations with unknown viscosity}\label{sec:NS}
Throughout this section, we are in $2$-dimensions, and are in the setting of \eqref{eq:introNS}, with forward map $\mathcal G$ given in \eqref{eq:NSforwardmap}. We set the parameter space to be $\Theta = [\nu_{\min}, \infty) \times L^2(\Omega)^2 \subseteq L^2(\Omega, \R^3)$. Similarly to the case of the HRDE, the parameter space is not a linear space, and we will resolve this problem by considering an appropriate link function. 
The scale of regularization spaces we will consider are 
\begin{equation}\label{eq:regularizationNS}
    \mathcal R^\beta = \R \times \dot H^\beta_{\diamond}, 
\end{equation}
which imbeds into $H^\beta(\Omega)^3$. 

We assume that the viscosity is constant to leverage existing regularity theory for the $2$D Navier-Stokes equations, and a ``variable viscosity'' term could be considered following the framework of \textsection \ref{sec:HRDE}. 

\begin{proposition}\label{prop:NSReg}
    Fix an integer $a \geq 1$. For $\varphi \in \dot H^a_\diamond(\Omega)$, $\nu \in [\nu_{\min}, \infty)$ and some forcing function $f \in \dot H_\diamond^{a-1}$. Let $u$ be a strong solution of \eqref{eq:introNS}. Suppose that $\nu, \|\varphi\|_{H^a} \leq B$. Then we have
    \[
        u\in H^{a + 1, (a+1)/2}, \quad \text{ with }\quad \|u\|_{H^{a+1, (a+1)/2}}\leq C,
    \]
    where $C = C(B, f, a, d, \nu_{\min}, S)$.
\end{proposition}
Solutions to \eqref{eq:introNS} exist and are unique, see Appendix \ref{sec:PDEReg}, \textsection \ref{sec:NSRegTheory} for details.

\begin{proposition}[Forward Regularity for 2D N-S]\label{prop:ForwardRegNS}
    Let $\beta  > d/2 +1$ and $\mathcal R = \mathcal R^\beta$ from \eqref{eq:regularizationNS}. Then, for all $M>0$, there exists constants $U$, $L$ depending only on $\nu_{\min}$, $M$, $d$, and $f$, such that 
    \begin{equation*}
        \sup_{(\nu, \varphi)\in \Theta \cap B_{\mathcal R}(M)} \|\mathcal G(\nu, \varphi)\|_{L^\infty(Q)}\leq U,
    \end{equation*}
    and 
    \begin{equation*}
        \|\mathcal G(\nu_1, \varphi_1) - \mathcal G(\nu_2, \varphi_2)\|_{L^\infty_tL^2_x}\leq L\big(\|\varphi_1 - \varphi_2\|_{L^2} + |\nu_1 - \nu_2|\big),
    \end{equation*}
    for all $(\nu_i, \varphi_i)\in \Theta \cap B_{\mathcal R}(M)$, where $i = 1, 2$.
\end{proposition}
The proof uses standard techniques from PDE analysis, e.g. Gr\"owall's inequality--see Appendix \ref{sec:PDEReg}, \textsection \ref{sec:ForwardRegNS}.
\begin{proposition}[Backward Stability for 2D N-S]\label{prop:NSBack}
    Suppose we have for $i=, 1, 2$ that $\theta_i = (\nu_i, \varphi_i) \in \Theta$, and the forcing $f \neq 0$ with $f \in \dot H^\beta_\diamond$. Then for any choice of regularization space $\mathcal R = \mathcal R^\beta$ from \eqref{eq:regularizationNS} so that $\beta > 4 + d/2$, and $\eta = (\beta - 1)/(\beta + 1)$. We have that for all $M > 0$, there is a constant $L'$ depending on $M, d, \nu_{\min},$ and $f$ such that
    \[
        \|\varphi_1 - \varphi_2\|_{L^2} + |\nu_1 -\nu_2| \leq L' \|u_{\nu_1, \varphi_1} - u_{\nu_2, \varphi_2}\|_{L^2}^{\eta} \quad \text{ for all }\theta_1, \theta_2 \in B_{\mathcal R}(M).
    \]
\end{proposition}

\begin{lemma}\label{lem:LaplaceNonZero}
    Suppose that the forcing function $f\in \dot H^1_\diamond$ is non-zero, and let $\theta \in H^1(\Omega)$. Let $u_{\nu, \theta}$ be the corresponding solution to \eqref{eq:introNS} with $\nu \leq M$.  Then, there exists a constant $c_0$ such that
    \[
        \int_0^1 \|A u_\theta\|_{L^2}^2 \, dt \geq c_0 > 0,
    \]
    where $c_0 = c_0(f, M, d)$.
\end{lemma}
\noindent
See Appendix \ref{sec:ProofsforApplications}, \textsection \ref{sec:LaplaceNonZero} for the proof. Again, because $\Theta$ is not quite a linear space, we need to have a link function of the form, for $\theta \in L^2(\Omega, \R^3)$,
\[
    g(\theta) = (\nu_{\min} + \exp(\theta_3), \theta_1, \theta_2),
\]
and again note that $g$ maps the $\mathcal R^\beta$ spaces into $\mathcal R^\beta \cap \Theta$ for $\beta > d/2$.

\begin{theorem}\label{thm:NSMain}
    Suppose that a Gaussian process prior $\Pi'$ verifying Condition \ref{con:RKHS} with $\mathcal R =\mathcal R^\beta$ for some $\beta > 4 + d/2$ as in \eqref{eq:regularizationNS} and RKHS $\mathcal H \subseteq H^\alpha(\Omega)$ for some $\alpha > \beta + d/2$. Let the prior $\Pi_N$ be the image measure under $g$ of the probability measure from \eqref{eq:priorrescaling} for $\kappa = 0$. Consider the forward map $\mathcal G : \Theta \to L^2(Q)$ derived from \eqref{eq:introNS}, and assume that $\theta_0\in \mathcal H$ with $f \in \dot H^\alpha_\diamond$ non-zero and $(\nu_0, \varphi_0) = g(\theta_0)$.

    Let $\Pi(\cdot | (Y_{ij}, X_{ij}))$ be the posterior distribution arising from observations of the model \eqref{eq:timedependent} where Condition \ref{con:variancebound} holds, and one of Condition \ref{con:inversematrix} or \ref{con:inversematrix2} holds. Let $\delta_N = N^{-(\alpha+1)/(2(\alpha+1) + d +2)}$, $\varepsilon_N = N^{-\alpha/(2\alpha + d)}$. Then for all $b>0$, we can choose $m$ large enough such that as $N\to \infty$
    \begin{align*}
        \Prob^N_{\nu_0, \varphi_0} \big( \Pi_N((\varphi, \nu)\in \Theta, \|\mathcal G(\nu, \varphi)-\mathcal G(\nu_0, \varphi_0)&\|_{L^2(Q)} \leq m\delta_N | D_N) \leq 1 - e^{-bN\varepsilon_N^2}\big) \to 0.
    \end{align*}
    Moreover, letting $\eta = (\beta - 1)/(\beta + 1)$ and constant $M$ large enough, we also have
    \begin{align*}
        \Prob_{\nu_0, \varphi_0}^N\big(\Pi_N((\nu, \varphi)\in \Theta : \|\varphi - \varphi_0\|_{L^2(\Omega)} + &|\nu - \nu_0| \leq M\delta_N^\eta |D_N)\leq  1 - e^{-bN\varepsilon_N^2}\big) = o(1).
    \end{align*}
\end{theorem}

\begin{appendix}
\section{Technical Results}\label{sec:Technical}
\subsection{A posterior contraction rate theorem}
We base the proof of our posterior contraction rate theorems on the following general Posterior contraction rate result, which is based on Theorem 8.22 in \cite{GvdV}. The proof is standard, and can be found in Appendix \ref{sec:PosteriorGeneral}.

\begin{theorem}[Posterior Contraction Rate]\label{thm:PosteriorGeneral}
    For each $\theta\in \Theta$, let $(\Prob_\theta^N)_{N \geq 1}$ be a sequence of probability laws on measurable spaces $(X_N, \mathcal F_N)$, and suppose they all have densities $p_\theta^N$ with respect to some probability measure $\mu_N$ on $X_N$. We assume $\text{Law}(D_N) = \Prob_{\theta_0}^N$ for some ground truth $\theta_0 \in \Theta$.

    Let $\Pi_N$ be a sequence of prior Borel probability measures on $\Theta$, and let $\Pi_N(\cdot \mid D_N)$ be the posterior probability distribution on $\Theta$ given by
    \begin{equation*}
        d\Pi_N(\theta| D_N) \propto p^N_\theta(D_N) d\Pi_N(\theta).
    \end{equation*} 
    Let $\mathcal G$ be as in \eqref{eq:forwardmap}, and assume that $\|\mathcal G(\theta_0)\|_\infty\leq U$. Suppose there is a sequence $\varepsilon_N \to 0$ such that $N\varepsilon_N^2 \to \infty$ and $N\varepsilon_N^2 \geq 1$, the sets
    \[
        \mathcal B_N = \{ \theta \in \Theta : KL(p_{\theta_0}^N, p_{\theta}^N) \leq N\varepsilon_N^2, V_{2, 0}(p_{\theta_0}^N, p_\theta^N) \leq N\varepsilon_N^2 \}
    \]
    satisfy, for all $N$ large enough, 
    \begin{equation}\label{eq:smallball}
        \Pi_N(\mathcal B_N) \geq e^{-AN\varepsilon_N^2},
    \end{equation}
    for some $A>0$. Further, assume that there exists a sequence of Borel subsets $\Theta_N \subseteq \Theta$ for which 
    \begin{enumerate}
        \item  For some $B > A + 2$,
            \begin{equation}\label{eq:regularizationset}
                \Pi_N(\Theta_N^c) \leq e^{-BN\varepsilon_N^2}, \text{ and }
            \end{equation}
        \item There are tests (indicator functions) $\phi_N$ and a sequence $\delta_N \to 0$ such that for some $C$, and $K \geq B$,
            \vspace{-3mm}
            \begin{equation}\label{eq:testcondition}
                \E^N_{\theta_0}[\phi_N] \to 0, \text{ and } \sup_{\substack{\theta \in \Theta_N \\ C\delta_N < \|\mathcal G(\theta)  -\mathcal G (\theta_0)\|_{L^2}}} \E_{\theta}^N[1 - \phi_N] \leq e^{-K N\varepsilon_N^2}
            \end{equation}
            \vspace{-5mm}

            Then, for any $0 < b < B- A - 2$, we have, as $N\to \infty$,
            \begin{equation*} 
                \Prob_{\theta_0}^N\big(\Pi_N(\theta \in \Theta_N :\|\mathcal G(\theta) - \mathcal G(\theta_0)\|_{L^2} \leq C\delta_N|D_N)\leq 1- e^{-bN\varepsilon_N^2} \big) \to 0.
            \end{equation*}
    \end{enumerate}
\end{theorem}
\begin{remark}
    We get as corollary that in fact
        \begin{equation*} 
            \Prob_{\theta_0}^N\big(\Pi_N(\theta \in \Theta :\|\mathcal G(\theta) - \mathcal G(\theta_0)\|_{L^2} \leq C\delta_N| D_N)\leq 1- e^{-bN\varepsilon_N^2} \big) \to 0.
        \end{equation*}
\end{remark}
\subsection{A small ball theorem}

Now we provide a theorem that shows some choice of posterior, rate, and regularization set allows us to satisfy the assumptions of Theorem \ref{thm:PosteriorGeneral}.
Recall that $\Theta \subseteq L^2(\mathcal Z, W)$ and $W$ is some finite-dimensional vector space.
For $\mathcal Z$ an $\ell$-dimension domain, the Sobolev spaces of $W$-valued vector fields are naturally defined, see \textsection \ref{sec:notation}. 

\begin{theorem}\label{thm:smallball}
    Suppose that Condition \ref{con:forwardregularity} and \ref{con:RKHS} hold for the forward map $\mathcal G$, regularity $\alpha >0$, some $\kappa \geq 0$, $U < \infty$, regularization space $\mathcal R$ and bounded domain $\mathcal Z \subseteq \R^{\ell}$ with smooth boundary (or $\mathcal Z  = \mathbb T^{\ell}$). Denote by $\Pi_N$ the law of the rescaled prior from \eqref{eq:priorrescaling}. Finally, suppose that either Condition \ref{con:inversematrix} or Condition \ref{con:inversematrix2} holds with the constant $C_2$. If $\theta_0 \in \mathcal H \cap \mathcal R$, then the hypotheses \eqref{eq:smallball} and \eqref{eq:regularizationset} in Theorem \ref{thm:PosteriorGeneral} hold for some $A>0$, for sequence
    \[
        \varepsilon_N = N^{-(\alpha + \kappa)/(2\alpha + 2\kappa + \ell)}, 
    \]
    and any $B$ if the regularization sets are chosen as
    \begin{equation}\label{eq:regset}
        \Theta_N = \{ \theta \in \mathcal R: \theta = \theta_1 + \theta_2, \|\theta_1\|_{(H^\kappa(\mathcal Z))^\ast} \leq M\varepsilon_N,  \|\theta_2\|_{\mathcal H} \leq M, \|\theta\|_{\mathcal R}\leq M\}, 
    \end{equation}
    for $M = M(B)$ sufficiently large.
\end{theorem}

\begin{proof}
    The proof is identical to Steps 1 to 3 of Theorem 2.2.2 in \cite{N}, except that by Proposition \ref{prop:inversematrix}, we have
    \[
        \mathcal B_N \supseteq \{ \theta : \|\mathcal G (\theta) - \mathcal G(\theta_0) \|_{L^\infty_t L^2_x} \leq c(U,C_2)\varepsilon_N, \|\mathcal G(\theta)\|_\infty \leq U \}.
    \]
\end{proof}

\section{Stability for parabolic PDEs with a divergence form operator}\label{sec:findingphi}

It this section, we will show initial conditions can be chosen so that Conditions \ref{con:Darcy} and \ref{con:Darcy2} are satisfied. This yields stability for the map $D\mapsto u_D$ for parabolic PDEs with elliptic operator $-\nabla \cdot(D\nabla u)$, where $D$ is unknown. We defer the proofs of the statements to the appendix \textsection \ref{sec:proofsforDarcy}.

\subsection*{Verifying Condition \ref{con:Darcy}}
\begin{remark}
    By $\nabla^2 \varphi$ we denote the Hessian. Equation \eqref{eq:movingcriticalpoints} translates to some non-degeneracy condition at every critical point of $\varphi$. The idea of the proof for Lemma \ref{lem:verifyingDarcy} is that \eqref{eq:movingcriticalpoints} force the critical points of the solution to move, see Appendix \ref{sec:proofsforDarcy}, \textsection \ref{sec:verifyingDarcy} for details. 

    Lemma \ref{lem:verifyingDarcy} provides an easier to work with condition instead of Condition \ref{con:Darcy} because it is easier to verify theoretically: one can find an arbitrarily small perturbation of $\varphi$ so that \eqref{eq:movingcriticalpoints} holds for some $\delta$. See Proposition \ref{prop:RDvarphi} for details. 
\end{remark}

\begin{lemma}\label{lem:verifyingDarcy}
    Let $D$ be the diffusivity and $\varphi$ be the initial condition in \eqref{eq:introRDE}, with $\|D\|_{H^{\alpha}} \leq M$ for $\alpha > d/2 + 5$. Suppose that 
    \begin{itemize}
        \item $\varphi$ has isolated critical points, and
        \item there is some constant $\delta$ such that for all $x^\ast\in \Omega$, 
            \begin{equation}\label{eq:movingcriticalpoints}
                \nabla \varphi(x^\ast) = 0 \implies |D\nabla (\Delta \varphi) + \Delta\varphi \nabla D + \nabla^2\varphi \nabla D| \geq \delta \text{ at } x^\ast.
            \end{equation}
    \end{itemize}
    Then, $u = u_D$ satisfies Condition \ref{con:Darcy} for some constant $c_0 = c_0(\delta, M, \alpha, D_{\min}, d, f)$.
\end{lemma}

\begin{proposition}\label{prop:RDvarphi}
    Suppose that $D: \Omega \to [D_{\min}, \infty)$ and $\|D\|_{C^1(\Omega)}\leq M$. Then, there exists a function $\varphi$ depending only on $d, M, D_{\min}$ such that $\varphi$ has isolated critical points and \eqref{eq:movingcriticalpoints} holds.
\end{proposition}
\begin{remark}
    See Appendix \ref{sec:proofsforDarcy}, \textsection \ref{sec:RDvarphi} for details. We leverage the existence of a single function to show that in fact satisfying Condition \ref{con:Darcy} is generic in the initial condition, which is the subject of Proposition \ref{prop:opendense}.
\end{remark}
\begin{proposition}\label{prop:opendense}
    Let $D$ be the diffusivity from \eqref{eq:introRDE}, with $D\in H^\alpha$ for $\alpha > d/2 + 5$. The set of initial conditions satisfying Condition \ref{con:Darcy} for some constant $c_0$ is of second category with respect to the $C^k$ norm for $k \geq 3$ (in particular it contains an open dense set).
\end{proposition}

\begin{remark}
    Having discussed when Condition \ref{con:Darcy} holds, it is worth mentioning notable times when \textit{doesn't} hold. Indeed, when there is a steady state (i.e. constant in time) solution $u$ of \eqref{eq:introRDE}, then the stationary points of $u$ trivially do not move.

    Another case to note is that of pure eigen-modes: if we consider the case when $f \equiv c$ for some constant $c$, and $e_\lambda\in L^2(\Omega)$ is an eigenfunction of $-\nabla \cdot(D\nabla )$ with eigenvalue $\lambda > 0$, then setting $\varphi = e_\lambda$, we see
    \begin{equation*}
        u(t, x) = e^{(-\lambda + c)t}e_\lambda(x),
    \end{equation*}
    and clearly $u$ does not have moving critical points. This suggests, as expected, that when the dynamics are degenerate we are unable to infer $D$. 
\end{remark}

\subsection*{Verifying Condition \ref{con:Darcy2}}

Throughout this section, fix $\Omega'$ to be an open subset of $\Omega$ such that $\overline {\Omega'}\subsetneq \Omega$. Then, we have 

\begin{proposition}\label{prop:NSvarphi}
    Suppose that $D: \Omega \to [D_{\min}, \infty)$ and $\|D\|_{C^1} \leq M$. Then, there exists a function $\varphi$ depending only on $d, M, D_{\min}, \Omega'$ such that Condition \ref{con:Darcy2} holds for $\Omega'$ and some choice of $c_0$, $\mu$ and $t = 0$. 
\end{proposition}

\begin{proof}
    Take any non-zero function $f$ such that $f:\Omega \to [0, 1]$, $f \in C^\infty_c(\overline {\Omega'}^c)$. Then, there is some constant $c(d, V) > 0$ such that 
    \[
        \int (1 - cf) \, dx = 0,
    \]
    and set $g = 1 - cf$. Then $g \in \dot C^\infty(\Omega)$, so by Proposition \ref{prop:EigenSobolevChar} there is a solution $\varphi \in C^\infty(\Omega)$ such that
    \[
        g = \nabla \cdot (D\nabla \varphi).
    \]
    Now we can use an idea from \cite{NvdGW} Lemma 24, and note that for $x\in \Omega'$, 
    \[
        1 = g(x) = D(x)\Delta \varphi(x) + \nabla D(x) \cdot \nabla \varphi(x),
    \]
    so either $\Delta \varphi(x) \geq 1/(2\|D\|_\infty)$ or $|\nabla \varphi(x)|^2 \geq 1/(2\|D\|_{C^1})^2$. Thus, we can choose $\mu = \mu(M, d, V)$ so that
    \[
        \inf_{x \in \Omega'} \left[\frac{1}{2}\Delta \varphi (x) + \mu |\nabla \varphi (x)|^2\right]\geq c_0,
    \]
    where $c_0$ depends on $M, d, \Omega', D_{\min}$. Taking this $\varphi$ to be the initial condition of $u_D$ gives Condition \ref{con:Darcy2}.
\end{proof}
\section{Proofs}\label{sec:Proofs}
\subsection{Proofs for \textsection \ref{sec:introduction}}
\subsubsection{Proof of Proposition \ref{prop:VARCheck}}
\begin{proof}[Proof of Proposition \ref{prop:VARCheck}]
    From Remark \ref{rmk:simplifyingConds}, it suffices to show that Condition \ref{con:inversematrix} holds for a general $VAR(p)$ process with a constant depending only on an upper bound of its parameters, $p$, and $d_V$. Now, if $X = (X_1, \dots X_T) \sim N(0, \Sigma)$ is a $VAR(p)$ process, then we let $Q = \Sigma^{-1}$ and $Q_{ij}$ to be $d_V\times d_V$ sub-matrices of $Q$ for $i, j = 1, \dots T$.
    But we know from \cite{L} that the precision matrix of a stationary $VAR(p)$ process is a $p-$banded block matrix, which is Toeplitz away from the extremes:
    i.e. $Q_{ij} = 0$ for $|i-j| > p$, and $Q_{ij} = Q_{i'j'}$ when $i - j = i'-j'$ and $p< i, i',j, j' < T - p+1$. The upper bound follows from the explicit form of the matrix derived in \cite{L}, equation (3.4.4).
\end{proof}
\subsubsection{Proof of Proposition \ref{prop:inversematrix}}\label{sec:KLdivandvar}    
The goal of this section is to prove Proposition \ref{prop:inversematrix} to show the small ball condition \eqref{eq:smallball} of Theorem \ref{thm:PosteriorGeneral} holds. It is necessary in the proof of Theorem \ref{thm:smallball}.

\begin{proof}[Proof of Proposition \ref{prop:inversematrix}]
For ease of notation, let $u = \mathcal G(\theta)$ and $v = \mathcal G(\vartheta)$. We will additionally write $p_\theta^N$ as $p_\theta$. Note that 
\[
    p_\theta(y, x) \propto \exp\left( -\frac{1}{2} \sum_{ij} (y_i - u(t_i, x_i))^T(\Gamma_N^{-1})_{ij}(y_j - u(t_j, x_j)) \right),
\]
where for $i, j = 1, \dots, N$, we will denote by $(\Gamma_N^{-1})_{ij}$ as the $d_V \times d_V$ submatrices of $\Gamma_N$. Note that $(\Gamma_N^{-1})_{ij}^T = (\Gamma_N^{-1})_{ji}$.
Hence, after some algebraic manipulation, the log-liklihood is 
\begin{align}
    \log(\frac{p_\vartheta}{p_{\theta}}) &= - \sum_{ij} \epsilon_i^T (\Gamma_N^{-1})_{ij} (u(t_j, X_j) - v(t_j, X_j)) \label{eq:logliklihoodexpansion}\\
                              & \hspace{10mm}- \frac{1}{2}\sum_{ij} (u(t_i, X_i) - v(t_i, X_i))^T(\Gamma_N^{-1})_{ij} (u(t_j, X_j) - v(t_j, X_j)),\nonumber
\end{align}
under the distribution of $p_\theta$.
Now, let $u(t_i, X_i) - v(t_i, X_i) \equiv u_i - v_i$. Then, we observe from the trace trick and $\mathbb E_\theta [\epsilon_i] = 0$ that the KL divergence satisfies
\begin{align}
    -\mathbb E_{p_{\theta}} [\log \frac{p_\vartheta}{p_{\theta}}] = \frac{1}{2}\sum_{ij} \tr\big[(\Gamma_N^{-1})_{ij} \E[(u_j - v_j)(u_i-v_i)^T]\big].\label{eq:KLExpression}
\end{align}
In the case of Condition \ref{con:inversematrix}, assume also for simplicity that $d_V = 1$. Then, use Cauchy-Schwarz to get 
\begin{equation*}
    |KL(p_\theta, p_\vartheta)| \leq \frac{1}{2}\sum_{ij} |(\Gamma_N^{-1})_{ij}| \|u(t_i, \cdot) - v(t_i, \cdot) \|_{L^2(\Omega)}\|u(t_j, \cdot) - v(t_j, \cdot) \|_{L^2(\Omega)}.
\end{equation*}
Then, we apply Condition \ref{con:inversematrix} and obtain 
\begin{equation}\label{eq:KLBound}
    KL(p_{\theta}, p_\vartheta) \lesssim N \|u - v\|_{L^\infty_tL^2_x}^2.
\end{equation}

Now, we move on to the more involved bound (still assuming $d_V = 1$) on the Kullback-Leibler variation $V_{2, 0}$.
\[
    V_{2, 0}(p_{\theta}, p_\vartheta) = \mathbb E_{p_{\theta}} [\log \frac{p_\vartheta}{p_{\theta}}]^2 - KL(p_{\theta}, p_\vartheta)^2
\]
If we square and expand \eqref{eq:logliklihoodexpansion}, noting $\E_{p_\theta}[\epsilon_{i}\epsilon_{i'}] = (\Gamma_N)_{ii'}$ and $\E_{p_\theta}[\epsilon_i] = 0$, we get
\begin{align}
&= \sum_{jj'} (\Gamma_N^{-1})_{jj'}\mathbb E_{p_{\theta}}[(u- v)(t_j, X_j)(u- v)(t_{j'}, X_{j'})]\nonumber \\
& \quad +\frac{1}{4}\sum_{ii'jj'}(\Gamma_N^{-1})_{ij}(\Gamma_N^{-1})_{i'j'}\mathbb E_{p_{\theta}}\Big[(u- v)(t_{i'}, X_{i'})(u- v)(t_{j'}, X_{j'})(u- v)(t_i, X_i)(u- v)(t_j, X_j)\Big] \nonumber \\
&\quad -\frac{1}{4}\sum_{ii'jj'}(\Gamma_N^{-1})_{ij}(\Gamma_N^{-1})_{i'j'}\mathbb E_{p_{\theta}}[(u- v)(t_{i'}, X_{i'})(u- v)(t_{j'}, X_{j'})]\mathbb E_{p_{\theta}} [(u- v)(t_i, X_i)(u- v)(t_j, X_j)]\label{eq:KLvariationeq}
\end{align}
We can bound the first term like the KL divergence case. The second two terms cancel out whenever $X_{i'}, X_{j'}$ are independent of $X_i, X_j$, i.e. when $\{i, j\} \cap \{i', j'\} = \emptyset$. When this is not the case, at least one of the following is true: $i = i'$, $i = j'$, $j = i'$, or $j = j'$. Hence, we can bound the above sum by (using H\"older's inequality)
\begin{align*}
&\leq \frac{1}{2}\sum_{ijj'}(\Gamma_N^{-1})_{ij}(\Gamma_N^{-1})_{ij'}\|(u-v)(t_i, \cdot)\|_{L^4_x}^2\|(u-v)(t_j, \cdot)\|_{L^4_x}\|(u-v)(t_{j'}, \cdot)\|_{L^4_x}\\
&\quad + \frac{1}{2}\sum_{ii'j}(\Gamma_N^{-1})_{ij}(\Gamma_N^{-1})_{i'i}\|(u-v)(t_i, \cdot)\|_{L^4_x}^2\|(u-v)(t_{i'}, \cdot)\|_{L^4_x}\|(u-v)(t_j, \cdot)\|_{L^4_x}\\
&\quad + \frac{1}{2}\sum_{ijj'}(\Gamma_N^{-1})_{ij}(\Gamma_N^{-1})_{jj'}\|(u-v)(t_i, \cdot)\|_{L^4_x}\|(u-v)(t_j, \cdot)\|_{L^4_x}^2\|(u-v)(t_{j'}, \cdot)\|_{L^4_x}\\
&\quad + \frac{1}{2}\sum_{ii'j}(\Gamma_N^{-1})_{ij}(\Gamma_N^{-1})_{i'j}\|(u-v)(t_i, \cdot)\|_{L^4_x}\|(u-v)(t_{i'}, \cdot)\|_{L^4_x}\|(u-v)(t_j, \cdot)\|_{L^4_x}^2
\end{align*}
Now, using Condition \ref{con:inversematrix}, we observe finally that
\begin{equation}\label{eq:KLvariationBound}
    V_{2, 0}(p_{\theta}, p_\vartheta) \lesssim N (\|u- v \|_{L_t^\infty L_x^4}^2 + \|u - v \|_{L_t^\infty L_x^4}^4).
\end{equation}
Consider the case of $(a)$ in Condition \ref{con:inversematrix2} (and no longer assume $d_V = 1$). Going back to \eqref{eq:KLExpression}, we see
\[
    KL(p_\theta, p_\vartheta) = \frac{1}{2}\sum_{ij} \tr \big[(\Gamma_N^{-1})_{ij}\E[(u_j - v_j)(u_i - v_i)^T]\big].
\]
But by Condition \ref{con:inversematrix2}.a, we know for $i \neq j$ that 
\[
    \E[(u_i - v_i)(u_j - v_j)^T] = \E[(u-v)(t_i, X_i)]\E[(u-v)(t_j, X_j)^T] = 0.
\]
So, plugging this into \eqref{eq:KLExpression},
\begin{equation}\label{eq:KLtrickery}
\begin{split}
    KL(p_\theta, p_\vartheta) &= \frac{1}{2} \sum_{i} \tr\big[(\Gamma_N^{-1})_{ii}\E[(u_i-v_i)(u_i-v_i)^T]\big]\\
    &\leq \frac{1}{2} \sum_{i}\tr\big[(\Gamma_N^{-1})_{ii}\big] \|u-v\|^2_{L^\infty_tL^2_x} \leq \frac{1}{2}\tr(\Gamma_N^{-1})\|u-v\|^2_{L^\infty_tL^2_x},
\end{split}
\end{equation}
where in the first inequality we have used the facts that 1) for two non-negative symmetric matrices $M$ and $L$, $\tr(ML)\leq \tr(M) \|L\|$, and 2) for $w\in V$, if we set $L = ww^T$, then $L$ is non-negative symmetric matrix and $\|L\| = |w|_V^2$. For the KL-variation, we derive an expression akin to \eqref{eq:KLvariationeq} for a general $d_V$,
\begin{align*}
    V_{2, 0}(p_\theta, p_\vartheta) &= \sum_{jj'}\tr\big[(\Gamma_N^{-1})_{jj'}\E [(u_{j'}-v_{j'})(u_{j} - v_{j})^T]\big]\\
    +\frac{1}{4}\E\bigg[ &\sum_{ii'jj'} (u_{i}-v_{i})^T(\Gamma_N^{-1})_{ij}(u_{j}-v_{j})(u_{i'}-v_{i'})^T(\Gamma_N^{-1})_{i'j'}(u_{j'}-v_{j'})\bigg]\\
    -\frac{1}{4}\sum_{ii'jj'}&\E[(u_{i}-v_{i})^T(\Gamma_N^{-1})_{ij}(u_{j}-v_{j})]\E[(u_{i'}-v_{i'})^T(\Gamma_N^{-1})_{i'j'}(u_{j'}-v_{j'})]
\end{align*}
but this time we note that second two terms cancel out whenever: a) $\{i, j\}\cap\{i', j'\} = \emptyset$, or b) there is one index distinct from the other three.
Thus, we get
\begin{align*}
    V_{2, 0}(p_\theta, p_\vartheta) &= \sum_{j}\tr\big[(\Gamma_N^{-1})_{jj}\E[(u_j-v_j)(u_j-v_j)^T] \big]\\
                              &+\frac{1}{4}\sum_{ij}\E[(u_i-v_i)^T (\Gamma_N^{-1})_{ij}(u_j-v_j)(u_i-v_i)^T(\Gamma_N^{-1})_{ij}(u_j-v_j)]\\
                              &\qquad\qquad - \E[(u_i-v_i)^T(\Gamma_N^{-1})_{ij}(u_j - v_j)]^2\\
                              &+\frac{1}{4}\sum_{ij, i\neq j}\E[(u_i-v_i)^T (\Gamma_N^{-1})_{ij}(u_j-v_j)(u_j-v_j)^T(\Gamma_N^{-1})_{ji}(u_i-v_i)]\\
                              &\qquad\qquad - \E[(u_i-v_i)^T(\Gamma_N^{-1})_{ij}(u_j - v_j)]^2.
\end{align*}
The first term is handled the same as in the KL-divergence case. For the other two terms, we consider the bound, for $i, j = 1, \dots N$, 
\begin{align*}
    &\E[(u_i-v_i)^T (\Gamma_N^{-1})_{ij}(u_j-v_j)(u_i-v_i)^T(\Gamma_N^{-1})_{ij}(u_j-v_j)] - \E[(u_i-v_i)^T(\Gamma_N^{-1})_{ij}(u_j - v_j)]^2\\
    &\quad \leq \E\big[\tr[(\Gamma_N^{-1})_{ji}(u_i - v_i)(u_i-v_i)^T (\Gamma_N^{-1})_{ij}(u_j - v_j)(u_j - v_j)^T]\big]\\
    &\quad \leq \E\big[\tr[(\Gamma_N^{-1})_{ji}(u_i - v_i)(u_i-v_i)^T (\Gamma_N^{-1})_{ij}]|u_i - v_i|_V^2\big]\\
    &\quad \leq \tr[(\Gamma_N^{-1})_{ij}(\Gamma_N^{-1})_{ji}]\E[|u_i-v_i|_V^2|u_j - v_j|_V^2]\\
    &\quad \lesssim\tr[(\Gamma_N^{-1})_{ij}(\Gamma_N^{-1})_{ji}]\|u-v\|_{L^\infty_tL^4_x}^4,
\end{align*}
where we have again used the facts noted after \eqref{eq:KLtrickery}.
Now, we observe via block matarix algebra that
\begin{equation*}
\sum_{ij} \tr[(\Gamma_N^{-1})_{ij}(\Gamma_N^{-1})_{ji}] = \tr(\Gamma_N^{-2}).
\end{equation*}
Putting this together, we obtain
\begin{equation*}
    V_{2, 0}(p_\theta, p_\vartheta) \lesssim \tr(\Gamma_N^{-2})\|\mathcal G(\theta) - \mathcal G(\vartheta)\|_{L^\infty_tL^4_x}^4+\tr(\Gamma_N^{-1})\|\mathcal G(\theta) - \mathcal G(\vartheta)\|_{L^\infty_t, L^2_x}^2.
\end{equation*}
Now, we conclude by noting from the Cauchy-Schwarz inequality, $\tr(\Gamma_N^{-1}) \leq \tr(\Gamma_N^{-2})^{1/2}N^{1/2}$,
so by condition $(b)$ in Condition \ref{con:inversematrix2}, and $\|\mathcal G(\theta)\|_\infty, \|\mathcal G(\vartheta)\|_\infty \leq U$, we are done.
\end{proof}

\subsection{Proofs for \textsection \ref{sec:theory}}\label{sec:ProofsforTheory}
\subsubsection{Proof of Proposition \ref{prop:longrange}}
\begin{proof}[Proof of Proposition \ref{prop:longrange}]
    For simplicity, we specialize to the case $d_V = 1$. If for sufficiently large $T$, $\max_j \|M_j\|_{\infty \to \infty} \leq c$, we have that 
    \begin{align*}
        (\Gamma^{(j)})^{-1} &= (D_j I + M)^{-1} =  (I - D_j^{-1}M + (D_j^{-1}M)^2 - \cdots) D_j^{-1}
    \end{align*}
    since $\|D_j^{-1}M_j\|_{\infty\to\infty} < 1$. Thus, taking norms of both sides, we get
    \[
        \|(\Gamma^{(j)})^{-1}\|_{\infty\to\infty} \leq c_1^{-1} \frac{1}{1-c_1^{-1}\|M\|_{\infty\to\infty}} \leq \frac{1}{c_1 - c}
    \]
    but as mentioned in the discussion after Condition \ref{con:inversematrix}, we are done. The ``in particular'' part follows from the inequalities (see \cite{HJ} Theorem 5.6.26):
    \[
        \| A \|_{\infty \to \infty} \leq \sqrt T \|A\| \quad \text{ and } \quad \|A\|_{\infty \to \infty} \leq T \|A\|_\infty.
    \]
\end{proof}
\subsubsection{An $M$-estimation result}\label{sec:MEstResult}
We are in the setting of \cite{vdG2001}, but with dependent errors: Consider the regression model
\[
    Y_i = g_0(z_i) + \epsilon_i, \text{ for }i = 1, \dots, N
\]
where $Y_1, \dots, Y_N$ are real-valued observations, $z_1, \dots, z_N$ are explanatory variables, with values in some space $\mathcal Z$, the measurement errors $\epsilon_i$ are centered Gaussian random variables, with invertible covariance matrix $\Gamma = \Gamma_N$. We assume that $g_0$ belongs to some class of regression functions with domain $\mathcal Z$, call it $\mathbb G$. Moreover, we assume that for some constant $M$, $\|\Gamma_N\|\leq M$.

As in \cite{vdG2001}, we let $\pen : \mathbb G \to [0, \infty)$, $g_\ast\in \mathbb G$, and define
\begin{align*}
    \|g -h \|_{(N)}^2 &= \frac{1}{N}\sum |g(z_i) - h(z_i)|^2\\
    \tau^2(g|\Tilde g) &= \| g - \Tilde g\|^2_{(N)} + \pen^2(g), \\
    \mathbb G_\ast(\delta) &= \{ g \in \mathbb G : \tau^2(g | g_\ast) \leq \delta^2 \},\\
    J_\ast (\delta) &= \int_0^\delta \big[\log N(u, \mathbb G_\ast(\delta))\big]^{1/2}\, du,
\end{align*}
where $N(u, \mathbb G_\ast(\delta))$ is the minimal number of ball of radius $u$ to cover $\mathbb G_\ast(\delta)$ in the $\|\cdot\|_{(N)}$ metric. Finally, let $\hat g_N$ be any penalized least-squares minimizer. 
\begin{theorem}\label{thm:Mestimation2}
    Let $\Lambda_\ast(\delta)\geq J_\ast(\delta)\lor \delta$ be some upper bound. Suppose that $\Lambda_\ast(\delta)/\delta^2$ decreases as $\delta$ increases. There exists a constant $c$ depending only on $M$, such that for
    \[
        \sqrt N \delta_N^2 \geq c \Lambda_\ast(\delta_N), 
    \]
    we have for all $\delta \geq \delta_N$, 
    \[
        \Prob\big(\tau^2(\hat g_N|g_0) \geq 2(\tau^2(g_\ast|g_0)+\delta^2) | z_1, \dots z_N\big) \leq c\exp[-\frac{N\delta^2}{c^2}].
    \]
\end{theorem}
\begin{proof}

    The proof is the same as Theorem 2.1 in \cite{vdG2001}, except that instead of using Corollary 8.4 in \cite{vdG2000}, we use Lemma 3.2 in the same reference. To satisfy the hypothesis of Lemma 3.2, note that $\epsilon = (\epsilon_i)_{i}$ is an $N(0, \Gamma)$ random variable, and hence for $\gamma \in \R^N$, $\gamma^T\epsilon \sim N(0, \gamma^T\Gamma\gamma)$, so by the standard tail bound for a Gaussian, 
    \[
        \Prob\left(\sum_{i = 1}^N \epsilon_i\gamma_i \geq a\right) \leq \exp\left(-\frac{a^2}{2\gamma^T\Gamma\gamma}\right) \leq \exp\left(-\frac{a^2}{2\|\Gamma\||\gamma|^2}\right).
    \]
    Then the rest of the proof continues as normal.
\end{proof}

\subsubsection{Proof of Proposition \ref{prop:Mestimation}}
\begin{proof}[Proof of Proposition \ref{prop:Mestimation}]
    The proof is similar to that of Proposition B.7 in \cite{NW24}. We use a slightly modified version of Theorem 2.1 in \cite{vdG2001}, see \textsection \ref{sec:MEstResult} for details. We note that the empirical norm $\| \cdot\|_{(N)}$ is bounded by $\| \cdot \|_\infty$, and is further bounded by the metric entropy of a ball of radius $m$ in $H^{a, a/2}$, which by Corollary \ref{cor:metricentropy} has metric entropy of order $H(\tau) \lesssim (m/\tau)^{(d+2)/a}$ for every $m >0$. To control the bias term, we first observe that for $\theta\in \Theta_N$, we have $\theta = \theta_1 + \theta_2$, with $\|\theta_1\|_{(H^{\kappa}(\mathcal Z))^\ast}\leq M\varepsilon_N$ and $\|\theta_2\|_{\mathcal H} \leq M$, $\|\theta\|_{\mathcal R}\leq M$ for some $M = M(B)$.
    By Condition \ref{con:forwardmapping}, we have that there is some $U = U(B)$ such that $\|u_\theta\|_{H^{6, 3}} \leq U$ and $\|u_{\theta_2}\|_{H^{6, 3}} \leq U$ and $\|u_{\theta_2}\|_{H^{a, a/2}} \leq U$. The bias term is then
    \[
        \|u_\theta - u_{\theta_2, L}\|_{(N)} \leq \|u_\theta - u_{\theta_2}\|_{(N)} + \|u_{\theta_2} - u_{\theta_2, L}\|_{(N)}.
    \]
    We next make use Lemma \ref{lem:deterministicMconcentration} with $v = u_\theta - u_{\theta_2}$ and $u_{\theta_2} - u_{\theta_2, L}$ to specialize to the event 
    \[
    A_N = \big\{ |\|v\|_{(N)}^2 - \|v\|_{L^2}^2| \leq \lambda \delta_N^2 \text{ for } v = u_\theta - u_{\theta_2}, u_{\theta_2} - u_{\theta_2, L}\big\},
    \]
   which has $\Prob_\theta^N(A_N) \geq 1 - \exp(-KN\delta_N^2)$ for $\lambda$ large enough. (Note that the event $A_N$ depends only on the design $(t_i, X_{ij})_{ij}$.) Then, we use \eqref{eq:biascontrol2} to see on $A_N$,
    \[
        \|u_\theta - u_{\theta_2, L} \|_{(N)} \leq \|u_\theta - u_{\theta_2}\|_{L^2} + C\delta_N,
    \]
    but now we use Condition \ref{con:forwardregularity} and $\|\theta_1\|_{(H^\kappa(\mathcal Z))^\ast} \leq M\varepsilon_N$ to get that 
    \[
        \|u_\theta - u_{\theta_2, L} \|_{(N)} \leq C(\varepsilon_N + \delta_N).
    \]
We finish by applying Theorem \ref{thm:Mestimation2} on all the designs $(t_i, X_{ij})_{ij}$ belonging to $A_N$.
\end{proof}

\begin{lemma}\label{lem:deterministicMconcentration}
    In the setting of \eqref{eq:timedependent} and \eqref{eq:setting}, suppose that $v = v_N \in H^{6, 3}\cap L^\infty$ is a (deterministic) sequence with $\|v_N\|_{H^{6, 3}}\leq U$ and $\|v_N\|_\infty\leq U$ for all $N$, and suppose that 
    \[
        n\lesssim N\delta_N\text{ and } T \gtrsim \delta_N^{-1},
    \]
    as well as $\|v_N\|_{L^\infty_tL^2_x} \lesssim \delta_N$.
    Then for any $K > 0$, there is a $\lambda(U)  > 0$ so that 
    \[
        \Prob^N_\theta\big( |\|v \|_{(N)}^2 - \|v \|_{L^2}^2 | \leq \lambda\delta_N^2\big) \geq 1 - e^{-KN\delta_N^2}.
    \]
\end{lemma}
\begin{proof}
    First, consider the collection of independent, mean-zero random variables 
    \[
        |v(t_i, X_{ij})|^2  - \|v(t_i, \cdot)\|^2_{L^2(\Omega)} \quad \text{ for }i = 1, \dots, T, j = 1, \dots, n
    \]
    Since $\|v\|_\infty \leq U$, this is a collection of independent random variables which are uniformly bounded and have mean zero. Also, we have that 
    \begin{align*}
        \Var(|v(t_i, X_{ij})|^2) &\leq \E^N(|v(t_i, X_{ij})|^4) \leq U^2\E^N(|v(t_i, X_{ij})|^2) \leq U^2\delta_N^2.
    \end{align*}
    Thus, applying Bernstein's inequality, we obtain
    \begin{align*}
        \Prob^N_\theta \left(\left|\frac{1}{N}\sum_{i, j} |v(t_i, X_{ij})|^2 -\frac{1}{T}\sum_{i=1}^T \|v(t_i, \cdot) \|^2_{L^2(\Omega)} \right| \geq \lambda \delta_N^2\right) &\leq 2 \exp\left(\frac{-N^2\delta_N^4\lambda^2}{2(NU^2\delta_N^2 + \frac{1}{3}U^2N\delta_N^2\lambda)}\right)\\
        &\leq 2\exp\left( - \frac{\lambda^2}{2U^2(1 + \lambda/3)} N\delta_N^2\right).
    \end{align*}
    Hence, taking $\lambda \geq 0$ large enough, we get
    \begin{align*}
        \Prob^N_\theta\left(\left|  \|v\|^2_{(N)} - \frac{1}{T}\sum_{i = 1}^T\|v(t_i, \cdot) \|^2_{L^2(\Omega)} \right| \leq \lambda \delta_N^2 \right) \geq 1 - \exp\left( -K N\delta_N^2\right).
    \end{align*}
    Now, we know that $\|v\|_{H^{6, 3}} \leq U$. Also, by Proposition \ref{prop:besovIsSobolev}, $v \in H^3([0, 1], L^2(\Omega))$. Then, from \textsection 5.9.2 Theorem 3 in \cite{Evans}, this means the function
    \[
        t \mapsto \|v(t, \cdot)\|_{L^2(\Omega)}^2
    \]
    is $C^2([0, 1])$ with its norm depending only on $U$. Now, consider
    \begin{lemma}\label{lem:approximationtheory}
       Now, we know that for $g\in C^2([0, 1])$, and for $t_i = (2i-1)/2T$, $i = 1, \dots T$,  we have
       \[
           \left|\frac{1}{T}\sum_{i = 1}^T g(t_i) - \int_0^1 g(t) \, dt \right| \leq \frac{ \| g \|_{C^2}}{24T^2}.
       \]
   \end{lemma}
   See \cite{MS}, p.286. We use the fact $T \gtrsim \delta_N^{-1}$ to finish.
\end{proof}

\subsubsection{Proof of Proposition \ref{prop:Mconcentration}}\label{sec:MconcentrationSupp}
First, we will show that for a fixed $i$, the empirical norm w.r.t. $\{X_{ij}: j = 1, \dots n\}$ concentrates around the $L^2(\Omega)$ norm.
\begin{lemma}\label{lem:spatialconcentration}
    With the setup from Proposition \ref{prop:Mestimation}, assume that $a > (d+1)/2$, and suppose that
    \[
        n \gtrsim N^{\frac{2d+3}{2a + d + 2}} \quad \text{ and }\quad T \lesssim N^{\frac{2a - d - 1}{2a + d +2}}.
    \]
    Then there exists a sequence $\gamma_N$ so that $N\delta_N^2 = o(\gamma_N)$ and for every $i = 1, \dots T$,
    \[
        \Prob^N_\theta\left(\left| \frac{\frac{1}{n}\sum_{j =1}^n |(\hat u - u_{\theta, L})(t_i, X_{ij})|^2}{\|(\hat u - u_{\theta, L})(t_i, \cdot ) \|^2_{L^2(\Omega)}} - 1\right|\leq \frac{1}{2}\right) \geq 1 - c'e^{-\gamma_N},
    \]
    where $u_{\theta, L}$ is as in \eqref{eq:projectionoftruth}, and $c$ is some positive constant.
\end{lemma}
\begin{proof}
    We know that both $\hat u (t_i, \cdot)$ and $u_{\theta, L}(t_i, \cdot)$ are in $\mathbf V_{L+1}$ (because every spatial wavelet in the sum for fixed $t$ has level $\leq L$)\footnote{See \textsection \ref{sec:WaveletFacts} for details}. Hence, we can apply the argument from step IV in Proposition B.7 in \cite{NW24}, for $\Gamma = I$, and $\hat \Gamma$ being
    \[
        \hat \Gamma_{kk'} = \frac{1}{n}\sum_{j = 1}^n \Phi_{Lk}(X_{ij})\Phi_{Lk'}(X_{ij}),
    \]
    where $\Phi^{\iota}_{Lk}$ are tensorized Meyer scaling functions, and we note that $\hat \Gamma$ is a $2^{(L+1)d} \times 2^{(L+1)d}$ matrix. Then, we let $M_L$ be the covering number of a ball of radius $\kappa$ for some fixed $\kappa^2 + 2\kappa < 1/4$ in $\R^{2^{(L+1)d}}$, with $\{ v_1, \dots, v_{M_L}\}$ the covering vectors, and we define $U \equiv c N^{d/(2a + d +2)} $. Then, arguing as in \cite{NW24}, we deduce
    \[
        M_J \lesssim e^{c_0 2^{Ld}} \leq e^{c_1 N^{d/(2a + d + 2)}} 
    \]
    where the constants are absolute. 
    Finally, in the application of Bernstein's inequality, we get
    \[
        \Prob(n|v_m^T(\hat \Gamma - I)v_m| > n/4) \leq 2 \exp\left(-\frac{n^2/16}{2nU + (2/12)nU}\right) \leq 2 e^{-c_0n/U}
    \]
    Now, we now that $n/U\gtrsim N^{\frac{d+3}{2a + d + 2}},$ and hence we can bound
    \begin{align*}
        \Prob \left( \max_{m = 1, \dots, M_J} |v^T_m (\hat \Gamma - I)v_m| \geq 1/4\right) &\leq M_J \max_m \Prob ( |v_m^T(\hat\Gamma - I)v_m| > 1/4) \\
        & \leq 2e^{c_1 N^{d/(2a + d +2)} - c_0N^{(d+3)/(2a + d + 2)}} \\
        &\leq c'e^{-c'' N^{(d+2.5)/(2a + d +2)}} = c'e^{-\gamma_N}.
    \end{align*}
    This probability bounds the complement of the probability in question, so we are done.
\end{proof}
\begin{proof}[Proof of Proposition \ref{prop:Mconcentration}]
    We observe from Lemma \ref{lem:spatialconcentration} and a union bound that 
    \[
        \Prob^N_\theta\left(\left| \frac{\frac{1}{n}\sum_{j =1}^n |(\hat u - u_{\theta, L})(t_i, X_{ij})|^2}{\|(\hat u - u_{\theta, L})(t_i, \cdot ) \|^2_{L^2(\Omega)}} - 1\right|\leq \frac{1}{2} \text{ for }i = 1, \dots T\right) \geq 1 - Tc'e^{-\gamma_N}.
    \]
    But since $T$ grows algebraically in $N$, and $N\delta_N^2 = o(\gamma_N)$, for sufficiently large $N \geq N_0$ we have after averaging over the $t_i$,
    \begin{align}
        \Prob^N_\theta\bigg(\frac{1}{T}\sum_{i=1}^T\|(\hat u - u_{\theta, L})(t_i, \cdot ) \|^2_{L^2(\Omega)} \leq \frac{2}{Tn}\sum_{i = 1}^T&\sum_{j =1}^n |(\hat u - u_{\theta, L})(t_i, X_{ij})|^2 \bigg) \geq 1 - \frac{1}{2}e^{-KN\delta_N^2}.\label{eq:timeaveragebound}
    \end{align}
    Next, from Proposition \ref{prop:Mestimation} by taking $c$ large enough, we have 
    \[
        \Prob^N_\theta( \|\hat \beta \|_{h^a}^2 \leq c | ((t_i, X_{ij}))_{ij}) \geq 1 - \frac{1}{2}e^{-KN\delta_N^2}.
    \]
    But we've defined $h^a$ so that $\| \hat \beta \|_{h^a} \asymp \| \hat u \|_{H^{a, a/2}}$, and since $a \geq 6$, we have $\|\hat u\|_{H^{6, 3}} \leq \|\hat u\|_{H^{a, a/2}}$.
    We know that additionally $u_{\theta, L} \in H^{6, 3}$ with $\| u_{\theta, L} \|_{H^{6, 3}}\leq U = U(B)$, so we know that
    \begin{equation}\label{eq:hatinsobolevball}
        \Prob^N_\theta( \| \hat u - u_{\theta, L} \|^2_{H^{6, 3}} \leq 2c + 2U^2) \geq 1 - \frac{1}{2}e^{-KN\delta_N^2}.
    \end{equation}
    On this event, from the Sobolev characterization of $H^{6, 3}$, we get that $\hat u - u_{\theta, L} \in H^{3}([0, 1], L^2(\Omega))$, with $\|\hat u - u_{\theta, L}\|_{H^3_tL^2_x} \leq c_0$, where $c_0$ depends on $u_\theta$ only through $U$. Then, from \textsection 5.9.2 Theorem 3 in Evans, this means the function 
    \[
        t \mapsto \| \hat u - u_{\theta, L}(t, \cdot) \|_{L^2(\Omega)}^2
    \]
    is $C^2([0, 1])$ with its norm depending only on $c$ and $U$. Now, consider Lemma \ref{lem:approximationtheory} but with $g = \hat u - u_{\theta - L}$. Then, combining both \eqref{eq:timeaveragebound} and \eqref{eq:hatinsobolevball}, we get
    \[
        \Prob^N_\theta(\| \hat u - u_{\theta, L} \|_{L^2([0, 1]\times \Omega)}^2 \leq 2 \| \hat u - u_{\theta, L}\|_{(N)}^2 + \frac{c}{T^2}) \geq 1 -  e^{-K N\delta_N^2},
    \]
    as desired.
\end{proof}

\subsection{Proofs for \textsection \ref{sec:applications}}\label{sec:ProofsforApplications}
\subsubsection{Proof of Lemma \ref{lem:levelsetmap}}\label{sec:levelsetmap}
\begin{proof}[Proof of Lemma \ref{lem:levelsetmap}]
    The map $DG$ we have defined is continuous and linear for any $(D, u) \in C^1(\Omega) \times F$.
    We can write
    \begin{align*}
        G(D+h, u+v) - G(D, u) &- DG(D, u)[h, v] \\
        &= \left(-\nabla \cdot(h\nabla v) - (f(u+v) - f(u) - f'(u)v), 0 \right),
    \end{align*}
    and we simply need so show the RHS is $o(\|h\|_{C^1} + \|v\|_F)$. Now, we have that
    \begin{align*}
        \|\nabla \cdot (h \nabla v ) \|_{L^2(Q)}&\leq \| \nabla h \cdot \nabla v \|_{L^2} + \|h\Delta v\|_{L^2} \\
                                                         &\leq d \|h\|_{C^1}\|v\|_{L^2_tH^1_x} + d\|h\|_\infty \|v\|_{L^2_tH^2_x}\\
                                                         &\leq 2d\|h\|_{C^1}\|v\|_F \leq d(\|h\|_{C^1} + \|v\|_F)^2,
    \end{align*}
    which bounds the first term in the sum. Next, note that from Taylor's theorem applied to $f$, for every $(t, x) \in Q$, there is some value $\xi( u, v) = \xi(u(t, x), v(t, x))$ such that
    \[
        f(u+v) = f(u) + f'(u)v + \frac{1}{2}f''(\xi(u, v))v^2.
    \]
    Hence, we have (also using the fact that $H^1(\Omega) \hookrightarrow L^4(\Omega)$ for $d\leq 4$):
    \begin{align*}
        \|f(u+v) - f(u) - f'(u)v\|_{L^2(Q)}^2 &= \int_{Q} \frac{1}{4} f''(\xi)^2 v^4\, dtdx \leq \frac{1}{4}\|f''\|_\infty^2 \int_0^1 \|v\|_{L^4}^4 \, dt\\
                                                       &\leq C \int_0^1 \|v\|_{H^1}^4\, dt \leq C \|v\|_{C([0, 1], H^1(\Omega))}^4.
    \end{align*}
    Thus, since $F \hookrightarrow C([0, 1], H^1(\Omega))$, we have
    \[
        \|f(u+v) - f(u) - f'(u)v\|_{L^2(Q)} \leq C\|v\|F^2 = o(\|h\|_{C^1} + \|v\|_F).
    \]
\end{proof}
\subsubsection{Proof of Proposition \ref{prop:HRDEForward}}
\begin{proof}[Proof of Proposition \ref{prop:HRDEForward}]
    The first part follows from the remark above Proposition \ref{prop:HRDEForward}. Next, by the fundamental theorem of calculus for Fr\'echet derivatives in Banach spaces (see e.g. \cite{LR} Proposition A.2.3) we can write, for $D_1, D_2 \in U$, and $D_w = (1-w)D_1 + wD_2$, 
\[
    u_{D_2}-u_{D_1} = \int_0^1 D\mathcal G(D_w)[D_2 - D_1] \, dw,
\]
with convergence of the integral in $F$, and hence also in $L^\infty_tL^2_x$. Thus, by triangle inequality, 
\[
    \|u_{D_2} - u_{D_1} \|_{L^\infty_tL^2_x} \leq \int_0^1 \|D\mathcal G(D_w)[D_2 - D_1]\|_{L^\infty_tL^2_x}.
\]
But now we can estimate, by Proposition \ref{prop:LinearParabolicRegularity} and $L^2_tH^1_x \cap H^1_t H^{-1}_x \hookrightarrow C([0, 1], L^2(\Omega))$ (Theorem 3 in \textsection 5.9 in \cite {Evans}),
\begin{align*}
    \|D\mathcal G(D_w)[D_2 - D_1]\|_{L^\infty_tL^2_x} &\lesssim \|\nabla \cdot ([D_2 - D_1]\nabla u_{D_w})\|_{L^2_t H^{-1}_x} \leq \| \nabla u_{D_w} \|_{L^\infty(Q)} \|D_2-D_1\|_{L^2(Q)}
\end{align*}

But now we apply Proposition \ref{prop:HRDEReg}, and since $\beta > d/2 + 1$ and $\nabla u \in H^{\beta, \beta/2}$, by Proposition \ref{prop:BesovEmbeddings}, $\nabla u \in L^\infty(Q)$. Note that the upper bound is uniform in $w$ as $\|D_w\|_{C^1} \leq U$, and $D_{\min}$ is the same for $D_1, D_2$ and $D_w$. Hence, 
\[
    \|u_{D_2} - u_{D_1} \|_{L^\infty_t L^2_x} \lesssim \|D_2 - D_1 \|_{L^2(Q)}.
\]
\end{proof}

\subsubsection{Proof of Lemma \ref{lem:transportstability}}
\begin{proof}[Proof of Lemma \ref{lem:transportstability}]
    The proof of $(b)$ is the same as in \textsection 2.1.2.1 in \cite{N}, noting that we can extend $h$ to an $H^1(\Omega)$ function by $0$, and using periodic divergence theorem. The proof of $(a)$ uses a variant of Lemma 3.7 in \cite{ABJN}. To start, note that the second inequality is an application of H\"older's inequality. Also, write $u$ for $u_{D_1}$. For any $t \in [0, 1]$, we have from Lemma 3.6 in \cite{ABJN} that
    \[
        \| \nabla \cdot(h\nabla u(t, \cdot))\|_{L^2(\Omega)}\geq \| \nabla \cdot(|h|\nabla u(t, \cdot))\|_{L^2(\Omega)}.
    \]
    From Proposition \ref{prop:HRDEReg} we have $\|u\|_{H^{\beta + 1, (\beta+1)/2}} \leq C(M, D_{\min}, \beta, d, f)$. Now, we use the duality characterization of the $L^2$ norm:
    \[
        C\| \nabla \cdot(|h|\nabla u(t, \cdot))\|_{L^2(\Omega)} \geq \sup_{\|v\|_{L^2(\Omega)}\leq C}\int_\Omega v\nabla \cdot(|h|\nabla u(t, \cdot))\, dx
    \]
    Thus, choosing $v = -u$ and using the divergence theorem 
    \[
        \sup_{\|v\|_{L^2(\Omega)}\leq C}\int_\Omega v\nabla \cdot(|h|\nabla u(t, \cdot)\, dx \geq \int_{\Omega} |h| |\nabla u(t, \cdot)|^2 \, dx.
    \]
    Putting the chain of inequalities together, and integrating in time (note that $h$ is independent of time) yields the result.
\end{proof}
\subsubsection{Proof of Proposition \ref{thm:HRDEBack}}
\begin{proof}[Proof of Proposition \ref{thm:HRDEBack}]
    We employ the idea of \textsection 2.2.2 in \cite{CN} on $G$ as defined in \eqref{eq:levelsetmap} to obtain
   \begin{align*}
       &\int_0^1 \left( -\nabla \cdot([D_1 - D_2]\nabla u_{D_1}), 0 \right) \, dw \\
       &\quad = -\int_0^1 \left( \frac{\partial}{\partial t}[u_{1} - u_{2}] - \nabla \cdot (D_2 \nabla(u_{1} - u_{2})) - f'(\mathcal G_w)(u_{1} - u_{2}), 0\right)\, dw,
   \end{align*}
   where $u_i = \mathcal G(D_i)$, and $\mathcal G_w = (1-w)u_1 + w u_2$.
   Hence, applying $L^2(Q)$ norms to both sides, and using that the RHS $\lesssim \|u_1 - u_2\|_{H^{2, 1}}$, we get 
   \[
       \| \nabla \cdot([D_1 - D_2]\nabla u_{D_1}) \|_{L^2(Q)} \lesssim \|u_1 -u_2\|_{H^{2, 1}}.
   \]
   Now, if Condition \ref{con:Darcy} holds, we have by Lemma \ref{lem:transportstability} for $h = D_1 - D_2$ that
   \[
       \|D_1 - D_2\|_{L^1(\Omega)} \lesssim \| \nabla \cdot([D_1 - D_2]\nabla u_{D_1}) \|_{L^2(Q)}.
   \]
   To finish, we use the following interpolation inequalities:\footnote{For interpolation between $H^{\beta, \beta/2}$ spaces, this follows from H\"older's inequality for sequence spaces. For Sobolev spaces, see p.44 in \cite{LM}.}
   \begin{align*}
       \|u_1 - u_2\|_{H^{2, 1}} &\leq \|u_1 - u_2\|^{\frac{\beta-1}{\beta +1}}_{L^2} \|u_1 - u_2 \|^{\frac{2}{\beta + 1}}_{H^{\beta + 1, (\beta+1)/2}} \\
       \|D_1 - D_2\|_{L^2} &\lesssim \|D_1 - D_2\|_{H^{-d/2-1}}^{\frac{\beta}{\beta + d/2 + 1}}\|D_1 - D_2\|_{H^{\beta}}^{\frac{d/2 + 1}{\beta + d/2 + 1}}\lesssim \|D_1 - D_2\|_{L^1}^{\frac{\beta}{\beta + d/2 + 1}}
   \end{align*}
If Condition \ref{con:Darcy2} holds, we are again finished by Lemma \ref{lem:transportstability} and interpolation.
\end{proof}

\subsubsection{Proof of Proposition \ref{prop:NSBack}}
\begin{proof}[Proof of Proposition \ref{prop:NSBack}]
    For $u_i = u_{\theta_i}$ and $w = u_1 - u_2$, $w$ satisfies (from subtracting the equations for $u_1, u_2$)
    \begin{align*}
        (\nu_1 - \nu_2)Au_2 &= -B(w, u_1) - B(u_2, w) - \nu_1Aw - \frac{dw}{dt},\\
        w(0) &= \varphi_1 - \varphi_2.
    \end{align*}
    Taking $L^2(Q)$ norms on both sides, we see that the RHS is bounded above by a constant multiple of $\|w\|_{H^{2, 1}}$ by standard estimates on the bi-linear terms see \cite{KN} Proposition 3.1.iii. On the LHS, we apply Lemma \ref{lem:LaplaceNonZero} and finally obtain
    \[
        |\nu_1 - \nu_2| \leq C(M, d, \nu_{\min}, f) \|w\|_{H^{2, 1}}.
    \]
    From \cite{Evans} \textsection 5.9.2 Theorem 3, we have that
    \begin{align*}
        \|\varphi_1 - \varphi_2\|_{L^2(\Omega)} &\leq \sup_{t\in [0, 1]} \|w(t)\|_{L^2(\Omega)} \lesssim \|w\|_{H^{2, 1}}.
    \end{align*}
    Now, we simply interpolate to obtain the result.
\end{proof}
\subsubsection{Proof of Lemma \ref{lem:LaplaceNonZero}}\label{sec:LaplaceNonZero}
\begin{proof}[Proof of Lemma \ref{lem:LaplaceNonZero}]
    The idea is that if $Au$ starts out bounded below, then the integral will be bounded below, and when $Au$ is sufficiently small, the dynamics of the equation force $u$ to grow in the direction of $f$.

    We start by testing \eqref{eq:introNS} against the forcing function $f$ to obtain
    \begin{equation*}
        \frac{d}{dt}\big[\langle u_\theta, f\rangle\big] = \|f\|^2 - \nu\langle Au, f\rangle - \langle B(u, u), f\rangle.
    \end{equation*}
    Rearranging this equation and applying Cauchy-Schwarz with the usual bounds on the bilinear form, we get 
    \begin{equation*}
        \|f\|^2_{L^2}\leq \frac{d}{dt} \big[\langle u_\theta, f \rangle\big] + C(\|u_\theta\|_{H^1} + \|u_\theta\|^2_{H^1})\|f\|_{H^1} \text{ on }[0, 1].
    \end{equation*}
    Integrate this equation from $s$ to $t$ ($0\leq s \leq t\leq 1$) to get 
    \begin{align*}
        (t-s)\|f\|^2_{L^2} &\leq \langle u_\theta(t), f\rangle_{L^2} - \langle u_\theta(s), f\rangle_{L^2} + C\|f\|_{H^1}\int_{s}^t \bigg(\|u_\theta\|_{H^1} + \|u_\theta\|^2_{H^1}\bigg) \\
        &\leq \|f\|_{L^2}\|u_\theta(t)\|_{L^2} + \|f\|_{L^2}\| u_\theta(s)\|_{L^2} + C\|f\|_{H^1}\int_{s}^t \bigg(\|u_\theta\|_{H^1} + \|u_\theta\|^2_{H^1}\bigg).
    \end{align*}
    Now we integrate $s$ from $0$ to $t$, obtaining
    \begin{align}
        \frac{t^2}{2}\|f\|^2_{L^2}\leq t\|f\|_{L^2}&\|u_\theta(t)\|_{L^2} + \|f\|_{L^2}\int_0^t \|u_\theta(s)\|_{L^2}\, ds\label{eq:weirdbound} \\
        &+ C\|f\|_{H^1} \int_0^t \, ds\int_s^t \bigg(\|u_\theta(\tau)\|_{H^1} + \|u_\theta(\tau)\|^2_{H^1}\bigg)\, d\tau.\label{eq:weirdintegral}
    \end{align}
    If we consider the region of integration for \eqref{eq:weirdintegral}, we see it is over all $\tau$ and $s$ such that $0\leq s\leq \tau\leq t$ Thus, we can change it to the integral
    \begin{align*}
        \int_0^t \, ds\int_s^t \bigg(\|u_\theta(\tau)\|_{H^1} + \|u_\theta(\tau)\|^2_{H^1}\bigg)\, d\tau &= \int_{0}^t\, d\tau \int_0^\tau \, ds \bigg(\|u_\theta(\tau)\|_{H^1} + \|u_\theta(\tau)\|^2_{H^1}\bigg)\\
        &=\int_{0}^t\tau \bigg(\|u_\theta(\tau)\|_{H^1} + \|u_\theta(\tau)\|^2_{H^1}\bigg) \, d\tau\\
        &\leq t\int_0^t \bigg(\|u_\theta(\tau)\|_{H^1} + \|u_\theta(\tau)\|^2_{H^1}\bigg)\, d\tau.
    \end{align*}
    Returning to \eqref{eq:weirdbound}, using $\|f\|_{L^2} \lesssim \|f\|_{H^1}$ and $t\in[0, 1]$, we get 
    \begin{equation*}
        \frac{t^2}{2}\frac{\|f\|_{L^2}^2}{\|f\|_{H^1}} \leq C\bigg(t\|u_\theta(t)\|_{H^1}+ t\int_0^1 \|u_\theta(s) \|_{H^1}^2\, ds + (1+t)\int_0^1 \|u_\theta\|_{H^1}\, ds\bigg), 
    \end{equation*}
    and then we integrate this from $0$ to $1$ again and use Cauchy-Schwarz as well as the Poincar\'e inequality to get the result.
\end{proof}

\subsection{Proof of Theorem \ref{thm:PosteriorGeneral}}\label{sec:PosteriorGeneral}
\begin{proof}[Proof of Theorem \ref{thm:PosteriorGeneral}]
    We know from Lemma 8.21 in \cite{GvdV} that
    \[
        \Prob^N_{\theta_0}\left( \int_\Theta \frac{p^N_\theta}{p^N_{\theta_0}}\, d\Pi_N(\theta) \leq \Pi_N(\mathcal B_N)e^{-2N\varepsilon_N^2}\right) \leq \frac{1}{N\varepsilon_N^2}.
    \]
    So if we consider events
    \[
        A_N = \{ \int_\Theta \frac{p^N_\theta}{p^N_{\theta_0}}\, d\Pi_N(\theta) \geq e^{-(A+2)N\varepsilon_N^2} \}, 
    \]
    then we see that $\Prob^N_{\theta_0}(A_N) \to 1$ as $N \to \infty$. Now, consider
    \[
        \bar \Theta_N = \Theta_N \cap \{ C\delta_N \geq  \|\mathcal G(\theta) - \mathcal G(\theta_0)\|_{L^2} \}.
    \]
    Then, by the definition of the posterior, and \eqref{eq:testcondition}, we have 
    \begin{align}
        \Prob^N_{\theta_0}\big( \Pi(\bar \Theta_N^c|D_N)\geq e^{-bN\varepsilon_N^2}\big) &= \Prob_{\theta_0}^N\left( \frac{\int_{\bar \Theta_N^c}p^N_{\theta}/p^N_{\theta_0} \, d\Pi(\theta)}{\int_\Theta p^N_\theta/p^N_{\theta_0} \, d\Pi(\theta)} \geq e^{-bN\varepsilon_N^2}, \phi_N = 0, A_N\right) + o(1)\nonumber \\
        &\leq \Prob_{\theta_0}^N\left( \int_{\bar \Theta_N^c}p^N_{\theta}/p^N_{\theta_0} \, d\Pi(\theta) (1-\phi_N) \geq e^{-(b + A + 2)N\varepsilon_N^2}\right) + o(1).\label{eq:posteriorlowerbound}
    \end{align}
    By Markov's inequality and Fubini's theorem, we can bound \eqref{eq:posteriorlowerbound} by
    \begin{align*}
        e^{(b+A + 2)N\varepsilon_N^2} \mathbb \int_{\bar \Theta_N^c}\mathbb E^N_{\theta_0}\big[\frac{p_\theta^N}{p_{\theta_0}^N}(1 - \phi_N)\big] \, d\Pi(\theta) \leq e^{(b+A + 2)N\varepsilon_N^2} \mathbb \int_{\bar \Theta_N^c}\mathbb E^N_{\theta}\big[(1 - \phi_N)\big] \, d\Pi(\theta).
    \end{align*}
    Now, we use the fact $\bar \Theta_N^c = \Theta_N^c \cup \{\|\mathcal G(\theta) - \mathcal G(\theta_0)\|_{L^2} > C\delta_N\}$ and the bound on type two errors \eqref{eq:testcondition} to get
    \begin{align*}
        \eqref{eq:posteriorlowerbound} &\leq e^{(b+A+2)N\varepsilon_N^2} \bigg(2 \Pi(\Theta_N^c) + \int_{\Theta_N \cap \{ C\delta_N < \| \mathcal G(\theta) - \mathcal G(\theta_0)\|_{L^2}\}} \hspace{-5mm}\mathbb E_\theta^N [ 1 -\phi_N] \, d\Pi(\theta)\bigg) \\
        &\lesssim e^{(b+A+2 - B)N\varepsilon_N^2} + e^{(b+A+2-K)N \varepsilon_N^2} \to 0,
    \end{align*}
    as $N \to \infty$.
\end{proof}

\subsection{Proofs for Appendix \ref{sec:findingphi}}\label{sec:proofsforDarcy}

\subsubsection{Proof of Lemma \ref{lem:verifyingDarcy}}\label{sec:verifyingDarcy}
\begin{proof}[Proof of Lemma \ref{lem:verifyingDarcy}]
    First, we show that all the critical points must move at some rate away from their original locations. For $\alpha > d/2 + 5$ the following computations are justified. Taking the gradient of \eqref{eq:introRDE}, we see
    \begin{align*}
        \nabla u_t &= \nabla(\nabla \cdot (D\nabla u)) + f'(u)\nabla u\\
                   &= \nabla D \Delta u + D\nabla \Delta u+\nabla^2 u\nabla D+\nabla^2 D \nabla u + f'(u)\nabla u,
    \end{align*}
    but if we evaluate this at the point $(0, x^\ast)$ for a critical point $x^\ast \in \Omega$ of $\varphi$, then we have $\nabla \varphi(x^\ast) = 0$, so the last two terms cancel, and we have
    \[
        \nabla u_t(0, x^\ast) = \nabla D\Delta\varphi + D \nabla \Delta\varphi + \nabla^2 \varphi \nabla D.
    \]
    Thus, by \eqref{eq:movingcriticalpoints}, we get that $|(d/dt)\nabla u(0, x^\ast)| \geq \delta$, so in particular for some $i = 1, \dots d$, $|\partial_t \partial_i u(0, x^\ast)| \geq \delta/d$. Now, using Proposition \ref{prop:HRDEReg}, we see
    \[
        \|\nabla u(\cdot, x^\ast)\|_{C^2([0, 1])}\leq C\|u\|_{H^{\alpha +1, \an}} \leq C(M, d, \alpha, D_{\min}, f).
    \]
    Now, we have from the mean-value theorem that there is some $0 < \Tilde t < t $ so that 
    \begin{align*}
        t\partial_t\partial_i u(\Tilde t, x^\ast)&= \partial_i u(t, x^\ast) - \partial_i u(0, x^\ast) \\
        &= \partial_i u(t, x^\ast).
    \end{align*}
    So, $|\partial_t\partial_i u(\Tilde t, x^\ast) - \partial_t\partial_i u(0, x^\ast)| \leq t \|\nabla u(\cdot,x^\ast)\|_{C^2}\leq tC(M) $, and hence for $t < C(M)^{-1}\delta/(2d)$, we see
    \begin{equation*}
        |\nabla u(t, x^\ast)| \geq |\partial_i u(t, x^\ast)| =  t|\partial_t \partial_i u(\Tilde t, x^\ast)| \geq t \frac{\delta}{2d},
    \end{equation*}
    so, if we fix some $C(M)^{-1}\delta/(2d)> t^\ast >0$ (which depends only on $\delta, M, d$) we get for all $t^\ast < t < C(M)^{-1}\delta/(2d)$
    \[
        |\nabla u(t, x^\ast)| > c(\delta, d, M).
    \]
    Now, using the fact that the Lipschitz constant of $\nabla u$ is bounded by some constant depending on $M, d, D_{\min}, \alpha, f$, we see that there is some $\eta = \eta(\delta, d, M, D_{\min}, \alpha, f)$ so that for all $t^\ast < t < C(M)^{-1}\delta/(2d)$ and $|x - x^\ast| < \eta$, 
    \[
        |\nabla u(t, x)| > c(\delta, d, M)/2.
    \]
    Then, we can integrate this in time to get that
    \begin{equation}\label{eq:DarcyCritical}
        \int_0^1 |\nabla u(t, x)|^2 \, dt \geq c_1 >0
    \end{equation}
    for any $x$ within distance $\eta$ of a critical point $x^\ast$. Now, if we let 
    \[
        R = \{ x \in \Omega : \text{ for all critical points }x^\ast, |x - x^\ast| \geq \eta/2\},
    \]
    then we know that $|\nabla \varphi| > 0$ on $R$, and $R$ is compact, so $|\nabla \varphi|$ is uniformly lower bounded on $R$. Thus, utilizing the Lipschitz constant of $\nabla u$, we get that there is some constant $c_2 = c_2(M, d, D_{\min}, \alpha, f)$ such that
    \begin{equation}\label{eq:DarcyNonCritical}
        \int_0^1 |\nabla u(t, x)|^2 \, dt \geq c_2 >0 \quad \text{for all}\,\, x\in R.
    \end{equation}
    But combining \eqref{eq:DarcyCritical} and \eqref{eq:DarcyNonCritical} verifies Condition \ref{con:Darcy} for $c_0 = \min(c_1, c_2)$.
\end{proof}

\subsubsection{Proof of Proposition \ref{prop:RDvarphi}}\label{sec:RDvarphi}
\begin{proof}[Proof of Proposition \ref{prop:RDvarphi}]
    Note that by the triangle inequality, $D \geq D_{\min}$, and $\|\nabla D\|\leq M$, we have
    \begin{align}
        |D\nabla(\Delta \varphi) + \Delta \varphi \nabla D + \nabla^2 \varphi \nabla D| &\geq |D| |\nabla(\Delta \varphi)| - |\Delta\varphi||\nabla D| - \|\nabla^2 \varphi\||\nabla D|\nonumber \\
                                                                                        &\geq D_{\min} |\nabla (\Delta \varphi)| - M(|\Delta \varphi| + \|\nabla^2 \varphi\|),\label{eq:criticallb}
    \end{align}
    where $\|\nabla^2 \varphi\|$ is the operator norm. The idea is then to construct a function that has a mild second derivative, but large third derivative. We search for a function of the form 
    \[
        \varphi(x) = \sum_{i = 1}^d g(x_i).
    \]
    Then, we compute that
    \begin{align}\label{eq:varphiformula}
        \begin{split}
        \nabla \varphi(x) &= (g'(x_1), \dots, g'(x_d))\\
        \nabla^2 \varphi (x) &= \diag(g''(x_1), \dots, g''(x_d)) \\
        \Delta \varphi(x) &= \sum_{i = 1}^d g''(x_i) \\
        \nabla \Delta \varphi(x) &= (g'''(x_1), \dots, g'''(x_d)).
    \end{split}
    \end{align}
    Thus, we get the bound
    \begin{equation}\label{eq:criticallb2}
        (\ref{eq:criticallb}) \geq D_{\min} |\nabla \Delta \varphi| - 2dM\|g''\|_\infty,
    \end{equation}
    which needs to be greater than some $\delta$ at every critical point of $\varphi$. Now, consider the functions $g_n: \mathbb T \to \R$:
    \[
        g_n(x) = -\frac{1}{2\pi n^2} \cos(2\pi nx) + \frac{1}{4\pi n^2}\sin(4\pi nx).
    \]
    Then, 
    \[
        g_n'(x) = \frac{1}{n} \sin( 2\pi nx) + \frac{1}{n}\cos(4\pi nx).
    \]
    Then, using the double angle formula, we see that $g_n'(x_c) = 0$ if and only if
    \begin{align*}
        0 &= \sin(2 \pi nx_c) + 1 - 2\sin^2(2\pi nx_c) \\
        &= (1 + 2\sin(2\pi nx_c))(1 - \sin(2\pi nx_c)),
    \end{align*}
    so let $C_n$ be the set of points on $\mathbb T$ such that $\sin( 2\pi n x_c) = -1/2$ or $1$. Next, note that
    \[
        |g_n''(x)| = |2\pi\cos(2\pi n x) - 4\pi \sin(4\pi nx)| \leq 6\pi,
    \]
    so $\|g''\|_\infty \leq 6\pi$. Finally, we have
    \[
        g_n'''(x) = -4\pi^2 n \sin(2 \pi n x) - 16\pi^2 n\cos(4\pi nx).
    \]
    Now, at any critical point, we have $\sin(2 \pi n x_c) = 1/2$ or $-1$. But 
    \begin{align*}
        \sin(2\pi n x_c) = -1/2 &\implies \cos(4 \pi n x_c) = 1/2, \\
        \sin(2 \pi n x_c) = 1 &\implies \cos(4 \pi n x_c) = -1.
    \end{align*}
    Hence, $g_n'''(x_c) = -6\pi^2 n$ or $12\pi^2 n$. Now, fix an $n$ large enough so that $6\pi^2 n  D_{\min}\sqrt d - 12dM\pi > 1$, noting that $n$ depends only on $M, D_{\min}$, and $d$. Then, from \eqref{eq:varphiformula}, we see that the critical points of $\varphi$ with this choice of $g = g_n$ is $(C_n)^d\subseteq \Omega$, and hence for $x_c \in (C_n)^d$
    \[
        |\nabla \Delta \varphi(x_c)| \geq 6\pi^2 n\sqrt d,
    \]
    so by \eqref{eq:criticallb2} we are done.
\end{proof}

\subsubsection{Proof of Proposition \ref{prop:opendense}}
\begin{proof}[Proof of Proposition \ref{prop:opendense}]
    We will show using Lemma \ref{lem:verifyingDarcy} that the set of $\varphi$ satisfying Condition \ref{con:Darcy} for some $c_0$ contains an open dense set. To start, set $L_D(\varphi) = D\nabla(\Delta \varphi) + \Delta \varphi \nabla D + \nabla^2 \varphi \nabla D$.
    Then, let $G$ be the set 
    \[
        G = \{ \varphi \in C^k(\Omega): \nexists x^\ast \in \Omega \text{ s.t. at } x^\ast,\, \nabla \varphi = 0 \text{ and } L_D(\varphi) = 0\}.
    \]
    First, we will show that $G^c$ is closed. Suppose that $\varphi_n\in G^c$ with $\varphi_n \to \varphi$ in $C^k$. Then, $\varphi_n \to \varphi$ in $C^3$. Also, there exists $x_n\in \Omega$ such that
    \[
        \nabla \varphi_n(x_n) = 0 \text{ and } L_D(\varphi_n)(x_n) = 0.
    \]
    But since $\Omega$ is compact, there exists a subsequence $n_k$ such that $x_{n_k} \to x$ for some $x\in \Omega$. But since $\varphi_n\to \varphi$ and $L_D(\varphi_n)\to L_D(\varphi)$ uniformly as $n\to \infty$, we have 
    \[
        \nabla \varphi(x) = 0 \text{ and } L_D(\varphi)(x) = 0,
    \]
    and hence $\varphi \in G^c$ so $G$ is open. Now we show that $G$ is dense. First, let $\varphi \in C^k$ and $\epsilon > 0$. Then, by Theorem 6.1.2 in \cite{H} there exists some $\Tilde \varphi\in C^k$ such that $\|\varphi - \Tilde \varphi \|_{C^k} < \epsilon/2$ and $\Tilde \varphi$ is Morse, i.e. has isolated critical points, all of which are non-degenerate. Since $\Omega$ is compact there can only be finitely many of them.

    Let $x_1, \dots, x_m$ be the set of points for which
    \[
        \nabla \Tilde \varphi(x_i) = 0 \text{ and } L_D(\Tilde \varphi)(x_i) = 0.
    \]
    Now, by translation and Proposition \ref{prop:RDvarphi}, we know there exists $g_i\in C^\infty$ such that 
    \begin{equation}\label{eq:fixingfunction}
        \nabla g_i(x_i) = 0 \text{ and } L_D(g_i)(x_i) \neq 0
    \end{equation}
    By multiplying by a cutoff function $\equiv 1$ on a neighborhood of $x^\ast$ and $\equiv 0$ around all other critical points, we can also assume that the support of each $g_i$ only contains a single critical point of $\Tilde \varphi$ (namely $x_i$). Then consider the function 
    \[
        \hat \varphi = \Tilde \varphi + \eta \sum_{i = 1}^m g_i
    \]
    for sufficiently small $\eta$. It is known that a sufficiently small perturbation of a Morse function does not change the number of critical points it has,\footnote{
    This fact is a corollary of the Stack of records theorem (Exercise 1.4.7 in \cite{GP}) applied to $F: \Omega \times S^1 \to \R^d\times S^1$ given by $F(x, t) = (\nabla(\Tilde \varphi + \sin(2\pi t) g)(x), t)$.} but we have chosen the $g_i$ such that any critical point of $\Tilde \varphi$ is also a critical point of $\hat \varphi$, and so the critical points of $\hat \varphi$ are the same as $\Tilde \varphi$ for sufficiently small $\eta$. Then \eqref{eq:fixingfunction} ensures that $\hat\varphi \in G$.
\end{proof}

\section{Wavelet Characterization of Parabolic Sobolev spaces}\label{sec:besovIsSobolev}
\subsection{Facts about wavelets}\label{sec:WaveletFacts}
We will recap a few necessary facts about wavelets to set notation. A reference for this material is \textsection 4.1 to \textsection 4.3 in \cite{GN}. First, we recall the \textit{boundary-corrected wavelets}, first introduced in \cite{CDV93} and used to measure regularity of functions on $[0, 1]$. 

Fix an $S\in \N$ and take any $L = L(S)$ large enough so that $2^{-L}(2S - 1)\leq 1$, where $L$ is called the \textit{base resolution} of the basis. Then, the \textit{boundary-corrected scaling functions} (at some level $l \geq L$) is a collection
\[\{ \phi^{\bc}_{lk} : k = 1, \dots, 2^ l\},\]
where when $k = S+1, S+2, \dots, 2^l - S$, $\phi_{lk}^{\bc}(x) = 2^{l/2}\phi(2^l x - k)$ for $\phi$ the $S$-regular Daubechies scaling function. The other values of $k$, a boundary correction is introduced, see \cite{CDV93} for details. We note the following properties:
\begin{enumerate}
    \item The collection $\{\phi^{\bc}_{lk} : k = 1, \dots 2^l\}$ is orthonormal,
    \item The linear span of the $2^l$ functions $\{ \phi^{\bc}_{lk} \}$, call it $V^{\bc}_l$, contains all polynomials of degree $\leq S - 1$,
    \item $V^{\bc}_l \subset V^{\bc}_{l+1}$ for all $l \geq L$, and $\bigcup_{l \geq L} V_{l}^{\an}$ is dense in $L^2([0, 1])$.
    \item $\supp \phi^{\bc}_{lk} \subseteq [0, 2^{-l}(S+k)]$ for $k = 1, \dots S$, and
    \item $\supp \phi^{\bc}_{lk} \subseteq [2^{-l}(k - S + 1), 1]$ for $k = 2^l - S +1 , \dots, 2^l $.
\end{enumerate}
And further \textit{boundary-corrected wavelets} can be constructed
\[\{ \psi^{\bc}_{lk} : k = 1, \dots, 2^l\},\]
where when $k = S+1, \dots, 2^l - S$, $\psi^{\bc}_{lk}(x) = 2^{l/2}\psi(2^l x - k)$ for $\psi$ the $S$-regular Daubechies wavelet, and the following holds.
\begin{enumerate}
    \item Let $W_l^{\bc}$ be the orthogonal complement of $V_l^{\bc}$ in $V_{l+1}^{\bc}$, which has dimension $2^l$. Then $\{\psi^{\bc}:k = 0, \dots 2^l - 1\}$ is an orthonormal basis of $W_l^{\bc}$.
    \item The following scaling properties hold:
        \begin{align*}
            \psi_{lk}^{\bc}(x) &= 2^{(l-L)/2}\psi_{Lk}(2^{l-L}x)\text{ for }k = 1, \dots, S, \\ 
            \psi_{lk}^{\bc}(x) &= 2^{(l-L)/2}\psi_{L,k+2^L-2^l}(2^{l-L}(x-1)+1)\\ 
                                  &\hspace{34.5mm}\text{ for }k = 2^l -S + 1, \dots, 2^l,\nonumber
        \end{align*}
    \item $\supp\psi_{lk}^{\bc}\subseteq \big[0, 2^{-l}(S+k)\big]$ for $k = 1, \dots S$, and
    \item $\supp\psi_{lk}^{\bc}\subseteq \big[2^{-l}(k-S+1), 1\big]$ for $k = 2^l-S+1, \dots, 2^l$.
    \item For all $j = 0, \dots, S-1$, 
\begin{equation}\label{eq:BCWaveletsOrthProp}
    \int_0^1 \psi^{\bc}_{lk}(x)x^j\, dx = 0.
\end{equation}
\end{enumerate} 
And from this we have as usual that 
\[
    \{\phi^{\bc}_{Lk} : k = 1, \dots 2^L\} \cup\bigcup_{l\geq L}\{\psi^{\bc}_{lk} : k = 1, \dots, 2^l\}
\]
is an orthonormal basis of $L^2([0, 1])$. 
Next, we set notation for wavelets on $\Omega = \mathbb T^d$. For $l \in \N$ and $k \in [2^{dl}]$, we have the periodic Meyer scaling functions on $\R^d$ written as $\Phi_{lk}$, and we set
\[\mathbf{V}_l = \Span(\Phi_{lk} : k \in [2^{dl}]).\]
As usual, we also have the corresponding periodic Meyer wavelets written as $\Psi_{l k}$ for $k \in [2^{d(l+1)}-2^{dl}]$, and 
\[\mathbf{W}_j = \mathbf{V}_{j+1} \cap \mathbf{V}^\perp_{j},\]
with $\{ \Psi_{lk} : k \in [2^{d(l+1)}-2^{dl}]\}$ form an orthonormal basis for $\mathbf{W}_l$.

\subsection{Wavelet Characterization of Parabolic Sobolev Spaces}
\begin{definition}\label{def:AnisotropicBesov}
    For $s > 0$, fix a smoothness $S$ so that $\phi^{\bc}_{lk}, \psi^{\bc}_{lk} \in C^r$ for a natural number $r> s$. We define the Besov spaces
    \[
        B^{s, \an}_{pq} = \{ f \in L^p([0, 1]\times \Omega) : \|f\|_{B^{s, \an}_{pq}}<\infty\},
    \]
    where
    \[
        \|f\|_{B^{s, \an}_{pq}} = \left( \sum_{j = L_0-1}^\infty2^{qj(s+\frac{d+2}{2}-\frac{d+2}{p})}\bigg( \sum_{\substack{\mathbb L(l)= j\\ k \in \mathcal K(l)}} |\langle f, \Psi^{\an}_{lk}\rangle|^p\bigg)^{q/p} \right)^{1/q},
    \]
    with the usual adjustments when $p$ or $q = \infty$.
    In the special case where $p = q = 2$, we note 
    \[
        B^{s, \an}_{22} = H^{s,\an}.
    \]
    The definition is also valid for $s = 0$, $p = \infty$, and $q = 1$. 
\end{definition}
We introduce some terminology that is helpful when dealing with these spaces:
\begin{definition}
    Let $p \in [1, \infty]$. Then, we say that $f\in L^p(Q)$ has \textit{finite wavelet support} if 
    \[\langle f, \Psi^{\an}_{lk}\rangle = 0\]
    for all but finitely many indices $l, k$. 
\end{definition}
\begin{proposition}\label{prop:BesovEmbeddings2}
   The usual embeddings between Besov spaces hold, where the effective dimension is $d+2$. For example, if $s > (d+2)(1/p-1/r)$, with $1\leq p, r\leq \infty$, 
   \[
       B^{s, \an}_{pq} \hookrightarrow L^r([0, 1]\times \Omega).
   \]
\end{proposition}
\begin{proof}
    We only use the cases when $p = q =2$, $s>(d+2)/2$, and $p = \infty$, $q = 1$, $s = 0$, so we will only prove these here. 
    See \textsection 5.3.2 in \cite{T} for the case of anisotropy on $\R^n$ spaces. Suppose that $u\in H^{s, s/2}$ for $s > (d+2)/2$ has finite wavelet support. Then, we can write $u$ as
    \[
        u(t, x) = \sum_{j \geq L_0} \sum_{(\iota, l), k} \langle u, \Psi^{\iota, \an}_{lk}\rangle \Psi^{\iota, \an}_{lk}(t, x),
    \]
    and denote by $u_j$ the projection of $u$ onto $W_{j-1}$ for $j > L_0$ and onto $V_{L_0}$ for $j = L_0$. We use Cauchy-Schwarz to obtain
    \begin{align*}
        |u(t, x)| &\leq \sum_{j \geq L_0} \sum_{(\iota, l), k} |\langle u, \Psi^{\iota, \an}_{lk} \rangle| |\Psi^{\iota, \an}_{lk}(t, x)| \leq \sum_{j \geq L_0} \|u_j\|_{L^2}\left(\sum_{(\iota, l), k} |\Psi^{\iota, \an}_{lk}(t, x)|^2\right)^{1/2}.
    \end{align*}
    Now, by standard wavelet theory one can show that
    \[
        \sum_{(\iota, l), k} |\Psi^{\iota, \an}_{lk}(t, x)|^2 = O(2^{(d+2)j}),
    \]
    so again by Cauchy-Schwarz we have
    \begin{align*}
        |u(t, x)| &\leq \sum_{j \geq L_0} 2^{(d+2)j/2}\|u_j\|_{L^2} \leq \left(\sum_{j \geq L_0} 2^{(d+2 - 2s)j} \right)^{1/2} \left(\sum_{j \geq L_0} 2^{2js}\|u_j\|^2_{L^2}\right)^{1/2} = C\|u\|_{H^{s, s/2}},
    \end{align*}
    and we apply the usual extension arguments to obtain the result. The proof is similar for the $s = 0$, $p = \infty$, and $q = 1$ case. 
\end{proof}
\begin{remark}\label{rmk:equivalentnorm}
    In light of \eqref{eq:atomicdecomp}, we see that indeed any choice of bases for the $W_{j}^{\an}$ and $V_{L_0-1}^{\an}$ would lead to an equivalent definition. From this we can deduce that an alternate choice of tensor basis yields an equivalent space. Indeed, for the proofs we sometimes need to consider finer properties of the periodic Meyer wavelets. This aligns with the notation of \cite{T} more closely. Throughout this section, we will write $\phi^{\bc}$ and $\psi^{\bc}$ for the boundary corrected scaling functions and wavelets, and $\phi$ and $\psi$ as the periodized Meyer scaling functions and wavelets. 

    Let $\mathcal I$ be the set of $d+1$-tuples consisting of $0$s and $1$s with not all entries $0$. Next, we define $\mathcal I^{j, \an}\subseteq \mathcal I \times \N^2$ where $\mathcal I^{L_0, \an} = \{(0, (2L_0, L_0))\}$, and for $j > L_0$, we define $(\iota, l) \in \mathcal I^{j, \an}$ if and only if
    \begin{align*}
        l_0 &= 2(j-1) \quad \text{if }\iota_0 = 0,\\
        l_0 &= 2(j-1) \text{ or }2j-1 \quad \text{if } \iota_0 = 1,\\
        l_1 &= j-1.
    \end{align*}
    Then, we define an alternate basis by
    \[\Psi^{\iota, \alt}_{lk}(t, x) = \psi^{\iota_t, \bc}_{l_tk_t}(t)\Psi^{\iota_x}_{l_x k_x}(x),\]
    where $\psi^{0, \bc} = \phi^{\bc}$, $\psi^{1, \bc} = \psi^{\bc}$, $\Psi^{\iota}$ are tensorized periodic wavelets defined as in \textsection 4.3.6 in \cite{GN}, $\iota = (\iota_t, \iota_x)$, $k = (k_t, k_x)\in [2^{l_0}]\times [2^{l_1}]^d\equiv \mathcal K'(l)$. Then, we have
    \[W^{\an}_j = \Span(\Psi^{\iota, \alt}_{lk} : (\iota, l)\in \mathcal I^{j+1}, k \in \mathcal K'(l)).\]
    Thus, using \eqref{eq:atomicdecomp} we see that an equivalent norm for $H^{s, \an}$ is
    \begin{equation}\label{eq:oldway}
        \|u\|_{H^{s, \an}}^2 = \sum_{j = L_0}^\infty 2^{2js}\sum_{(\iota, l)\in \mathcal I^{j, \an}}\sum_{k \in \mathcal K'(l)} |\langle u, \Psi^{\iota, \alt}_{lk}\rangle|^2,
    \end{equation}
    and we mainly use this characterization in this section.
\end{remark}
In order to show Proposition \ref{prop:besovIsSobolev}, we first need the following auxiliary lemma. Recall that $u_j$ is the orthogonal projection onto $V_{L_0 -1}^{\an}$ or $W_{j}^{\an}$ from \eqref{eq:atomicdecomp}.
\begin{lemma}\label{lem:rieszBlocks}
    Let $u\in L^2(Q)$ have finite wavelet support, and $j \geq L_0$. Then,
    \begin{enumerate}[label = (\alph*)]
        \item for any $i = 1, \dots d$, there is a constant independent of $j$ and $u$ such that
            \[
                \bigg\|\bigg(\frac{\partial u}{\partial x_i}\bigg)_j\bigg\|^2_{L^2(Q)} \leq C 2^{2j}\big(\|u_{j-1}\|^2_{L^2(Q)} + \|u_{j}\|^2_{L^2(Q)} + \|u_{j+1}\|^2_{L^2(Q)}\big), \text{ and }
            \]
        \item for any $\alpha \geq 2$, there is a constant independent of $j, j'$ and $u$ such that
            \[
                \bigg\|\frac{d u}{dt}\bigg\|^2_{H^{\alpha - 2, \an}} \leq C \|u\|_{H^{\alpha, \an}}^2,
            \]
            where we pick $S = S(\alpha)$ the smoothness of the boundary corrected Daubechies wavelets is large enough.
    \end{enumerate}
    In particular, we have that if $v\in H^{\alpha, \an}$  then
    \begin{enumerate}[label = (\roman*)]
        \item if $\alpha \geq 1$, $\partial v/\partial x_i \in H^{\alpha -1, \an}$, and
        \item if $\alpha \geq 2$, $d v/dt \in H^{\alpha -2, \an}$
    \end{enumerate}
\end{lemma}

\begin{proof}
    We do the in particular part first. For $(i)$, note that for $u \in H^{\alpha, \an}$ such that $\frac{\partial u}{\partial x_i}$ exists, we can compute using $(a)$
    \begin{align*}
        \bigg\|\frac{\partial u}{\partial x_i}\bigg\|_{H^{\alpha -1, \an}}^2 &= \sum_{j\geq L_0} 2^{2j(\alpha - 1)}\bigg\|\frac{\partial u}{\partial x_i}\bigg\|_{L^2}^2 \leq 3C \sum_{j \geq L_0}2^{2j\alpha}\|u_j\|_{L^2}^2 = 3C \|u\|_{H^{\alpha, \an}}^2,
    \end{align*}
    and since such $u$ are dense in $H^{\alpha, \an}$ we can extend via the usual arguments. For $(ii)$ again take $u\in H^{\alpha, \an}$ such that $\frac{d u}{dt}$ exists, and using $(b)$, extend via the usual arguments.

    First, we claim that
    \[
        \bigg(\frac{\partial u}{\partial x_i}\bigg)_j = \bigg(\frac{\partial u_{j-1}}{\partial x_i}\bigg)_j +\bigg(\frac{\partial u_j}{\partial x_i}\bigg)_j+ \bigg(\frac{\partial u_{j+1}}{\partial x_i}\bigg)_j.
    \]
    This is because we have
    \[
        \frac{\partial u}{\partial x_i} = \sum_{j' \geq L_0} \sum_{(\iota, l), k} \langle u, \Psi^{\iota, \alt}_{lk}\rangle \frac{\partial \Psi^{\iota, \alt}_{lk}}{\partial x_i},
    \]
    and we note that for some other $(\iota', l')\in \mathcal I^{j, \an}$ and $k \in \mathcal K(l)$, 
    \[
        \bigg\langle \frac{\partial \Psi^{\iota, \alt}_{lk}}{\partial x_i}, \Psi^{\iota', \alt}_{l'k'}\bigg\rangle = \langle \psi^{\iota_t, \bc}_{l_tk_t}, \psi^{\iota_t', \bc}_{l_t'k_t'}\rangle \bigg\langle \frac{d \psi^{\iota_i}_{j'-1, k_i}}{dx}, \psi^{\iota_i'}_{j-1,k_i'}\bigg\rangle\prod_{\substack{m = 1 \\ m \neq i}}^d\langle \psi^{\iota_m}_{j'-1, k_m}, \psi^{\iota_m'}_{j-1, k_m'}\rangle,
    \]
    so if $|j' - j| > 1$, then this inner product is $0$ since the Meyer wavelets have disjoint Fourier support in that case, and at least one term in the product will have a Meyer wavelet at the higher level. See Theorem 4.2.9 in \cite{GN}.
    Thus, to bound $(\partial_i u)_j$ in $L^2$ we can bound
    \begin{align*}
        \|(\partial_i u)_j\|_{L^2} &\leq \|(\partial_i u_{j-1})_j\|_{L^2} + \|(\partial_i u_{j})_j\|_{L^2} + \|(\partial_i u_{j+1})_j\|_{L^2} \\
                                   &\leq \|\partial_i u_{j-1}\|_{L^2} + \| \partial_i u_{j}\|_{L^2} + \|\partial_i u_{j+1}\|_{L^2}.
    \end{align*}
    Thus, to deduce $(a)$, it suffices to prove that for some constant independent of $j$, 
    \[
        \|\partial_i u_{j}\|_{L^2}^2 \leq C2^{2j}\|u_j\|^2_{L^2}.
    \]
    Now, consider the fact that
    \begin{align*}
        \|\partial_i u_j\|^2_{L^2} = \sum_{(\iota, l), (\iota', l') \in \mathcal I^{j, \an}} \sum_{k, k'} \langle u, \Psi^{\iota, \alt}_{lk}\rangle \langle u, \Psi^{\iota', \alt}_{l'k'} \rangle \langle \partial_i \Psi^{\iota, \alt}_{lk}, \partial_i \Psi^{\iota', \alt}_{l'k'}\rangle.
    \end{align*}
    So, consider for $(\iota, l), (\iota', l')\in \mathcal I^{j, \an}$ and $k \in \mathcal K'(l)$, $k' \in \mathcal K'(l')$:
    \[
        \bigg\langle \partial_i \Psi^{\iota, \alt}_{lk}, \partial_i \Psi^{\iota', \alt}_{l'k'}\bigg\rangle = \langle \psi^{\iota_t, \bc}_{l_tk_t}, \psi^{\iota_t', \bc}_{l_t'k_t'}\rangle \bigg\langle \frac{d \psi^{\iota_i}_{j-1, k_i}}{dx}, \frac{ d \psi^{\iota_i'}_{j-1,k_i'}}{dx}\bigg\rangle\prod_{\substack{m = 1 \\ m \neq i}}^d\langle \psi^{\iota_m}_{j-1, k_m}, \psi^{\iota_m'}_{j-1, k_m'}\rangle.
    \]
    The only times when this is non-zero is when $(\iota, l) = (\iota', l')$, and $k = k'$ except at the $i$-th index, so we restrict to this case. Now, from standard wavelet theory using the fact that Meyer wavelets and scaling functions are Schwartz, we have the bound
    \[
        \bigg|\bigg\langle \frac{d \psi^{\iota_i}_{j-1, k_i}}{dx}, \frac{ d \psi^{\iota_i'}_{j-1,k_i'}}{dx}\bigg\rangle\bigg| \leq C 2^{2j}\sum_{m\in \Z} \frac{1}{(1+|k_i - k_i' + m'2^{j-1}|)^2}.
    \]
    Thus, we see that 
    \begin{align*}
        \|\partial_i u_j \|^2_{L^2} &= \sum_{\substack{(\iota, l) \in \mathcal I^{j, \alt} \\\iota_i'}}\sum_{\substack{k\in \mathcal K'(l)\\ k_i' \in [2^{j-1} -1]}}\langle u, \Psi^{\iota, \alt}_{lk}\rangle \langle u, \Psi^{\iota', \alt}_{l'k'} \rangle \bigg\langle \frac{d \psi^{\iota_i}_{j-1, k_i}}{dx}, \frac{ d \psi^{\iota_i'}_{j-1,k_i'}}{dx}\bigg\rangle \\
                                    &\leq C2^{2j}\sum_{\substack{(\iota, l) \in \mathcal I^{j, \alt}\\ \iota_i'}}\sum_{\substack{k\in \mathcal K'(l)\\ k_i' \in [2^{j-1} -1]}}|\langle u, \Psi^{\iota, \alt}_{lk}\rangle||\langle u, \Psi^{\iota', \alt}_{l'k'} \rangle| \sum_{m\in \Z} \frac{1}{(1 + |k_i - k_i' + m2^{j-1}|)^2}\\
                                    &\leq C2^{2j}\|u_j\|^2_{L^2},
    \end{align*}
    where in the final inequality we have used Young's inequality and the fact that
    \[
        \sum_{k'_i \in [2^{j-1} - 1]}\sum_{m\in \Z}\frac{1}{(1+|k_i - k_i' + m2^{j-1}|)^2} \leq \frac{\pi^2}{3}.
    \]
    Hence, we have established $(a)$. We prove $(ii)$ using a similar idea. However, we will need the original basis for this part. As before, we can write
    \[
    \frac{ du}{dt} = \sum_{j, j' \geq L_0 -1}\sum_{k_t, k_x} \langle u, \Psi^{\an}_{jj'k_tk_x}\rangle \frac{d \Psi^{\an}_{jj'k_tk_x}}{dt}.
    \]
    Thus, if we re-expand this into the alternate basis, we see
    \[
        \frac{d u}{dt} = \sum_{ii'jj'} \sum_{k_tk_t'k_xk_x'} \langle u, \Psi^{\an}_{jj'k_tk_x}\rangle \bigg\langle \frac{d \Psi^{\an}_{jj'kk_tk_x}}{dt}, \Psi^{\an}_{ii'k_t'k_x'} \bigg\rangle \Psi^{\an}_{ii'k_t'k_x'}.
    \]
    Now, if we examine the inner product
    \[
        \bigg\langle \frac{d \Psi^{\an}_{jj'k_tk_x}}{dt}, \Psi^{\an}_{ii'k_t'k_x'} \bigg\rangle = \bigg\langle \frac{d \psi^{\bc}_{jk_t}}{dt}, \psi^{\bc}_{ik_t'}\bigg\rangle \langle \Psi_{j'k_x}, \Psi_{i' k_x'}\rangle,
    \]
    we see that it is zero unless $j' = i'$ and $k_x = k_x'$ (in which case it is $1$). This means the action of the time derivative decouples from space. So, if we wish to show that 
    \[
        \bigg\| \frac{ du}{dt} \bigg\|_{H^{\alpha - 2, \an}}^2 \leq C \|u\|^2_{H^{\alpha, \an}},
    \]
    it's equivalent to show that the matrices from $\ell^2 \to \ell^2$ given by
    \[
        M^{j'}_{(j, k_t), (i, k_t')} = 2^{-\alpha\max(j/2, j')}2^{(\alpha - 2)\max(i/2, j')}\langle \partial_t \psi^{\bc}_{jk_t}, \psi^{\bc}_{ik_t'} \rangle
    \]
    are bounded uniformly in $j'$. For this, we will use a discrete version of Schur's test, see \cite{HS} Theorem 5.2. First, we note that when $j \leq i$, we can use \eqref{eq:BCWaveletsOrthProp} to see that, for $S$ the smoothness of the boundary-corrected Daubechies wavelets and Taylor's theorem, there is a constant depending only on $S$ such that
    \[
        |\langle \partial_t \psi^{\bc}_{jk_t}, \psi_{i k_t'}^{\bc}\rangle| \leq C 2^{j}2^{-(S-1)(i-j)} = C2^{j}2^{-(S-1)|i-j|}.
    \]
    Then, we see that for $2j' \leq j \leq i$,
    \begin{align}
        |M^{j'}_{(j, k_t), (i, k_t')}| &\leq C2^{-\alpha\max(j/2, j')}2^{(\alpha - 2)\max(i/2, j')}2^{j}2^{-(S-1)|i-j|}\leq C2^{(-S+\alpha/2)|i-j|}.\label{eq:schurtailbound}
    \end{align}
    Now, since we have an eye towards using Schur's test, for a fixed $(j, k_t)$, we want
    \[
        \sum_{(i, k_t')} |M_{(j, k_t), (i, k_t')}^{j'}| \leq C,
    \]
    where $C$ does not depend on $j'$, $j$ or $k_t$. We observe that since the boundary corrected wavelets have compact support, for a fixed $i$, there are only $2^{|i-j|}$ many $k_t'$ such that the supports of $\psi^{\bc}_{jk_t}$ and $\psi^{\bc}_{jk_t'}$ intersect. Hence, by \eqref{eq:schurtailbound}, we get
    \begin{equation}\label{eq:schurtailcase1a}
        \sum_{\substack{(i, k_t') \\ i\geq j}} |M_{(j, k_t), (i, k_t')}^{j'}| \leq C\sum_{i\geq j}  2^{(-S + \alpha/2 + 1)|i-j|}.
    \end{equation}
    From this, we now know fix $S$ large enough so that $S > \alpha/2 + 1$. This is the heart of the argument, and we'd ideally like to simply perform the same estimate in the regime where $i \leq j$. However, doing this requires the use of integration by parts. In the case where either $\psi^{\bc}_{jk_t}$ or $\psi^{\bc}_{ik_t'}$ is an internal wavelet, integration by parts is justified, we get the usual estimate
    \[
        |M_{(j, k_t), (i, k_t')}^{j'}| \leq C 2^{-\alpha \max(j/2, j')}2^{(\alpha - 2)\max(i/2, j')}2^{i}2^{-(S-1)|i-j|}.
    \]
    So, if we take a fixed interior index $(j, k_t)$ with $2j' \leq j$, we see
    \begin{align}
        \sum_{\substack{(i, k_t') \\ i\leq j}}|M_{(j, k_t), (i, k_t')}^{j'}| &\leq C\sum_{i\leq j} 2^{|i-j|}2^{-\alpha j/2}2^{(\alpha - 2)\max(i/2, j')}2^{i}2^{-(S-1)|i-j|}\leq C\sum_{i\leq j} 2^{-(S-2)|i-j|} \leq C. \label{eq:schurtailcase1b}
    \end{align}
    Combining \eqref{eq:schurtailcase1a} and \eqref{eq:schurtailcase1b} together we get the needed bound. In the case where both $\psi^{\bc}_{jk_t}$ and $\psi^{\bc}_{ik_t'}$ are boundary wavelets (note that for each $i$ and $j$ there are only $2S^2 = O(1)$ such pairs), this boundary term is unavoidable. But in this case we use an estimate leveraging the weighting factors of in $M$.
Now, consider the case where $(j, k_t)$ is fixed to be a boundary index (say the left boundary for simplicity) and $i\leq j$ with $(i, k_t')$ also a boundary index. Then, we still integrate by parts and obtain
    \begin{align*}
        |\langle \partial_t \psi^{\bc}_{j k_t}, \psi^{\bc}_{ik_t'}\rangle| &\leq |\psi^{\bc}_{jk_t}(0)\psi^{\bc}_{ik_t'}(0)| + |\langle \psi^{\bc}_{jk_t}, \partial_t \psi^{\bc}_{ik_t'}\rangle| \leq C(2^{(i+j)/2} + 2^{i}2^{-(S-1)|i-j|}).
    \end{align*}
    Now the second term is well-behaved as we've seen, so it remains to bound
    \begin{align*}
        \sum_{i\leq j} 2^{-\alpha j/2}&2^{(\alpha - 2)\max(i/2, j')}2^{(i+j)/2} = 2^{-(\alpha -1)j/2}\sum_{i\leq j}2^{(\alpha - 2)\max(i/2, j')}2^{i/2} \\
                                                                               &\leq 2^{-(\alpha - 1)j/2}\left(\sum_{i\leq 2j'}2^{(\alpha -2)j'}2^{i/2} + \sum_{2j' < i \leq j}2^{(\alpha - 1)i/2}\right) \leq C,
    \end{align*}
    as required. This handles shows the maximum row sum is bounded in the case where $(j, k_t)$ is such that $2j'\leq j$. The remaining cases, i.e. 
    \begin{itemize}
        \item the maximum row sum is bounded when $(j, k_t)$ is such that $j \leq 2j'$,
        \item the maximum column sum is bounded when $(i, k_t')$ is such that $2j' \leq i$, and
        \item the maximum column sum is bounded when $(i, k_t')$ is such that $2j' \geq i$,
    \end{itemize}
    follow from the same techniques, and we won't do them here. Note that we require $\alpha \geq 2$ because in the case where $(i, k_t')$ is fixed, and $2j' \leq i \leq j$, we need the sum (coming from the boundary wavelets)
    \[
        \sum_{i\leq j} 2^{-\alpha j/2}2^{(\alpha -2)i/2}2^{(i+j)/2} = \sum_{i \leq j}2^{-(\alpha - 1)(j-i)/2}
    \]
    to converge (so indeed any $\alpha > 1$ would do).
\end{proof}

We now prove the main claim, Proposition \ref{prop:besovIsSobolev}.
\begin{proposition*}[\textbf{\ref{prop:besovIsSobolev}}]
    Let $m\in \N$, then 
    \begin{align*}
        \bigcap_{k=0}^{m} H^k([0, 1], H^{2m-2k}(\Omega)) &= H^{2m, \an},\\
        \bigcap_{k = 0}^{m+1} H^k([0, 1], H^{2m+1 - 2k}(\Omega)) &\subseteq H^{2m+1, \an},
    \end{align*}
    where $\subseteq$ indicates a continuous embedding, and equality indicates equivalent norms. In particular, we have $H^{s, \an} = H^{s, s/2}$ for $s\geq 0$.
\end{proposition*}
\begin{remark}
    We note that one containment being continuous, and the equality as sets automatically implies the norms are equivalent by standard functional analysis. See Remark 4.3.1 in \cite{GN}. 
\end{remark}
\begin{proof}
    Firstly, let $X = \bigcap H^k([0, 1], H^{2m-2k}(\Omega))$, and let $Q = [0, 1]\times \Omega$.  We will first show that $H^{2m, \an} \subseteq X$, and then $X\subseteq H^{2m, \an}$.

    Let $f \in H^{2m, \an}$. Then to show that $f \in X$ it suffices to show that $D_t^bD_x^\beta f \in L^2(Q)$, where $2b + |\beta| \leq 2m$. In our setting, what we mean by this is that for any function $\chi \in C_c((0, 1)\times \Omega)$, we have
    \[
        \int_{Q} D^b_tD^\beta_x f (t, x) \chi(t, x)\, dtdx = (-1)^{b+ |\beta|}\int_{Q}f(t, x) D^b_tD^\beta_x \chi (t, x)\, dtdx.
    \]
    Then, we use Lemma \ref{lem:rieszBlocks} to see that indeed $f\in X$. 
    
    Next, let $f \in X$. Then we wish to show that $f \in H^{2m, \an}$, i.e.
    \begin{equation}\label{eq:besovNormToBound}
        \sum_{j > L_0} 2^{2j(2m)}\sum_{(\iota, l), k} |\langle f, \Psi_{lk}^{\iota, \alt}\rangle|^2 < \infty
    \end{equation}
    Now, since $f \in L^2([0, 1], H^{2m}(\Omega))$, we know that $f(t, \cdot) \in H^{2m}(\Omega)$ for almost every $t$. Thus, by the wavelet characterization of $H^{2m}(\Omega)$ (see \textsection 4.3.6 in \cite{GN}), we get the following. See Theorem 2.55 in \cite{T2} for the corresponding wavelet characterization of $H^{m}([0, 1])$.
    \[
        \langle f(t, \cdot), \Psi^{\iota_x}_{j k_x}\rangle = 2^{-j(2m)}a^{\iota_x}_{jk_x}(t)
    \]
    for some sequence of functions $a^{\iota_x}_{jk_x}$ such that 
    \[
        \sum_{j = 0}^\infty \sum_{\substack{\iota_x\in \mathcal I\\ k_x\in [2^{j}-1]^d}} |a^{\iota_x}_{j k_x}(t)|^2 < \infty \quad \text{ for almost every }t.
    \]
    But we again observe from the wavelet characterization that 
    \begin{align*}
        C \|f\|_{L^2([0, 1], H^{2m}(\Omega))}^2 &\geq \int_{[0, 1]} C\|f(t, \cdot)\|^2_{H^{2m}(\Omega)} \, dt \geq \int_{[0, 1]}\sum_{j = 0}^\infty \sum_{\substack{\iota_x\in\mathcal I \\ k_x \in [2^j-1]^d}} |a^{\iota_x}_{jk_x}(t)|^2 \, dt\\
                                                  &\geq \sum_{j = 0}^\infty \sum_{\substack{\iota_x\in\mathcal I \\ k_x \in [2^j-1]^d}} \int_{[0, 1]}|a^{\iota_x}_{jk_x}(t)|^2 \, dt.
    \end{align*}
Thus, if we now consider
\begin{align*}
    \sum_{j >L_0} \sum_{\substack{(\iota, l) \in \mathcal I^{j, \an}\\ k \in \mathcal K(l)\\ \iota_x \neq 0}} 2^{2j(2m)} |\langle f, \Psi^{\iota, \alt}_{lk} \rangle |^2 &= \sum_{j >L_0} \sum_{\substack{(\iota, l) \in \mathcal I^{j, \an}\\ k \in \mathcal K(l)\\ \iota_x \neq 0}} 2^{2j(2m)} |\langle \langle f, \Psi^{\iota_x}_{l_1k_x}\rangle, \psi^{\iota_t}_{l_0k_t} \rangle |^2 \\
    &= \sum_{j >L_0} \sum_{\substack{(\iota, l) \in \mathcal I^{j, \an}\\ k \in \mathcal K(l)\\ \iota_x \neq 0}} |\langle a^{\iota_x}_{l_1k_x}, \psi^{\iota_t}_{l_0k_t} \rangle |^2.\\
\end{align*}
\vspace{-15mm}

But for a fixed $(\iota_t, l_0)$, the collection $(\psi^{\iota_t, \bc}_{l_0k_t})$ for $k_t \in [2^{l_0}-1]$ is orthonormal, and there are only three possibilities for $(\iota_t, l_0)$, namely $(0, 2(j-1)$, $(1, 2(j-1))$, or $(1, 2j-1)$, so
\begin{align*}
    \sum_{j >L_0} \sum_{\substack{(\iota, l) \in \mathcal I^{j, \an}\\ k \in \mathcal K(l)\\ \iota_x \neq 0}} 2^{2j(2m)} |\langle f, \Psi^{\iota, \alt}_{lk} \rangle |^2 \leq 3 \sum_{j >L_0} \sum_{\substack{\iota_x \in \mathcal I\\ k_x \in [2^{j-1}-1]}} \|a^{\iota_x}_{j-1, k_x}\|^2_{L^2([0, 1])}\leq 3C\|f\|^2_{L^2([0, 1], H^{2m}(\Omega))}.
\end{align*}
This is nearly the complete Besov space norm. However, we are missing the summands where $(\iota, l)\in \mathcal I^{j, \an}$ satisfies $\iota_x = 0$. Note that in such a case, $\iota_t = 1$. To handle this, we will use the wavelet characterization of Sobolev spaces on the interval, see Theorem 2.55 in \cite{T2}. We know that $f \in H^{m}([0, 1], L^2(\Omega))$, and again from Fubini and the fact wavelets are $L^\infty$ functions, we deduce that $\langle f, \Psi^{\iota_x}_{j k_x}\rangle_\Omega \in H^{m}([0, 1])$, with 
    \[
        D^b_t \langle f, \Psi^{\iota_x}_{j k_x} \rangle_\Omega = \langle D^b_t f, \Psi^{\iota_x}_{j k_x} \rangle_\Omega
    \]
    Now, if we define the functions 
    \[
        a^{\iota_x}_{jk_x}(t) = \langle f, \Psi^{\iota_x}_{jk_x}\rangle,
    \]
    then as stated earlier, $a^{\iota_x}_{j k_x} \in H^m([0, 1])$, and moreover for $b \leq m$,
    \begin{align*}
        \sum_{j = 0}^\infty \sum_{\substack{ \iota_x \in \mathcal I \\ k_x \in [2^j-1]^d}} \| D^b_t a^{\iota_x}_{j k_x}\|^2_{L^2([0, 1])} &= \sum_{j = 0}^\infty \sum_{\substack{\iota_x \in \mathcal I \\ k_x \in [2^j-1]^d}} \int_0^1 |\langle D^b_t f, \Psi_{jk_x}^{\iota_x} \rangle_\Omega |^2 \, dt = \int_0^1 \| D^b_t f \|^2_{L^2(\Omega)} \, dt.
    \end{align*}
    Hence, we deduce
    \[
        \sum_{j = 0}^\infty \sum_{\substack{ \iota_x \in \mathcal I \\ k_x \in [2^j-1]^d}} \|a^{\iota_x}_{j k_x}\|^2_{H^{m}([0, 1])} = \|f\|^2_{H^{m}([0, 1], L^2(\Omega))} < \infty.
    \]
    Next, note that we are trying to show that the terms in the Besov norm with $\iota = (1, 0)$ are finite, i.e.
    \begin{align*}
        \sum_{j >L_0} \sum_{k_x \in [2^{j-1}-1]^d} 2^{2j(2m)}\Big( &\sum_{k_t \in [2^{2(j-1)}-1]}|\langle f, \Psi^{(1, 0), \alt}_{(2(j-1), j-1), (k_t, k_x)}\rangle|^2  \\
                                                                     &+\quad \sum_{k_t \in [2^{2j-1}-1]} |\langle f, \Psi^{(0,1), \alt}_{(2j-1, j-1), (k_t, k_x)}\rangle|^2\Big) \\
            = \sum_{j >L_0} \sum_{k_x \in [2^{j-1}-1]^d} 2^{2j(2m)}\Big( &\sum_{k_t \in [2^{2(j-1)}-1]}|\langle \langle f, \psi^{bc}_{2(j-1), k_t}\rangle, \Psi^{0}_{j-1, k_x}\rangle|^2  \\
                                                                     &+\quad \sum_{k_t \in [2^{2j-1}-1]} |\langle \langle f, \psi^{bc}_{2j-1, k_t}\rangle, \Psi^{0}_{j-1, k_x}\rangle|^2 \Big).
    \end{align*}
    But, here we note that the collection $(\Psi^0_{j-1, k_x})$ for $k_x \in [2^{j-1}-1]^d$ is an orthonormal basis of $\mathbf V_{j-1}$ from \textsection \ref{sec:WaveletFacts}, so the above sum equals
    \begin{align*}
        =\sum_{j >L_0} \sum_{l \leq j-1}\sum_{\substack{\iota_x \in \mathcal I \\ k_x \in [2^{l}-1]^d}} 2^{2j(2m)}\Big( &\sum_{k_t \in [2^{2(j-1)}-1]}|\langle a^{\iota_x}_{lk_x}, \psi^{bc}_{2(j-1), k_t}\rangle|^2  \\
                                                                                                                          &+\quad \sum_{k_t \in [2^{2j-1}-1]} |\langle a^{\iota_x}_{lk_x}, \psi^{\bc}_{2j-1, k_t}\rangle|^2 \Big) \\
        \leq  \sum_{l =0}^\infty\sum_{\substack{\iota_x \in \mathcal I \\ k_x \in [2^{l}-1]^d}} \sum_{j >L_0}2^{2j(2m)}\Big( &\sum_{k_t \in [2^{2(j-1)}-1]}|\langle a^{\iota_x}_{lk_x}, \psi^{bc}_{2(j-1), k_t}\rangle|^2  \\
                                                                                                                          &+\quad \sum_{k_t \in [2^{2j-1}-1]} |\langle a^{\iota_x}_{lk_x}, \psi^{\bc}_{2j-1, k_t}\rangle|^2 \Big) \\
            \leq C \sum_{l = 0}^\infty \sum_{\substack{\iota_x \in \mathcal I \\ k_x \in [2^l-1]^d}} \| a^{\iota_x}_{lk_x}\|^2_{H^{m}([0, 1])} &= C\|f\|^2_{H^m([0, 1], L^2(\Omega))},
    \end{align*}
    where the last inequality follows from the wavelet characterization of the space $H^{m}([0, 1])$. Thus, we have shown that $f \in H^{2m, \an}$, (and $\|\cdot\|_{H^{2m, \an}} \lesssim \|\cdot \|_X$), so we are done. The argument is the same for the $2m+1$ case. The in particular part follows from an interpolation argument. For any $s\geq 0$, let $m\in \N$ be such that $2m > s$. Then, from Proposition 2.1 in \cite{LM2} we have for $2m(1-\theta) = s$
    \[
        [H^{2m, m}, L^2]_{\theta} = H^{s, s/2}.
    \]
    But we also have
    \[
        [H^{2m, \an}, L^2]_{\theta} = H^{s, \an},
    \]
    so indeed these spaces must coincide. To see the interpolation result for the wavelet spaces, consider the definition of interpolation given in \cite{LM} Definition 2.1: The $\Lambda$ operator is given by
    \[
        \Lambda(u) = \sum_{j \geq L_0} \sum_{(\iota, l), k} 2^{2mj}\langle u, \Psi^{\iota, \alt}_{lk}\rangle\Psi^{\iota, \alt}_{lk}.
    \]
    It is diagonal with respect to the anisotropic basis, and so
    \[
        \Lambda^{1-\theta}(u) =\sum_{j \geq L_0} \sum_{(\iota, l), k} 2^{2m(1-\theta)j}\langle u, \Psi^{\iota, \alt}_{lk}\rangle\Psi^{\iota, \alt}_{lk}.
    \]
\end{proof}

\section{PDE Regularity}\label{sec:PDEReg}
\subsection{Heterogeneous Reaction-Diffusion Equation}\label{sec:HRDERegTheory}

Consider solution $u: Q\to \R$ to the equation
\begin{align}
    \frac{\partial u}{\partial t} - \nabla \cdot (D \nabla u) &= f(u) \quad \text{ on } (0, 1]\times \Omega, \label{eq:HRDE}\\
    u(0, \cdot) &= \theta \quad \text{ on } \Omega, \nonumber
\end{align}
where $f: \R \to \R$ is in $C^\infty_c(\R)$, and $\theta \in H^1(\Omega)$ is an initial condition. We moreover have $D : \Omega \to [D_{\min}, \infty)$, with $D_{\min} >0$, and $\|D\|_\infty < \infty$. A weak solution to \eqref{eq:HRDE} is a map $u \in L^2([0, 1], H^1(\Omega))\cap C([0, 1], L^2(\Omega)$ such that $u' \in L^2([0, 1], H^{-1}(\Omega))$ satisfies
\begin{equation}
    \langle u', v\rangle + \langle D\nabla u, \nabla v \rangle = \langle f(u), v\rangle, \quad \forall v \in H^1(\Omega)\text{ and a.e. } t\in [0, 1].
\end{equation}
A strong solution to \eqref{eq:HRDE} is a weak solution $u$ with, in addition, 
\[
    u \in L^2([0, 1], H^2(\Omega)) \cap C([0, 1], H^1(\Omega))\qquad u' \in L^2([0, 1], L^2(\Omega)), 
\]
in which case $u$ solves \eqref{eq:HRDE} in $L^2([0, 1], L^2(\Omega))$.

\subsubsection{The Transport operator $\nabla \cdot(D\nabla \cdot)$}
In proving results involving a variable diffusivity term, we first need to record some facts about the operator $\nabla \cdot (D\nabla\cdot)$. Throughout, we will assume that $D\in C^1(\Omega)$, $D: \Omega \to [D_{\min}, \infty)$ with $D_{\min} >0$.

Recall that for any function space $X \subseteq L^1(\Omega)$, we define
\[
    \dot X = \bigg\{ f \in X : \int_\Omega f(x) \, dx = 0\bigg\}.
\]
Moreover, for Sobolev spaces we define $H^{-s}(\Omega) = (\dot H^s(\Omega))^\ast$.

Then, following arguments similar to \textsection 3.1.2 in \cite{N24}, we get that there exists an orthonormal basis of $L^2(\Omega)$ of eigenfunctions $e_{0, D} = 1$, and $e_{1, D}, \dots, e_{j, D}, \dots \in H^2(\Omega)$ with corresponding Eigenvalues $\lambda_0 = 0 < \lambda_{1, D} \leq \lambda_{2, D}\leq \cdots \uparrow \infty$ such that
\[
    -T_De_{j, D} = \nabla \cdot (D \nabla e_{j, D}) = -\lambda_{j, D}e_{j, D}\quad \text{ for }j \in \N.
\]
We can define spaces, for $k \in \N$
\[
    \bar H^k_D(\Omega) = \left\{ \phi \in \dot L^2(\Omega) : \sum_{j \geq 1} \lambda^k_{j, D} \langle \phi, e_{j, D}\rangle^2_{L^2} \equiv \|\phi\|_{\bar H^k}^2 < \infty\right\}.
\]
They satisfy properties as stated in the following theorem:

\begin{proposition}\label{prop:EigenSobolevChar}
    Assume that $\| D \|_{H^s(\Omega)} \leq U$ for some $s > d$. Then we have for $k \leq s+1$, 
    \[
        \bar H^k_{D}(\Omega) = \dot H^k(\Omega) \quad \text{ and } \| \phi\|_{H^k} \simeq \|\phi\|_{\bar H^k_D}.
    \]
\end{proposition}
We will not prove this proposition here, but note the reason we get equality of the spectrally defined spaces (and not a strict containment) is because we are working with periodic boundary conditions and not with Neumann or Dirichlet boundary conditions. The is fundamentally similar to how the domain of the Laplacian coincides with Sobolev spaces in the periodic case, see \textsection 6.3 in \cite{R}.

This allows us to define the operator $T_D^{1/2}$ via
\[
    T_D^{1/2}(u) = \sum_{j \geq 1} \lambda_{j, D}^{1/2} \langle u, e_{j, D} \rangle e_{j, D}, 
\]
which is an operator from $H^k(\Omega) \to H^{k-1}(\Omega)$ for all $k \leq s + 1$. Finally, we note that for $\phi \in H^k$, 
\[
    \| T^{k/2}_D\phi \|_{L^2(\Omega)}^2 + \|\varphi \|_{L^2(\Omega)}^2 \simeq \|\varphi \|_{H^k(\Omega)}^2.
\]
Now, we consider the linear parabolic equation
\begin{equation}\label{eq:LinearParabolicEq}
\begin{split}
    \frac{\partial u}{\partial t} + T_D u + au&= h \quad \text{ on } [0, S]\times \Omega, \\
    u(0) &= \theta \quad \text{ on } \Omega,
\end{split}
\end{equation}
where $h \in L^2([0, S], L^2(\Omega))$ and $a \in L^\infty(Q)$. Following the same proofs as in \textsection 6.3 and \textsection 7.2 in \cite{Evans}, but keeping track of the exact assumptions needed on the coefficients, we obtain the following theorem:
\begin{proposition}\label{prop:LinearParabolicRegularity}
    Suppose $\theta \in H^1(\Omega)$, $h \in L^2([0, S], L^2(\Omega))$, $a\in L^\infty(Q)$, and $D: \Omega 
    \to [D_{\min}, \infty)$, $D_{\min} >0$, $D \in C^1(\Omega)$. Then there exists a unique strong solution of \eqref{eq:LinearParabolicEq} $u$ satisfying
    \[
        u \in L^2([0, S], H^2(\Omega))\cap C([0, S], H^1(\Omega)), \quad u' \in L^2([0, S), L^2(\Omega)), 
    \]
    and we have the estimates
    \[
        \|u\|_{L^2([0, S], H^2(\Omega)} + \|u'\|_{L^2([0, S], L^2(\Omega))}\leq C (\| h\|_{L^2([0, S], L^2(\Omega))} + \| \theta\|_{H^1(\Omega)})
    \]
    and
    \[
        \|u\|_{L^2([0, S], H^1(\Omega)} + \|u'\|_{L^2([0, S], H^{-1}(\Omega))}\leq C (\| h\|_{L^2([0, S], H^{-1}(\Omega))} + \| \theta\|_{L^2(\Omega)}),
    \]
    where $C$ depends only on an upper bound on the $C^1$-norm of $D$, $D_{\min}$, $d$, $\|a\|_\infty$, and $S$.
\end{proposition}
\subsubsection{Regularity Theory for HRDE}

\begin{theorem}\label{thm:strongsolsHRDE}
    For $f\in C_c^\infty(\R)$, $D \in C^1(\Omega)$, and $\theta \in H^1(\Omega)$, $d \in \N$, there exists a strong solution $u = u_{\theta, D}$ to the heterogeneous reaction-diffusion equation \eqref{eq:HRDE} that is unique in $C([0, 1], L^2(\Omega))$.
\end{theorem}
\begin{proof}
    This proof of existence adapts ideas from \textsection 3.1.1 in \cite{N26+} and p.537f in \cite{Evans}.
    Define the Banach space $X \equiv C([0, S], L^2(\Omega))$, for some $S >0$ to be chosen, and imbue it with the norm
    \[
        \| u\|_X = \sup_{0 < t < S} \| u(t) \|_{L^2(\Omega)}.
    \]
    For any $u \in X$ define a new function $h(t, \cdot) = f(u(t, \cdot)) \in L^\infty([0, S], L^\infty(\Omega)$. The \textit{linear} parabolic PDE
    \begin{align*}
        \frac{\partial w}{\partial t} + T_Dw &= h \quad \text{ on } (0, S]\times \Omega, \\
        w(0, \cdot) &= \theta \quad \text{ on }\Omega,
    \end{align*}
    has a unique strong solution $w \in L^2([0, S], H^2(\Omega))\cap C([0, S], H^1(\Omega))$ with $w' \in L^2([0, S], L^2(\Omega))$ by Proposition \ref{prop:LinearParabolicRegularity}. For $u_i \in X$, $h_i = f(u_i)$, $i = 1, 2$, consider solutions $w_i \equiv A(u_i) \in X$ to \eqref{eq:LinearParabolicEq} and define the constant
    \[
        i_{1, 2} = \int_\Omega (w_1 - w_2) = \langle w_1 - w_2, 1\rangle_{L^2(\Omega)}.
    \]
    We can test the equation \eqref{eq:LinearParabolicEq} against $1$ and get that
    \[
        \frac{d}{dt}\langle w_1 - w_2, 1 \rangle + \langle T_D (w_1 - w_2), 1\rangle = \langle f(u_1) - f(u_2), 1\rangle.
    \]
    Now, we can use divergence theorem to show the middle term vanishes. If we integrate from $0$ to $t$, we see
    \begin{align*}
        |i_{1, 2}(t)| = |i_{1, 2}(t) - i_{1, 2}(0)| = \left|\int_0^t\int_\Omega (f(u_1(s, x)) - f(u_2(s, x))) \, dxds\right| \\
        \leq S\sup_{0 < t < S} \| f(u_1) - f(u_2) \|_{L^2(\Omega)} \leq S\| f\|_{\lip}\|u_1 - u_2 \|_X.
    \end{align*}
    Testing \eqref{eq:LinearParabolicEq} for $w_1$ and $w_2$ against $w_1 - w_2$, and then subtracting them, we get
    \begin{align*}
        \frac{d}{dt}\|w_1 - w_2\|^2_{L^2(\Omega)} &+ 2D_{\min} \| \nabla (w_1 - w_2 - i_{1, 2}) \|_{L^2}^2 \\
                                                  &\leq \frac{d}{dt}\|w_1 - w_2\|^2_{L^2(\Omega)} + 2\| \sqrt D \nabla(w_1-w_2-i_{1, 2})\|^2_{L^2} \\
                                                  &= 2 \langle w_1 - w_2, h_1 - h_2\rangle_{L^2(\Omega)}\\
                                                  & \leq (\epsilon/2)\|w_1 - w_2 - i_{1, 2} + i_{1, 2} \|^2_{L^2(\Omega)} + \frac{2}{\epsilon} \| f(u_1) - f(u_2) \|^2_{L^2(\Omega)} \\
                                                  &\leq \epsilon \| \nabla(w_1 - w_2 - i_{1, 2}) \|^2_{L^2(\Omega)} + (S^2 + \frac{2}{\epsilon})\|f\|_{\lip}^2\|u_1 - u_2\|_X^2.
    \end{align*}
    Then we can choose $\epsilon$ small enough and subtracting, we get
    \[
        \frac{d}{dt}\| w_1 - w_2 \|_{L^2(\Omega)}^2 \leq c(S^2 + 1) \| u_1 - u_2 \|_X^2,
    \]
    for some constant $c = c(\|f\|_{\lip}, D_{\min}, d)$. Integrating from $0$ to $t$, and using the fact that $(w_1 - w_2)(0) = 0$, we get
    \[
        \| A(u_1) - A(u_2) \|_X^2 = \sup_{0 < t < S} \| w_1(t) - w_2(t) \|^2_{L^2(\Omega)} \leq c(S^2 + 1)S \|u_1 - u_2 \|_X^2.
    \]
    We deduce for some fixed $S_0$ small enough that $A$ is a contraction mapping on $X$, so by the contraction mapping theorem, there exists a unique $u$ satisfying $u = A(u)$ on the interval $[0, S_0]$. Since $u$ is a strong solution to \eqref{eq:HRDE} if and only $u = A(u)$, we obtain the desired result.
\end{proof}

First, we prove spatial regularity of solutions:

\begin{proposition}\label{prop:VVSpatialReg}
    Assume that $D\in C^1(\Omega)$. For $\theta \in H^1(\Omega)$, $D \in H^s(\Omega)$,  and $f \in C^\infty_c(\R)$, let $u$ be a solution of \eqref{eq:HRDE}. Suppose that $\| \theta \|_{H^\gamma(\Omega)}, \|D\|_{H^\gamma(\Omega)} \leq B$ for $0\leq \gamma \leq s$ and assume that $d \leq 3$. Then we have
    \[
        \int_0^S \|u_{\theta, D}(t)\|^2_{H^{\gamma+1}(\Omega)} \, dt + \sup_{0\leq t \leq S} \| u_{\theta, D} \|_{H^{\gamma}(\Omega)} \leq C, 
    \]
    for some $C = C(B, f, \gamma, d, D_{\min}, \|D\|_{C^1}, S)$, and writing $u' = (\partial u/\partial t)$, we also have
    \[
        \int_0^S \|u'_{\theta, D}(t) \|^2_{H^{\gamma -1 }} \, dt + \ess \sup_{0 \leq t \leq S} \|u'_{\theta, D}(t) \|_{H^{\gamma - 2}} \leq C'.
    \]
    In particular, we have $u_{\theta, D} \in C([0, S], H^\gamma).$
\end{proposition}
\begin{proof}
    We proceed by the usual arguments, where we assume that the solutions $u$ are sufficiently regular and establish a-priori estimates. The general case follows from going back to a Galerkin approximation and then using uniqueness of strong solutions, as is standard.

    \textbf{Step 1: Preliminary identities and estimates.} Start with equation \eqref{eq:HRDE}, and test it against $T_D^{\gamma} u$. Then, we get
    \[
        \frac{1}{2}\frac{d}{dt} \langle u, T_D^{\gamma} u \rangle + \langle T_D u, T_D^{\gamma} u \rangle = \langle f(u), T_D^{\gamma} u \rangle.
    \]
    Using the fact that $T_D$ is symmetric, we observe that
    \begin{equation}\label{eq:Generalgammacase}
        \frac{1}{2}\frac{d}{dt} \| T_D^{\gamma/2} u \|^2_{L^2} + \| T_D^{(\gamma + 1)/2} u \|^2_{L^2}= \langle T_D^{(\gamma -1)/2}f(u), T_D^{(\gamma + 1)/2} u \rangle.
    \end{equation}
    For the case $\gamma = 0$, we get
    \begin{equation}\label{eq:0gammacase}
        \frac{1}{2}\frac{d}{dt} \|u(t)\|_{L^2(\Omega)}^2 + \| \sqrt D \nabla u \|_{L^2(\Omega)}^2 = \langle f(u), u \rangle.
    \end{equation}
    Since $f(u)$ is uniformly bounded and of compact support, we know that
    \[
        \langle f(u), u \rangle_{L^2} = \int_\Omega f(u) u \mathbbm{1}_{|u|\leq K} \leq c_f < \infty, \quad \text{ for some }K > 0. 
    \]
    So, then going back to \eqref{eq:0gammacase}, we get
    \[
        \frac{d}{dt} \| u(t) \|_{L^2}^2 + 2 \| \sqrt D \nabla u \|_{L^2}^2 \leq 2c_f,
    \]
    which by integrating we obtain $\sup_{0 \leq t \leq S} \| u(t) \|_{L^2}^2 \leq 2Sc_f + \|\theta\|_{L^2}^2< \infty$ as well as 
    \[
        \int_0^S \| \nabla u \|^2_{L^2} \, dt \leq \frac{Sc_f}{D_{\min}} + \frac{1}{D_{\min}}\|u(0)\|^2_{L^2} \leq C(B, f, S, D_{\min}).
    \]
    Since $\|u\|^2_{H^1} \simeq \| u \|^2_{L^2} + \| \nabla u \|^2_{L^2}$ this proves the proposition in the case $\gamma = 0$.

    Next, we note that from the Cauchy-Schwarz inequality, Proposition \ref{prop:EigenSobolevChar}, and Young's inequality
    \[
        \langle T_D^{\gamma - 1/2} u, T_D^{\gamma + 1/2} u \rangle \lesssim \|f(u)\|_{H^{\gamma - 1}} \|T_D^{\gamma+1/2}u\|_{L^2} \leq 2\|f(u) \|_{H^{\gamma - 1}}^2+ \frac{1}{4}\| T^{\gamma + 1/2}_D u \|_{L^2}^2.
    \]
    Thus, returning to \eqref{eq:Generalgammacase}, we see
    \begin{equation}\label{eq:GronwallTypeIneq}
        \frac{d}{dt} \|T_D^{\gamma/2}u\|_{L^2}^2 + \|T^{(\gamma + 1)/2}_Du\|^2_{L^2} \lesssim \| f(u) \|_{H^{\gamma - 1}}^2.
    \end{equation}
    \textbf{Step 2: Proof of a-priori estimates.} We now want to prove the estimates in the proposition for $\gamma \geq 1$. We argue by induction. Suppose that
    \[
        \int_0^S \|u_{\theta, D}(t)\|^2_{H^{\gamma}(\Omega)} \, dt + \sup_{0\leq t \leq S} \| u_{\theta, D} \|_{H^{\gamma- 1}(\Omega)} \leq C, 
    \]
    In which case we have that
    \[
        \|u(t) \|_{H^{\gamma - 1}} \leq C,
    \]
    with constant $C = C(S, B, D_{\min}, \|D\|_{C^1}, d)$. We show that as a consequence that the RHS of \eqref{eq:GronwallTypeIneq} is bounded by a fixed constant. For $\gamma = 1$, this is follows from $f \in L^\infty$. For $\gamma = 2$, we know
    \[
        \|f(u)\|_{H^1} \lesssim \|f(u)\|_{L^2} + \max_{|\alpha| = 1}\|f'(u)D^\alpha u \|_{L^2} \leq \|f\|_\infty + \|f'\|_\infty\|u\|_{H^1} \leq C.
    \]
    For $\gamma = 3$, we have again that $\|f(u) \|_{H^2} \lesssim \|f\|_\infty + \max_{|\alpha| = 2}\|D^\alpha(f(u))\|_{L^2}$ and need to estimate the last term. We know that
    \[
        D^\alpha (f(u)) = f''(u)D^{\alpha_1} uD^{\alpha_2} u + f'(u)D^\alpha u.
    \]
    The second term is easily bounded by $\|f'\|_\infty \|u\|_{H^2}$, and the first is bounded by
    \begin{equation*}
        \|f''\|_\infty \bigg(\int_\Omega (D^{\alpha_1}u D^{\alpha_2}u)^2\, dx\bigg)^{1/2} \lesssim \| D^{\alpha_1}u \|_{L^4} \|D^{\alpha_2} u \|_{L^4},
    \end{equation*}
    and we note that when $d\leq 4$, that $H^1(\Omega) \hookrightarrow L^4(\Omega)$, and so we get that the above is bounded by $\|u\|_{H^2}^2$. Next, the restriction $d \leq 3$ is required. We have the multiplier inequality for Sobolev norms is
    \[
        \|fg\|_{H^a} \lesssim \|f\|_{H^a} \|g\|_{H^a}
    \]
    for $a > d/2$. Thus, we can apply it in the case of $H^2$-norms when $d\leq 3$. Now, for $\gamma = 4$, we have 
    \begin{align*}
        \|f(u) \|_{H^3} &\lesssim c + \max_{|\alpha| = 1}\|f'(u)D^\alpha u \|_{H^2} \lesssim c + \|f'(u)\|_{H^2}\max_{|\alpha | = 1} \|D^\alpha u \|_{H^2} \\
                        &\lesssim c + \|u\|_{H^3}^3,
    \end{align*}
    using the same bound we did in the previous case but now with $f'''$ replacing $f''$. For further $\gamma \geq 5$, we do the same argument as above, utilizing the multiplier inequality for $H^{\gamma - 2}$-norms.

    Thus, we have completed the induction, and we have shown that for all $\gamma$, there exists and constant $c_0$ (depending on the usual constants) such that 
    \[
        \frac{d}{dt} \| T_D^{\gamma} u \|^2_{L^2} + \|T_D^{\gamma + 1/2} u \|^2_{L^2} \leq c_0.
    \]
    To prove the a-priori estimates, we now simply do the usual integration argument, and get that 
    \[
        \|T^\gamma_D u(t)\|_{L^2}^2 + \int_0^S \|T_D^{\gamma + 1/2} u \|_{L^2}^2 \, dt \leq c_1(B, f, \gamma, S, d, D_{\min}, \|D\|_{C^1}).
    \]
    Now, we combine this with the bounds on $\|u(t) \|_{L^2}$, and use Proposition \ref{prop:EigenSobolevChar} to deduce the desired a-priori bound. The derivative a-priori bound follows from the fact
    \[
        \left\|\frac{du}{dt}\right\|_{H^{\gamma-1}}^2 = \| T_D u + f(u)\|_{H^{\gamma-1}}^2 \leq 2\|u\|_{H^{\gamma + 1}}^2 + 2\|f(u)\|^2_{H^{\gamma - 1}}.
    \]
    Now that we have deduced the a-priori estimates, we employ the usual Galerkin arguments with respect to the basis $\{e_{j, D}\}_{j\geq 0}$. Note that basis is required as we need for $T_D$ to preserve $E_J = \Span \{ e_{j, D}\}_{j = 0}^J$.
\end{proof}

Now we deduce the temporal regularity of $u$. We will again use a Galerkin argument.
\begin{proof}[Proof of Proposition \ref{prop:HRDEReg}]
    We will show via induction on $k$ that for $\gamma + 1 - 2k \geq -1$, 
    \begin{align*}
        &\frac{d^j u}{dt^j} \in L^2([0, S], H^{\gamma + 1 - 2j}(\Omega))\quad \text{ for } j = 0, \dots, k,\\
        &\left\| u^{(j)} \right\|_{L^2([0, S], H^{\gamma +1 - 2j}(\Omega))} \leq C,
    \end{align*}
    where $C = C(B, f, \gamma, d, D_{min}, \|D\|_{C^1}, S, k)$. We have proven the base case of $k = 1$ by Proposition \ref{prop:VVSpatialReg}. Now, using the standard Galerkin argument, it suffices to establish an a-priori bounds on $u^{(k+1)}$. We know that $u$ satisfies 
    \[
        \frac{ d u}{dt} = \nabla \cdot (D\nabla u ) + f(u), 
    \]
    so by differentiating this identity $k$-times w.r.t to time, we get
    \[
        \frac{d^{k+1} u}{dt^{k+1}} = \nabla \cdot ( D \nabla \frac{d^k u}{dt^k}) + \frac{d^{k+1}}{dt^{k+1}} \big( f(u) \big).
    \]
    Now, we can bound 
    \[
        \|\nabla \cdot (D\nabla u^{(k)})\|_{L^2([0, S], H^{\gamma -1-2k})} \lesssim \| u^{(k)} \|_{L^2([0, S], H^{\gamma + 1 - 2k})},
    \]
    which is in turn bounded by some constant from the induction hypothesis. Now we bound the second term. Using Fa\`a di Bruno's formula, we see
    \begin{align*}
        \frac{d^{k+1}}{dt^{k+1}}\big[f(u)\big] &= \sum \frac{k!}{m_1!\cdots m_k!} f^{(m_1 + \cdots m_k)}(u)\prod_{j = 1}^k \left(\frac{u^{(j)}}{j!}\right)^{m_j}, \\
        &\quad \text{ where }m_1 + 2m_2 + \cdots + km_k = k.
    \end{align*}
    Thus, it suffices to bound each term of the sum in $L^2([0, S], H^{\gamma - 1 - 2k})$. Next, by the induction hypothesis we have
    \begin{equation}\label{eq:bestspaceforuj}
        u^{(j)} \in C([0, S], H^{\gamma - 2k}(\Omega)) \quad \text{ for }j = 0, \dots k-1,
    \end{equation}
    with its norm bounded by some constant times the norms on $u^{(j)}$ and $u^{(j+1)}$ (See theorem 4 of \textsection 5.9.2 in \cite{Evans}).

    Suppose now that $\gamma - 1 - 2k \geq 2$. Then, we know $H^{\gamma - 1 - 2k}$ satisfies a multiplier inequality for $d\leq 3$, so for some non-negative integers $m_1, \dots, m_k$, such that $m_1 + 2m_2 + \cdots + km_k = k$, we have
    \begin{align}
        &\int_0^S \|f^{(m_1 + \cdots m_k)}(u)\prod_{j = 1}^k (u^{(j)})^{m_j}\|_{H^{\gamma - 1 - 2k}}^2 \, dt \nonumber \\
        \lesssim &\int^S_0 \|f^{(m_1 + \cdots m_k)}(u)\|_{H^{\gamma - 1 - 2k}}^2\prod_{j=1}^k \|u^{(j)}\|^{2m_j}_{H^{\gamma - 1-2k}} \, dt.\label{eq:easycasegoal}
    \end{align}
    Now, in the last proposition's proof, we showed that if $\sup_{0\leq t\leq S} \|u\|_{H^a}\leq C$, then $\sup_{0\leq t \leq S}\|f(u)\|_{H^a}\leq c'$. For any choice of $a$ provided $d\leq 3$. But we know from \eqref{eq:bestspaceforuj} that $\sup_{0\leq t \leq S}\|u\|_{H^\gamma}\leq c$, so 
    \begin{equation}\label{eq:normboundonf}
        \sup_{0\leq t \leq S}\|f^{(m_1 + \cdots + m_k)}(u)\|_{H^{\gamma - 1 -2k}} \leq C.
    \end{equation}
    Similarly, we also have for each $j = 1, \dots k-1$,
    \[
        \sup_{0\leq t \leq S} \|u^{(j)}\|_{H^{\gamma - 1- 2k}} \lesssim \sup_{0\leq t \leq S} \|u^{(j)}\|_{H^{\gamma - 2j}} \leq C.
    \]
    Thus, we have established that the RHS of \eqref{eq:easycasegoal} is bounded by a constant in all cases except when $m_k \neq 0$. In this case, $m_1 = \cdots = m_{k-1} = 0$, and $m_k = 1$. Then, we get that (using again \eqref{eq:normboundonf})
    \begin{align*}
        \eqref{eq:easycasegoal} &\lesssim \int_0^S \|u^{(k)}\|_{H^{\gamma - 1 - 2k}}^2 \, dt \\
        &\leq  \|u^{(k)}\|_{L^2([0, S], H^{\gamma +1 - 2k})} \leq C,
    \end{align*}
    so, we have established the a-priori bound when $\gamma - 1 - 2k \geq 2$.

    Next, suppose that $\gamma - 1 - 2k = 0$. The cases $\gamma - 1 - 2k = 1, -1$ use the same ideas. Then, we want to show that for $m_1 + 2m_2 +\cdots + km_k = 1$, 
    \[
        \int_0^S \|f(u)\prod_{j = 1}^k(u^{(j)})^{m_j}\|^2_{L^2(\Omega)}\, dt\leq C.
    \]
    Since $f$ is bounded, we know
    \begin{equation}
        \int_0^S \|f(u)\prod_{j = 1}^k(u^{(j)})^{m_j}\|^2_{L^2(\Omega)}\, dt\leq C\int_0^S \int_\Omega \prod_{j = 1}^ku^{(j)}(t, x)^{2m_j}\, dxdt. \label{eq:fullexpansion}
    \end{equation}
    Now, since $m_1 + 2m_2 + \cdot + km_k = k$, we know
    \[
        \frac{1}{k/m_1} + \frac{1}{k/2m_2} +\cdots \frac{1}{m_k} = 1.
    \]
    Thus, by generalized H\"older's inequality, we get
    \[
        \eqref{eq:fullexpansion} \leq C\int_0^S\prod_{j = 1}^k \|u^{(j)}\|_{L^{2k/j}(\Omega)}^{2m_j}\, dt.
    \]
    Thus, utilizing \eqref{eq:bestspaceforuj}, we know that $\|u^{(j)}\|_{L^{2k/j}(\Omega)}$ is bounded uniformly in time by some constant provided that $H^{\gamma - 2j}$ embeds into $L^{2k/j}$. In this case, we know $\gamma = 1 + 2k$, so for $j \leq k -1$, $\gamma - 2j \geq 3$. Hence, we know that indeed $H^{\gamma - 2j}$ embeds into $L^\infty$. The case where $m_k \neq 0$ is again similar, so we are done when $\gamma = 1+2k$. Finally, to finish this proof use Proposition \ref{prop:besovIsSobolev}.
\end{proof}

\subsection{$2$-d Navier Stokes Equations}\label{sec:NSRegTheory}

In the setting of the $2$-d Navier-Stokes equations, we have
\begin{equation}\label{eq:NSAppendix}
    \begin{split}
        \frac{du}{dt} - \nu \Delta u + B(u, u) = f \quad \text{ on }[0, S]\times \Omega, \\
        u(0, \cdot) = \varphi,
    \end{split}
\end{equation}
where $f:\Omega \to \R^2$ is a time-independent forcing function, $\varphi \in \dot H^1_\diamond$, and $\nu \in (\nu_{\min}, \infty)$. A strong solution to \eqref{eq:NSAppendix} is some $u: [0, S]\to \dot H^1_\diamond(\Omega)$ such that
\[
    u \in L^2([0, S], \dot H^2_\diamond)\cap C([0, S], \dot H^1_\diamond), \quad u' \in L^2([0, S], \dot H_\diamond),
\]
and $u$ solves \eqref{eq:NSAppendix} for in $L^2([0, S], \dot H_\diamond)$. We know from Proposition $1$ of \cite{NT} that strong solutions to exist and are unique in 2-dimensions. 

\begin{proof}[Proof of Proposition \ref{prop:NSReg}]
    The proof proceeds similarly to \textsection \ref{sec:HRDERegTheory}, except we utilize Proposition 3.6 in \cite{KN}, and then differentiate in time as performed in Proposition \ref{prop:HRDEReg}. In this case the computations are easier, as instead of having to apply Fa\'a di Bruno's formula on the non-linearity $f(u)$, we simply have to use the product rule on the $(u\cdot \nabla)u$ term. 
\end{proof}
\subsubsection{Proof of Proposition \ref{prop:ForwardRegNS}}\label{sec:ForwardRegNS}
\begin{proof}[Proof of Proposition \ref{prop:ForwardRegNS}]
    The first part is shown follows from Proposition \ref{prop:NSReg} coupled with Proposition \ref{prop:BesovEmbeddings}. Next, let $M \geq 0$, and suppose that $(\varphi_1, \nu_1), (\varphi_2, \nu_2) \in \Theta$ with $\|\varphi_i\|, \nu_i \leq M$ for $i = 1, 2$. Also, let $u_i$ be the solution to \eqref{eq:introNS} with parameters $(\varphi_i, \nu_i)$. We want to show a forward regularity estimate of the form
    \[
        \sup_{0\leq t \leq 1} \|u_1 - u_2\|_{L^2} \leq K\big(\|\varphi_1-\varphi_2\| + |\nu_1 - \nu_2|\big).
    \]
    Set $w = u_1 - u_2$. Then, subtracting the difference of the two equations and rewriting the bilinear terms akin to p.250 in \cite{R}, we see that $w$ satisfies 
    \begin{align*}
        \frac{dw}{dt} + \nu_1 A w &= -B(w, u_1) - B(u_2, w) - (\nu_1 - \nu_2)Au_2, \\
        w(0) &= \varphi_1 - \varphi_2.
    \end{align*}
    Taking inner products with $w$, and noting $\langle B(u_2, w), w\rangle = 0$ (Proposition 9.1 in \cite{R}), we see that
    \[
        \frac{1}{2} \frac{d}{dt} \left[ \|w\|_{L^2}^2\right] + \nu_1\|A^{1/2}w\|_{L^2} = -\langle B(w, u_1), w\rangle - (\nu_1 - \nu_2)\langle Au_2, w\rangle.
    \]
    Then, we use (9.26) in \cite{R}, as well as Young's inequality and the divergence theorem to obtain
    \[
        \frac{1}{2} \frac{d}{dt} \left[ \|w\|_{L^2}^2\right] + \nu_1\|A^{1/2}w\|_{L^2}^2 \leq \frac{\nu_1}{2}\|A^{1/2}w\|^2_{L^2} + \frac{k^2}{\nu_1}\|w\|_{L^2}^2\|u_1\|_{L^2}^2 + |\nu_1 - \nu_2|^2 \frac{\|A^{1/2}u_2\|_{L^2}^2}{\nu_1}.
    \]
    Then we apply Gr\"onwall's inequality to obtain
    \[
        \|w(t)\|^2 \leq \exp\left(\int^t_0\frac{2k^2}{\nu_1}\|u_1(s)\|_{L^2}^2\, ds\right)\left(\|w(0)\|^2_{L^2}+ |\nu_1 - \nu_2|^2 \int_0^t \frac{1}{\nu_1} \|u_2(s)\|^2_{L^2} \, ds\right),
    \]
    but we know that $\|u_i\|_{L^2(Q)} \leq C(M)$ for $i = 1, 2$ by Proposition \ref{prop:NSReg}, and so
    \[
        \sup_{t\in [0, 1]} \|w(t)\|^2_{L^2} \leq C(M, d, \nu_{\min}) \left(\|w(0)\|_{L^2}^2 + |\nu_1 - \nu_2|^2\right).
    \]
    Taking square roots yields the second result.
\end{proof}

\end{appendix}
\begin{acks}[Acknowledgments]
    The author would like to thank R. Nickl for suggesting the topic and many helpful conversations, A. Castre for helpful discussions, and D. Konen for providing an easier argument for Lemma \ref{lem:LaplaceNonZero} in correspondence.
\end{acks}

\begin{funding}
The author was supported by an ERC Advanced Grant (UKRI G116786).
\end{funding}

\end{document}